\documentclass[11pt,a4paper]{amsart}
\usepackage[T1]{fontenc} 
\usepackage[english]{babel} %hyphenation
\usepackage[utf8]{inputenc}
\usepackage{lmodern} 
\usepackage{amssymb}
\usepackage{float}
\usepackage{amscd}
\usepackage{amsthm}

\usepackage[top=2cm, bottom=2cm, left=2cm, right=2cm,twoside=false]{geometry}
\usepackage{array}
\usepackage{longtable}
\usepackage{graphicx}
\usepackage{url} 
\usepackage{wrapfig}
\usepackage{color}
\usepackage[usenames,x11names]{xcolor}
\usepackage[all]{xy}
\usepackage{enumerate}
\usepackage{multirow}
\usepackage{units} %\nicefrac
\usepackage{yfonts}
\usepackage{mathrsfs}
\usepackage{caption}
\usepackage{subcaption}
\usepackage{multicol}
\usepackage[normalem]{ulem}

\usepackage[export]{adjustbox}
\usepackage{lscape}

\usepackage[hypertexnames=false]{hyperref} 
\hypersetup{linkcolor  =DodgerBlue3, citecolor  = teal, urlcolor   = teal, colorlinks = true, hyperfootnotes =false}

\renewcommand{\to}{\longrightarrow}

\def\Spec{\operatorname{Spec}}

\def\Sing{\operatorname{Sing}}

\def\Pic{\operatorname{Pic}}

\newcommand{\F}{\mathbb{F}}

\newcommand{\C}{\mathbb{C}}

\newcommand{\Q}{\mathbb{Q}}

\newcommand{\cL}{{\mathcal{L}}}
\renewcommand{\P}{\mathbb{P}}

\renewcommand{\tilde}{\widetilde}

\newcommand{\w}{\omega}
\newcommand{\wg}{\omega_{\geq 5}}

\renewcommand{\epsilon}{\varepsilon}
\renewcommand{\phi}{\varphi}

\newcommand{\ind}{\mathrm{ind}}

\newcommand{\tip}{^{\mathrm{tip}}}

\def\gcd{\operatorname{gcd}}
\def\NS{\operatorname{NS}}
\def\Exc{\operatorname{Exc}}
\def\Supp{\operatorname{Supp}}
\def\Bk{\operatorname{Bk}}

\renewcommand{\leq}{\leqslant}
\renewcommand{\geq}{\geqslant}

\newcommand{\Qa}{\mathcal{Q}_{4}}

\theoremstyle{plain}
\newtheorem{tw}{Theorem}[section]
 \makeatletter \@addtoreset{step}{tw}\makeatother

\theoremstyle{definition}
\newtheorem{dfn}[tw]{Definition}
\newtheorem{lem}[tw]{Lemma}
\newtheorem{prop}[tw]{Proposition}
\newtheorem{cor}[tw]{Corollary}
\newtheorem{ex}[tw]{Example}

\newtheorem{conjecture}[tw]{Conjecture}
\newtheorem{notation}[tw]{Notation}
\newtheorem{rem}[tw]{Remark} 

\theoremstyle{remark}
 
\newtheorem{claim}{Claim}
\newtheorem*{claim*}{Claim}

\makeatletter \def\subsection{\@startsection{subsection}{3}%
	\z@{.5\linespacing\@plus.7\linespacing}{.5\linespacing}%
	{\bfseries\itshape}} \makeatother  %could use \normalfont\bfseries too

\makeatletter \def\@tocline#1#2#3#4#5#6#7{\relax \ifnum #1>\c@tocdepth \else \par \addpenalty\@secpenalty\addvspace{#2} \begingroup \hyphenpenalty\@M \@ifempty{#4}{\@tempdima\csname r@tocindent\number#1\endcsname\relax}{\@tempdima#4\relax} \parindent\z@ \leftskip#3\relax \advance\leftskip\@tempdima\relax \rightskip\@pnumwidth plus4em \parfillskip-\@pnumwidth #5\leavevmode\hskip-\@tempdima \ifcase #1 \or\or \hskip 1em \or \hskip 2em \else \hskip 3em \fi #6\nobreak\relax \dotfill\hbox to\@pnumwidth{\@tocpagenum{#7}}\par \nobreak \endgroup  \fi}
\makeatother

\makeatletter \renewenvironment{proof}[1][\proofname]{
	\par\pushQED{\qed}\normalfont
	\topsep6\p@\@plus6\p@\relax
	\trivlist\item[\hskip\labelsep\bfseries#1\@addpunct{.}]
	\ignorespaces}{
	\popQED\endtrivlist\@endpefalse} \makeatother

\def\:{\colon}

\numberwithin{equation}{section}
\def\8{\infty}

\begin{document}

\large
\title[Complex planar curves homeomorphic to a line]{Complex planar curves homeomorphic to a line\\ have at most four singular points} 
%\\ \today}
\author[Mariusz Koras]{Mariusz Koras$^\dagger$}
\address{Mariusz Koras: Institute of Mathematics, University of Warsaw, ul. Banacha 2, 02-097 Warsaw}\email{koras@mimuw.edu.pl}
\thanks{$\dagger$ The first author died when preparing the final version of the manuscript}
\author[Karol Palka]{Karol Palka}
\address{Karol Palka: Institute of Mathematics, Polish Academy of Sciences, ul. \'{S}niadeckich 8, 00-956 Warsaw, Poland}
\email{palka@impan.pl}

\thanks{The first author was supported by the National Science Centre, Poland, Grant No. 2013/11/B/ST1/02977. The second author was partially supported by the Foundation for Polish Science under the Homing Plus programme, cofinanced from the European Union, RDF and by the National Science Centre, Poland, grant No. 2015/18/E/ST1/00562}



\begin{abstract}We show that a complex planar curve homeomorphic to the projective line has at most four singular points. If it has exactly four then it has degree five and is unique up to a projective equivalence.
\end{abstract}

\maketitle

\section{Main result}

All varieties considered are complex algebraic. 

\smallskip Classical theorems of Abhyankar--Moh \cite{AbMo-approximate_roots} and Suzuki \cite{Suzuki_AMSthm} and Zaidenberg-Lin \cite{LinZaid-LZ_theorem} give a complete description of affine planar curves homeomorphic to the affine line. Every such curve is, up to a choice of coordinates, given by one of the equations $x=0$ or $x^n=y^m$ for some relatively prime integers $n>m\geq 2$; see \cite{GurjarMiyanishi_AMS_and_LZ_thms} and \cite{Palka-AMS_LZ} for proofs based on the theory of open surfaces. In particular, it has at most one singular point. 

An analogous projective problem, the classification of planar curves homeomorphic to the projective line, is much more subtle and remains open. Such curves are necessarily rational and their singularities are cusps (are locally analytically irreducible), hence they are called \emph{rational cuspidal curves}. Apart from being an object of study on its own, they have strong connections with surface singularities via \emph{superisolated singularities} \cite{Luengo-superisolated_sing}. Invariants of those links of superisolated singularities which are rational homology $3$-spheres, including the Seiberg--Witten invariants, are tightly related with invariants of rational cuspidal curves, see \cite{FLMN_cuspidal_curves}, \cite{BaLuMH-superisolated_sing}. This connection has led to counterexamples to several conjectures concerning normal surface singularities \cite{LMN-links_superisolated_sing} and to many proved or conjectural constrains for invariants of rational cuspidal curves \cite{FLMN_cusps_and_open_surfaces}. Deep relations of such curves with low-dimensional topology have been studied from various standpoints, including Heegaard--Floer homology \cite{BoroLivi-HeegaardFloer_and_cusps}, \cite{BarBoSe_HF_and_RCC}, lattice cohomology \cite{BoNe-lattices} or involutive Floer homology \cite{BoHoSch-involutive_Fl_hom-_RCC}. Applications of the semigroup distribution property proved in \cite{BoLi-semigroups_RCC} have led to some partial classification results, see \cite{Liu-thesis}, \cite{Bodnar}. Conjectures concerning free divisors and local cohomology of Milnor algebras have been investigated in \cite{DiSt-free-divisors}, \cite{DiSt-free-divisors2_RCC} and \cite{DiSt3-freenes_RCC}. Despite some partial success of the above approaches, most of important questions remained unanswered.

\smallskip 
Recently, using techniques based on the minimal model program, we proved the Coolidge-Nagata conjecture \cite{Palka-Coolidge_Nagata1}, \cite{KoPa-CooligeNagata2}, which asserts that every rational cuspidal curve can be mapped onto a line by a Cremona transformation of the plane. The proof is necessarily non-constructive and the transformation needed for such a birational rectification may be complicated. An important remaining problem was to bound the number of possible singular points. We prove the following optimal result and we characterize the extremal case.

\begin{tw}\label{thm:at_most_4_cusps} A complex planar curve homeomorphic to the projective line has at most four singular points. If it has exactly four then, up to a choice of coordinates on the plane, it can be parameterized as (cf.\ Example \ref{ex:Q4}) 
\begin{equation}
 [s:t]\to [s^4t:s^2t^3-s^5:t^5+2s^3t^2].
\end{equation}
\end{tw}
\smallskip 

The problem of bounding the number of cusps of rational cuspidal curves was posed by Sakai \cite[p.\ 14]{Zaid-open_MONTREAL_problems} and specific conjectures were made, among others, by Orevkov and Piontkowski \cite{Piontkowski-number_of_cusps}. Orevkov and Zaidenberg proved that rigid rational cuspidal curves have at most $9$ cusps \cite{OrZa-rigid_cusp}. The absolute bound by $8$ was obtained by \cite{Tono-number_of_cusps} using a logarithmic Noether inequality. Later the second author improved it to $6$ \cite[Theorem 1.4]{Palka-minimal_models}.  Results of this type do not seem to be available using the more topologically-oriented methods discussed above, at least at their current stage of development. %For instance, the obstruction of Borodzik and Livingston \cite{BoLi-semigroups_RCC} depends only on the union of the multiplicity sequences of singularities of rational cuspidal curve, hence cannot be used to bound the number of singular points. A more sophisticated approach, using involutive Floer theory \cite{BoHoSch-involutive_Fl_hom-_RCC} works only for rational cuspidal curves of odd degree.

\smallskip  Numerous families of rational cuspidal curves with a small number of cusps have been constructed by various authors and some partial classification results were obtained. For uni- and bi-cuspidal curves see \cite{Kashiwara}, \cite{Yos-cusp}, \cite{Fenske_1and2-cuspidal_curves}, \cite{Tono-equations_cusp_curves}, \cite{Tono_1cusp_with_kod_1}, \cite{OrevkovCurves}, \cite{FLMN_one_pair},  \cite{Tono_nie_bicuspidal}, \cite{Bodnar}. The richest source of examples are closures of $\C^*$-embeddings into $\C^2$ \cite{BoZo-annuli}, \cite{CKR-Cstar_good_asymptote}, but even those are not understood completely. Known examples of rational cuspidal curves with three cusps are: Namba's tricuspidal quintic \cite[2.3.10.8]{Namba_geometry_of_curves}, two discrete series discovered by Flenner-Zaidenberg \cite{FLZa-_class_of_cusp}, \cite{FlZa_cusps_d-3} and one discovered by Fenske \cite{Fenske_cusp_d-4}. Inequalities bounding the degree of the curve in terms of the maximal multiplicity of its cusps were proved in \cite{MaSa-degree_of_cusp_curves} and \cite{OrevkovCurves}.

 % Zariski pairs \cite{Bartolo-pi_1(RCC)}

\medskip
%In general, it is an interesting problem whether for a given complex algebraic surface and a set of curves on it with fixed topology, the number of possible singular points of such curves is bounded and, if so, how to prove reasonably low effective bounds. For affine curves it is conjectured \cite{Lin-Zaid_conjecture} that the number of singular points of a planar curve $C$ can be at most $2b_1(C)+1$, where $b_1$ is the first Betti number of $C$. Apart from topological affine lines ($b_1(C)=0$) the inequality is known only for topological punctured affine lines ($b_1(C)=1$) due to a result of Borodzik and Żoł\c{a}dek \cite{BoZo-annuli_max_3cusps}. A weaker general inequality was proved by Orevkov \cite{Orevkov_onTonos_thm_cusp_curves}. 

Finally, let us comment on a recent progress on the classification problem for rational cuspidal curves. A well understood case is when the surface $\P^2\setminus \bar E$ is not of log general type, i.e.\ when $\kappa(K_X+D)\leq 1$, where $(X,D)\to (\P^2,\bar E)$ is the minimal log resolution of singularities. Here the global structure of $\P^2\setminus \bar E$ is known \cite[Proposition 2.5]{Palka-Coolidge_Nagata1} and in fact curves of this type have been classified, see \cite{FLMN_cuspidal_curves} for a review. In particular, they have at most two cusps by \cite{Wakabayashi-cusp}. The remaining case $\kappa(K_X+D)=2$ is difficult. So far, the most successful approach is a careful analysis of possible runs of the minimal model program for the pair $(X,\frac{1}{2}D)$ in terms of \emph{almost minimal models}, as defined in \cite{Palka-minimal_models}. We extend this method in the current article. We mention the following key conjecture, which remains open. Note that for $(X,D)$ as above the affine surface $X\setminus D=\P^2\setminus \bar E$ is $\Q$-acyclic. 
\begin{conjecture}[Negativity Conjecture, \cite{Palka-minimal_models}] \label{con:negativity}
If $(X,D)$ is a smooth completion of a $\Q$-acyclic surface then 
\begin{equation}
\kappa(K_X+\tfrac{1}{2}D)=-\8.
\end{equation}
\end{conjecture}

\noindent As discussed in Conjecture 2.5 loc.\ cit., the conjecture generalizes both the weak rigidity conjecture by Flenner-Zaidenberg \cite[p.\ 16]{Zaid-open_MONTREAL_problems} and the Coolidge-Nagata conjecture. T. Pełka and the second author have already shown that the approach with almost minimal models is very effective: they classified up to a projective equivalence all rational cuspidal curves with a complement of log general type under the assumption that for such complements the Negativity Conjecture holds \cite{PaPe_cuspidal-Cstst-fibrations}, \cite{PaPe_cuspidal-delPezzo}. In particular, in that case the topology of singularities completely determines the class of a projective equivalence. The classification implies also that there are no tricuspidal rational curves satisfying the Negativity Conjecture other than the ones discovered already. Summarizing, given the above theorem, to have a complete understanding of rational cuspidal curves it remains now to prove the above conjecture for complements of those with at most $3$ cusps.

\smallskip 
As suggested by Conjecture \ref{con:negativity}, properties of the divisor $K_X+\frac{1}{2}D$ and the associated almost minimal model play a major role in the proof of Theorem \ref{thm:at_most_4_cusps}. The key step, which takes most of our effort, is Theorem \ref{thm:almost_minimal}, saying that for a rational cuspidal curve $\bar E\subseteq \P^2$ with at least four cusps the surface $(X_0,\frac{1}{2}D_0)$, where $(X_0,D_0)\to (\P^2,\bar E)$ is a minimal weak resolution, is almost minimal. In contrast, in case of at most three cusps the process of almost minimalization may contract zero, two, three or four curves not contained in the boundary, see figures for various types of curves drawn in \cite[Section 3]{PaPe_cuspidal-Cstst-fibrations} and Lemmas 3.8(e), 3.12(f) in \cite{PaPe_cuspidal-delPezzo}.

\smallskip 

We thank Maciej Borodzik for helpful remarks and we thank Tomasz Pełka for a careful reading of a preliminary version of the manuscript.
\clearpage
\tableofcontents

\section{Preliminaries}\label{sec:prelim}

In this article curves are irreducible and reduced. For a reduced divisor $T$ on a smooth projective surface we denote the number of its irreducible components by $\#T$ and we put \begin{equation}\label{eq:d(T)}
d(T)=\det(-Q(T)), 
\end{equation} where $Q(T)$ is the intersection matrix of $T$; we put $d(0)=1$. We define the \emph{arithmetic genus of $T$} as $p_a(T)=\frac{1}{2}T\cdot(K+T)+1$, where $K$ is the canonical divisor and we define the \emph{branching number} of a component $C\subseteq T$ by \begin{equation}\label{eq:beta}
\beta_T(C):=C\cdot (T-C).
\end{equation} 
We say that $C$ is a \emph{branching} component of $T$ if $\beta_T(C)\geq 3$ and that it is a \emph{tip} of $T$ if $C\neq 0$ and $\beta_T(C)\leq 1$. If $C\subseteq T$ is a $(-1)$-curve such that $0<\beta_T(C)\leq 2$ and $C$ meets other components of $D$ normally (transversally) and at most once each then we say that $C$ is \emph{superfluous} in $D$. If $T$ is a simple normal crossing (snc) divisor we say that it is \emph{snc-minimal} if it contains no superfluous $(-1)$-curves. A \emph{rational tree} is a reduced connected divisor of arithmetic genus zero. It is an snc-divisor with a simply connected support. It is a (rational) \emph{chain} if it has no branching component. A rational chain is \emph{admissible} if it has a negative definite intersection matrix and contains no $(-1)$-curve.

Assume $T$ is an admissible chain. It can be written as  $T=T_1+\cdots+ T_k$, $k\geq 0$ where $T_i\cong \P^1$, $i=1,\ldots,k$ are its irreducible components, $T_i\cdot T_{i+1}=1$ for $i\leq k-1$ and $T_i\cdot T_j=0$ if $|i-j|>1$. Once the order of components as above is fixed, we write $T=[-T_1^2,\ldots,-T_k^2]$. Since $T$ is admissible, we have $T_i^2\leq -2$ for $i\in\{1,\ldots,k\}$. We put $\delta(T)=\frac{1}{d(T)}$. In particular, $\delta([(2)_k])=\frac{1}{k+1}$, where $(2)_k$ denotes the sequence $(2,\ldots,2)$ of length $k$. Following \cite{Fujita-noncomplete_surfaces} we define the \emph{bark} of $T$ as 
\begin{equation}\label{eq:Bk}
\Bk' T=\sum_{i=1}^k \frac{d(T_{i+1}+\cdots+T_k)}{d(T)}\ T_i .
\end{equation}
and the \emph{inductance} of $T$ as 
\begin{equation}\label{eq:inductance}
\ind(T)=\frac{d(T-T_1)}{d(T)}.
\end{equation}
We check that $T_i\cdot \Bk' T$ equals $-1$ if $i=1$ and equals $0$ otherwise and that $(\Bk' T)^2=-\ind(T)$. 

For a general reduced divisor $T$, an ordered rational chain $T_1+\ldots+T_k\subseteq T$ as above is called a rational \emph{twig} of $T$ if $T_1$ is a tip of $T$ and $\beta_T (T_i)=2$ for $i>1$. A rational tree $T$ is a (rational) \emph{fork} if it has exactly one branching component and three maximal twigs (maximal in the sense of inclusion of supports). Assuming $T$ is snc-minimal, connected and does not contract to a quotient singular point (that is, it is neither $[1]$, nor a rational admissible chain nor a rational fork with admissible twigs and a negative definite intersection matrix), we define $\delta(T)$, $\Bk'(T)$ and $\ind(T)$ as the sum of respective quantities computed for all maximal admissible twigs of $T$. For a general definition covering the case of resolutions of quotient singularities see \cite[\S 2]{Miyan-OpenSurf}.

If two trees $T_1$, $T_2$ meet normally at a unique point then, denoting the components which have a common point by $C_j\subseteq T_j$, $j=1,2$, from elementary properties of determinants we infer the following formula:
\begin{equation}\label{eq:d(T1+T2)} d(T_1+T_2)=d(T_1)d(T_2)-d(T_1-C_1)d(T_2-C_2).
\end{equation}

\begin{lem}\label{lem:ind} Let $R=[a_1,\ldots,a_k]$ be an admissible chain and let $R'=[a_1,\ldots,a_{k-1}]$. Then $\ind(R)> \ind(R')$.
\end{lem}

\begin{proof} 
Put $d'([a_1,a_2,\ldots,a_k]):=d([a_2,\ldots,a_k])$ and $d''([a_1,a_2,\ldots,a_k]):=d'([a_2,\ldots,a_k])$, where by definition $d'(0)=0$. By \eqref{eq:d(T1+T2)}
\begin{equation}\label{eq:d_recurrence}
d(R)=a_1d'(R)-d''(R),
\end{equation}
hence $\ind([a_1,\ldots,a_k])^{-1}=a_1-\ind([a_2,\ldots,a_k])$. Using this formula we proceed by induction on $\#R$.
\end{proof}  

\begin{lem}[\cite{Fujita-noncomplete_surfaces}, 4.16] \label{lem:Sigma} Fix a $\P^1$-fibration of a smooth projective surface $X$ and a reduced divisor $D$ on $X$. For every fiber $F$ denote by $\sigma(F)$ the number of components of $F$ not contained in $D$ and put $\displaystyle \Sigma=\sum_{F\nsubseteq D}(\sigma(F)-1)$. Let $h$ and $\nu$ denote, respectively, the number of horizontal components of $D$ and the number of fibers contained in $D$. Then 
\begin{equation}
\Sigma=h+\nu-2+\rho(X)-\#D,
\end{equation}
where $\rho(X)$ is the Picard rank of $X$.
\end{lem}

\begin{lem}[Hurwitz formula]\label{lem:Hurwitz}
Let $F$ be a scheme-theoretic fiber of some fibration of a smooth projective surface and let $E$ be a rational cuspidal curve not contained in fibers. Denote by $r_F(q)$ the ramification index at $q\in E$ of the restriction of the fibration to $E$. Then for every $p\in E$ we have
\begin{equation}
\#\{q\in E\setminus \{p\}:r_F(q)\geq 2\}\leq \sum_{q\in E\setminus\{p\}}(r_F(q)-1)=2F\cdot E-1-r_F(p).
\end{equation}
In particular, $\#\Sing E\leq 2F\cdot E-r_F(p)$.
\end{lem}

\begin{lem}(Producing elliptic fibrations).\label{lem:ell_fib}
Let $E$ be a smooth rational curve on a smooth birationally ruled projective surface $X$ and let $C$ be a $(-1)$-curve such that $E\cdot C=2$. Then $X$ is rational and the following hold. 
\begin{enumerate}[(a)]
\item If $h^0(2K+E)\neq 0$ and $E^2=-4$ then $|E+2C|$ induces an elliptic fibration of $X$.
\item If $E^2=-3$ then after blowing up once some point on $E\setminus C$ the linear system of the proper transform of $E+2C$ induces an elliptic fibration.
\end{enumerate}
\end{lem}

\begin{proof}Clearly, $X$ is not $\P^2$. Let $X\to B$ be a $\P^1$-fibration. Since all fibers are rational trees, $E$ or $C$ is horizontal, so $B$ is rational. It follows that $X$ is rational. To find the necessary elliptic fibrations for (a) and (b) we may replace $X$ with its image after the contraction of $C$. Now $E$ becomes a nodal rational curve, so $E\cdot (K+E)=0$.

(a) We have $K\cdot E=E^2=0$, so $(2K+E)^2=4K^2$. Suppose that $K^2<0$. Then the divisor $2K+E$, which is effective by assumption, is not nef. Since $E^2=0$, we have $X\ncong \P^2$, so there exists a curve $\ell$ such that $(2K+E)\cdot \ell<0$ and $\ell^2\leq 0$. We have $\ell\neq E$, so $K\cdot \ell<0$, hence $\ell\cong \P^1$. The curve $\ell$ is not a $0$-curve, because otherwise it is a fiber of a $\P^1$-fibration of $X$, for which $(2K+E)\cdot \ell\geq 0$, as $h^0(2K+E)\neq 0$. Thus $\ell^2<0$, hence $\ell$ is a $(-1)$-curve with $\ell \cdot E\leq 1$. But if $\ell\cdot E=1$ then after the contraction of $\ell$ the image of $E$ has arithmetic genus $1$ and intersects the canonical divisor negatively, hence is in the fixed part of the direct image of $2K+E$, which is impossible, as $h^0(2K)=0$. Thus $\ell\cdot E=0$. Contracting $\ell$ we reduce the proof inductively to the case $K^2\geq 0$. 
 
Now we argue as in the proof of \cite[Theorem 3.3, Claim]{KumarMurthy-rectifiability}. By Riemann-Roch we have $h^0(-K)\geq K^2+1\geq 1$ and $h^0(K+E)\geq p_a(E)=1$, where $p_a$ denotes the arithmetic genus. Write $$E=(K+E)+(-K).$$ Clearly, $E$ is not in the fixed part of $|K+E|$. It is also not in the fixed part of $|-K|$, because otherwise $h^0(-E)=h^0((2K+E)+2(-K-E))\geq 1$, which is impossible. Thus $E$ is linearly equivalent to an effective divisor whose support does not contain $E$, which gives $h^0(E)\geq 2$ and shows that the linear system $|E|$ has no fixed components. But $E^2=0$, so it has no base points either. Since $E$ is a connected reduced member of the system, a general member is smooth and reduced of genus $p_a(E)=1$, hence it is an elliptic curve.

(b) After the contraction of $C$ we have $E^2=1$. By Riemann-Roch $h^0(E)\geq E^2+1=2$. Let $U\in |E|$ be a smooth member. We have $U\cdot E=E^2=1$, so $U$ meets $E$ normally in a unique smooth point $p$. Let $E'$ be the proper transform of $E$ under the blow-up of $p$. Then $h^0(E')\geq 2$ and $|E'|$ has no fixed components. Since $(E')^2=0$, the linear system $|E'|$ induces an elliptic fibration of the new surface, as required.
\end{proof}

The following example shows that the assumption $h^0(2K+E)\neq 0$ in Lemma \ref{lem:ell_fib}(a) is necessary.
 
\begin{ex}\label{rem:Cayley-Bacharach} Let $\bar E$ be an irreducible planar cubic with a singular point $q$ and let $p_1,\ldots, p_9\in \bar E\setminus \{q\}$ be distinct points. Let $\theta\:X\to \P^2$ denote the blowup at $p_1,\ldots,p_9$ and $E$ the proper transform of $\bar E$. We have $E^2=0$ and $K+E\sim 0$, hence $h^0(2K+E)=h^0(K)=0$. Assume that $E$ coincides with the support of some fiber of an elliptic fibration of $X$. Then $F\sim nE$ for some positive integer $n$ and a general fiber $F$. The curve $\theta(F)$ meets $\bar E$ exactly in $p_1,\ldots,p_9$, at each point with multiplicity $n$. We have $\theta(F)\sim 3n \ell$, where $\ell$ is a line, so intersecting with $\bar E$ we obtain $n(p_1+\cdots+p_9)\sim 9n o$, where $o\in \bar E$ is any flex point of $\bar E \setminus \{q\}$. Thus $n(p_1+\ldots+p_9)=0$ in the group law of $(\bar E\setminus \{q\},o)$. The group is isomorphic to $(\C^*,1)$ if $q\in \bar E$ is a node and to $(\C^1,0)$ if $q\in \bar E$ is a cusp, so the latter equality is impossible for a general choice of points $p_1,\ldots, p_9$. This shows that in general $\Supp E$ is not the support of an elliptic fiber. To obtain a configuration as in Lemma \ref{lem:ell_fib}(a) we blow up once at the singular point of $E$.
\end{ex}

\medskip

Let $\bar E\subset \P^2$ be a rational cuspidal curve. Denote by $\pi_0\:(X_0,D_0)\to (\P^2,\bar E)$ the minimal weak resolution of singularities, that is, a composition of a minimal sequence of blow-ups such that the proper transform $E_0\subseteq X_0$ of $\bar E$ is smooth. Clearly, $\P^2\setminus \bar E=X_0\setminus D_0$. We denote the cusps of $\bar E$ by $q_1,\ldots, q_c$. Since $(\Pic \P^2)\otimes \Q$ is generated by $\bar E$, the components of $D_0$ freely generate $(\Pic X_0)\otimes \Q$. The divisor $D_0-E_0$ has a negative definite intersection matrix.

\smallskip
Let $Q_j\subseteq D_0$ be the reduced exceptional divisor over the cusp $q_j$, $j\in \{1,\ldots,c\}$. It is a rational tree with a negative definite intersection matrix and a specific dual graph whose all vertices have degree not bigger than $3$. The Eisenbud--Neumann diagram of  this graph (defined as the image of the graph after the contraction of vertices of degree $2$) is
\begin{figure}[h]

\centerline{
\xymatrix{ {\circ}\ar@{-}[r] &{\circ}\ar@{-}[r]\ar@{-}[d] &{\circ}\ar@{-}[r]\ar@{-}[d] &{\ldots}\ar@{-}[r] &{\circ}\ar@{-}[r]\ar@{-}[d] &{\circ} \\ {} &{\circ} &{\circ} &{\ldots} &{\circ} & {} }
}
\caption{The Eisenbud-Neumann diagram of $Q_j=\pi_0^{-1}(q_j)$.}\label{Fig:EN_Qj}

\end{figure}
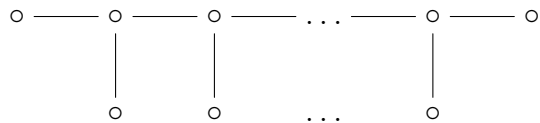

Since every $q_j\in \bar E$ is locally analytically irreducible, $Q_j$ can be seen as being produced by a \emph{connected sequence of blow-ups}, i.e.\ we can decompose the morphism contracting it to a point into a sequence of blow-ups $\sigma_1\circ\ldots\circ\sigma_s$, such that the center of $\sigma_{i+1}$ belongs to the exceptional curve of $\sigma_i$ for $i\geq 1$. Thus the components of $Q_j$ are linearly ordered as proper transforms of subsequent exceptional divisors of $\sigma_1, \sigma_2, \ldots, \sigma_s$. The last one, call it $C_j$, is the unique $(-1)$-curve in $Q_j$. It is contained in some twig of $Q_j$. Although $E_0$ is smooth, $\pi_0$ is not a log resolution, so there may be (at most one) component of $Q_j-C_j$ meeting $E_0$, call it $\tilde C_j$; put $\tilde C_j=0$ if there is no such. We have $\tilde C_j\cdot E_0=1$ if $\tilde C_j\neq 0$.

\begin{dfn} A cusp of a planar curve is \emph{semi-ordinary} if it is locally analytically isomorphic to the singular point of $x^2=y^{2m+1}$ at $(0,0)\in\Spec \C[x,y]$ for some $m\geq 1$.
\end{dfn}

Note that an \emph{ordinary cusp} (called also \emph{simple}) is a semi-ordinary cusp with $m=1$. The exceptional divisors of the minimal log resolution and of the minimal weak resolution of a semi-ordinary cusp are $[2,1,3,(2)_{m-1}]$ and $Q_j=[1,(2)_{m-1}]$, respectively. We have $C_j\cdot E_0=2$ and $\tilde C_j=0$ for semi-ordinary cusps.

\medskip
\section{Almost minimal models}

Let $\bar E\subseteq \P^2$ be a curve homeomorphic to the projective line, that is, a rational cuspidal curve. We will study the minimal log resolution $\pi\: (X,D)\to (\P^2,\bar E)$ using minimal model program techniques. For basic notions and theorems of the program we refer the reader to \cite{KollarKovacs-2DlogMMP} and \cite{Matsuki_MMP_intro,KollarMori-bir_geom}. In this article any final output of the birational part of a run of the MMP (which is a birational morphism) will be called a \emph{minimal model}, even if $\kappa=-\8$. Recall that the program applied to $(X,rD)$, where $r\in\Q\cap [0,1]$, produces a sequence of contractions of \emph{log exceptional curves}, that is, curves which on the respective image of $X$ have negative self-intersection numbers and intersect the direct image of $K_{X}+rD$ negatively. It is well-known that for $r=1$ minimal models of log smooth surfaces are singular in general and that arbitrary quotient singularities may appear. However, surfaces with quotient singularities are not understood well enough (even the log del Pezzo surfaces of rank one), which causes problems with applications of the MMP techniques to concrete problems concerning log surfaces or quasi-projective surfaces. The main point of the construction of an \emph{almost minimal model} is to avoid deeper log singularities introduced by the usual run of the log MMP, or rather to delay their introduction until all necessary contractions in the open part of the surface are done.  

\smallskip 
To prove Theorem \ref{thm:at_most_4_cusps} we may, and will, assume that the number of cusps of  $\bar E$ is at least $4$. By \cite{Wakabayashi-cusp} the surface $\P^2\setminus \bar E$ is of log general type. It is also $\Q$-acyclic, hence as a consequence of the logarithmic Bogomolov-Miyaoka-Yau inequality \cite{MiTs-lines_on_qhp} (see also \cite[\S 10]{Palka-recent_progress_Qhp}), it does not contain affine lines. It follows that $(X,rD)$ for $r=1$ is \emph{almost minimal} in the sense of \cite[\S 2.3.11]{Miyan-OpenSurf}; equivalently, the morphism onto a minimal model contracts only components of $D$. Unfortunately, the latter fact turned out to be of limited use when studying the surface $\P^2\setminus \bar E$ because of the complicated geometry of possible divisors $D$. We will therefore use a generalization of the theory of almost minimal models to the case of fractional boundaries proposed in \cite{Palka-minimal_models}. We work with $r=\frac{1}{2}$. This choice turns out to be optimal for many reasons. Now new curves in $X\setminus D$ will be contracted and we need to carefully control the whole process.

Instead of working with the minimal log resolution it is in fact more convenient to work with the minimal weak resolution $\pi_0\:(X_0,D_0)\to (\P^2,\bar E)$, as defined in the previous section. As an outcome of the construction of an almost minimal model of $(X_0,\frac{1}{2}D_0)$ given in \cite[\S 3]{Palka-minimal_models} we obtain a sequence of birational contractions between smooth projective surfaces $$\psi_{i+1}\:(X_{i},D_{i})\to (X_{i+1},D_{i+1}),\ \  0\leq i\leq n-1,$$ where, $D_{i+1}$ is a reduced divisor, a direct image of $D_{i}$ and $X_{i+1}\setminus D_{i+1}$ is an open affine subset of $X_{i}\setminus D_{i}$. Let us recall this inductive construction. Assume $(X_i,D_i)$ is defined already. Write $K_i$ for the canonical divisor on $X_i$. By a \emph{$(-2)$-twig} of $D_i$ we mean a twig of $D_i$ consisting of $(-2)$-curves. A \emph{maximal $(-2)$-twig} is a $(-2)$-twig which is maximal with respect to the inclusion of supports. Such twigs intersect trivially with the canonical divisor and hence their role in the construction is special. There are also some specific $(-1)$-curves in $D_i$ we need to take care of.

\begin{notation}\label{def:Delta,Upsilon,Dflat} Let $(X_i,D_i)$ be as above.  \begin{enumerate}[(a)]

\item Let $\Delta_i$ be the sum of all maximal $(-2)$-twigs of $D_i$. 

\item Let $\Upsilon_i$ be the sum of $(-1)$-curves $L$ in $D_i$, for which either $\beta_{D_i}(L)=3$ and $L\cdot \Delta_i=1$ or $\beta_{D_i}(L)=2$ and $L$ meets exactly one component of $D_i$.

\item Decompose $\Delta_i$ as $\Delta_i=\Delta^+_i+\Delta^-_i$, where $\Delta^+_i$ consists of those maximal $(-2)$-twigs of $D_i$ which meet $\Upsilon_i$.

\item Put $D_i^\flat=D_i-\Upsilon_i-\Delta^+_i-\Bk'\Delta^-_i$.
\end{enumerate}\end{notation}
 
\begin{dfn} Let $(X_i,D_i)$ be as above. A curve $A\subseteq X_i$ is \emph{almost log exceptional on $(X_i,\frac{1}{2}D_i)$} if $A$ is a $(-1)$-curve not contained in $D_i$ such that 
\begin{equation}\label{eq:ALE} A\cdot D_i=2,\ \  A\cdot \Delta_i^-=1,\ \  A\cdot(\Upsilon_i+\Delta_i^+)=0
\end{equation} and the component of $\Delta^-_i$ meeting $A$ is a tip of $\Delta^-_i$ (but not necessarily a tip of $D_i$).  
\end{dfn}

Note that $A\cap (X_i\setminus D_i)\cong \C^*$, hence $e_{top}(X_i\setminus (D_i+A))=e_{top}(X_i\setminus D_i)$. If there is no almost log exceptional curve on $(X_i,\frac{1}{2}D_i)$ then we put $n=i$ and we call $n$ the \emph{length} of the chosen process of almost minimalization and $(X_n,\frac{1}{2}D_n)$ an \emph{almost minimal model} of $(X_0,\frac{1}{2}D_0)$. Otherwise we choose $A_{i+1}=A$ as above and we define a birational morphism $$\psi_{i+1}\:X_i\to X_{i+1}$$ as a local snc-minimalization of $D_i+A_{i+1}$, that is, a composition of a maximal  (not necessarily unique) sequence of contractions of superfluous $(-1)$-curves contained in $D_i+A_{i+1}$ and its images starting from $A_{i+1}$ such that the total contracted divisor is connected, see Figure \ref{Fig:alm_log_exc}. By \cite[Lemma 3.4]{Palka-minimal_models} each component of $\Upsilon_i$ meets at most one component of $\Delta_i$. In particular, $\Upsilon_i+\Delta_i$ can be contracted by a birational morphism. The following lemma is a motivation for the above definition of an almost log exceptional curve. It shows that the above definition of an almost minimal model is analogous to the original definition in the theory of open surfaces known in case of a reduced simple normal crossing boundary divisor (see \cite[2.3.11, 2.4.3]{Miyan-OpenSurf}). 

 \begin{figure}[h]\centering
\includegraphics[scale=0.21]{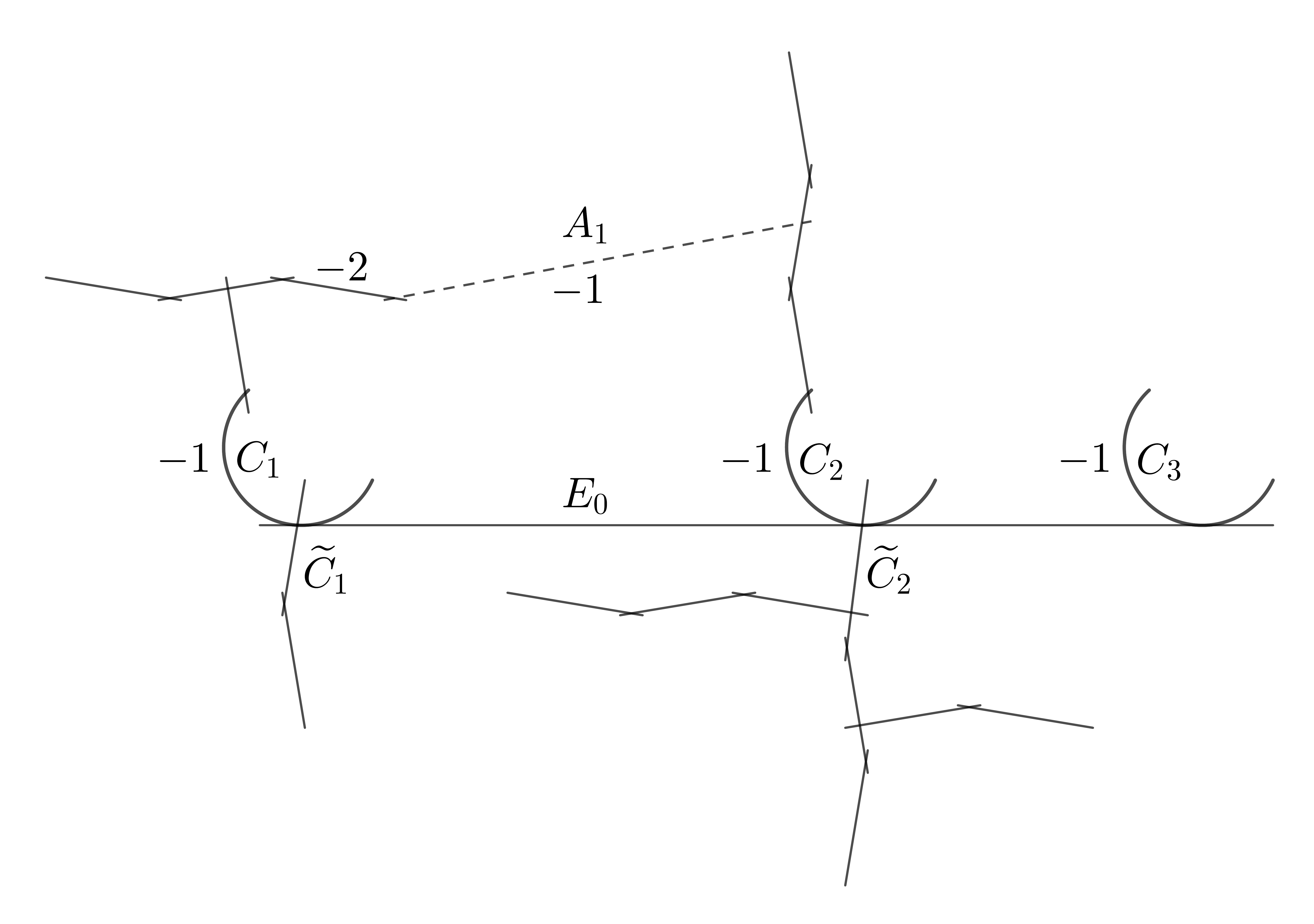}\caption{An example of $(X_0,D_0+A_1)$ for $c=3$. Here $\Exc \psi_1=[1,2]$.}\label{Fig:alm_log_exc}
 \end{figure}

\begin{lem}\label{lem:producing_A} Let $(X_i,D_i)$ be as above and let $\alpha_i\: (X_i,D_i)\to (Y_i,D_{Y_i})$ be the contraction of $\Upsilon_i+\Delta_i$. Then 
\begin{enumerate}[(a)]
    \item $\alpha_i^*(K_{Y_i}+\frac{1}{2}D_{Y_i})=K_i+\frac{1}{2}D_i^\flat$.
    \item $A$ is almost log exceptional on $(X_i,D_i)$ if and only if $(\alpha_i)_*A$ is log exceptional on $(Y_i,D_{Y_i})$.
    \item $\kappa(K_i+\frac{1}{2}D_i)=\kappa(K_X+\frac{1}{2}D)$.
    \end{enumerate} \end{lem}

\begin{proof}See \cite{Palka-minimal_models}: Lemma 3.4, Lemma 4.1(viii) and Corollary 3.5.\end{proof}

\begin{cor}\label{cor:almost_min_models} Let $(X_n,\frac{1}{2}D_n)$ be an almost minimal model of $(X_0,\frac{1}{2}D_0)$ as defined above. One of the following holds: \begin{enumerate}[(a)]
\item $X_n\setminus D_n$ (and hence $\P^2\setminus \bar E$) has a $\C^{**}$-fibration with no base points on $X_n$, where $\C^{**}=\C^1\setminus \{0,1\}$,
\item $\kappa(K_n+\frac{1}{2}D_n)\geq 0$ and $K_n+\frac{1}{2}D_n^\flat$ is the (numerically effective) positive part of the Zariski-Fujita decomposition of $K_n+\frac{1}{2}D_n$,
\item $-(K_n+\frac{1}{2}D_n^\flat)$ is ample off $\Upsilon_n+\Delta_n$ and trivial on $\Upsilon_n+\Delta_n$; and $\rho(X_n)=1+\#(\Upsilon_n+\Delta_n)$.
\end{enumerate} 
\end{cor}

\begin{proof}See \cite{Palka-minimal_models}, proof of Theorem 4.5(4),(6).
\end{proof}

Note that in case (c) the surface $(Y_n,\frac{1}{2}D_{Y_n})$ is a log del Pezzo and $Y_n$ has at most canonical singularities; their number is equal to $b_0(\Delta_n^-)$.

\begin{ex} Let $(X_0,D_0)\to (\P^2,\bar E)$ be the minimal weak resolution of the four-cuspidal quintic from the statement of Theorem \ref{thm:at_most_4_cusps}. All cusps of $\bar E$ are semi-ordinary, so $D_0-E_0=\Delta_0^++\Upsilon_0$, with $\Delta_0^+=[2,2]$, and the morphism $\alpha_0$ is simply the minimal weak resolution, i.e.\ $\alpha_0=\pi_0$, hence $Y_0=\P^2$. In particular, $n=0$. Clearly, $\rho(Y_0)=1$ and $2K_{Y_0}+\bar E\sim -\ell$, where $\ell$ is a line on $\P^2$, so $-(K_{Y_0}+\frac{1}{2}\bar E)$ is ample. \end{ex}

We recall basic properties of the process of almost minimalization defined above.  Given two effective $\Q$-divisors $T_1$, $T_2$ we denote by $T_1\wedge T_2$ the unique effective $\Q$-divisor such that for $j=1,2$ the divisors $T_j-T_1\wedge T_2$ are effective and have no common component.
\begin{lem}\label{lem:psi-properties}\
\begin{enumerate}[(a)]
\item $\Exc \psi_{i+1}$ is a chain equal to $V+A_{i+1}+W$, where $V$ and $W$ are zero or twigs of $D_i$.
\item The support of $D_0-\Exc \psi\wedge D_0$ is connected and simply connected.
\item If $U\subseteq D_i$ is a component not contracted by $\psi_{i+1}$ then $\beta_{D_{i+1}}(\psi_{i+1}(U))=\beta_{D_i+A_{i+1}}(U)$.
\item $\Exc \psi_{i+1}$ is disjoint from $\Delta_i^++\Upsilon_i$. In particular, the process of almost minimalization does not touch exceptional divisors over semi-ordinary cusps.
\end{enumerate}
\end{lem}

\begin{proof}
(a) Since $D_i$ is connected, by the definition of $\psi_{i+1}$ and by \eqref{eq:ALE} all components of $\Exc \psi_{i+1}$ have $\beta_{D_i+A_{i+1}}=2$, hence the connected components of $\Exc \psi_{i+1}-A_{i+1}$ are zero or are twigs of $D_i$. Note that adding $A_{i+1}$ to $D_i$ creates a new loop in the support of the divisor. Now (b) and (c) follow from (a) and from the fact that $D_0$ is connected and simply connected.

(d) Suppose that $i\geq 0$ is the smallest natural number such that $\psi_{i+1}$ touches $\Delta_i^++\Upsilon_i$. By \eqref{eq:ALE} the curve $A_{i+1}$ is disjoint from $\Delta_i^++\Upsilon_i$, so we can write $\Exc \psi_{i+1}=V+A_{i+1}+W$ as in (a) with $V\cdot U>0$ for some component $U\subseteq \Upsilon_i$. By the definition of $\Upsilon_i$ it follows that $U$ meets $D_i-U$ in exactly three different points, two of which belong to distinct twigs of $D_i$. But then $U$ is not in the image of $\Upsilon_0$, so by (a) the third point is the image of some $A_j$ for $j\leq i$. But then $D_0-\Exc \psi\wedge D_0$ is not connected; a contradiction with (b).
\end{proof}

We now introduce more notation which helps to control the geometry of the divisors $D_i$. As above, $n$ denotes the length of the chosen process of almost minimalization. It equals the number of curves not contained in $D_0$ contracted by $\psi$.

\begin{notation}\label{not:tau_etc} Let $(X_i,D_i)$, $i\in \{0,\ldots, n\}$ and $j\in \{1,\ldots, c\}$ be as above.
\begin{enumerate}[(a)]
\item Denote by $\Upsilon_i^0$ the subdivisor of $\Upsilon_i$ consisting of components which do not meet $\Delta_i^+$. By definition, every component of $\Upsilon_i^0$ is a $(-1)$-curve meeting exactly one component of $D_i$.
\item Denote the numbers of semi-ordinary and ordinary cusps of $\bar E$ by $c_0$ and $c_0'$, respectively.
\item Put $\tau_j=C_j\cdot E_0$ and $\tau=\tau_1+\cdots+\tau_c$ (see the notation in Section \ref{sec:prelim}).
\item Put $s_j=1$ if $\tilde C_j=0$ and $s_j=0$ otherwise. Put $s=s_1+\cdots+s_c$.
\item Put $\tau_j^*=\tau_j-s_j-1$ and $\tau^*=\tau_1^*+\cdots+\tau_c^*=\tau-s-c$. Note that $\tau_j^*\geq 0$. 
\item Put $$p_2=p_2(\P^2,\bar E):=h^0(2K_X+D)=h^0(2K_0+D_0).$$
\item Put $\delta=\delta(\Delta_n^-)$.
\end{enumerate}
\end{notation}

Although $Q_j=\pi_0^{-1}(q_j)$ has simple normal crossings, the divisor $D_0$, and in fact every divisor $D_i$ for $i=0,\ldots, n$, has a unique non-snc point over each $q_j$, which belongs to $E_i$. The exceptional divisor of the minimal log resolution \begin{equation}\label{eq:log_res}
\varphi_i\:(X_i',D_i')\to (X_i,D_i)
\end{equation} 
over $q_j$ is $[1,(2)_{\tau_j-1}]$. Clearly, $(X,D)=(X_0',D_0')$ and $\pi=\pi_0\circ\varphi_0$. If $\P^2\setminus \bar E$ has no $\C^{**}$-fibration (cf.\ Proposition \ref{lem:basics}(c)) then by \cite[Lemmas 3.7 and 5.1(iv)]{Palka-minimal_models}, $D_i'$ is snc-minimal. As discussed before, components of $Q_j$, and similarly of $Q_j'=\pi^{-1}(q_j)$, are naturally linearly ordered as proper transforms of blow-ups constituting $\pi_0^{-1}$, respectively $\pi^{-1}$. Consequently, the sets of maximal twigs of $D_0$ and of $D$ lying over each $q_j$ come with a natural linear order. Let $\Delta_i'$ denote the sum of maximal $(-2)$-twigs of $D_i'$.

\begin{lem}\label{lem:log BMY} \ 
\begin{enumerate}[(a)]
\item The contribution of the first two maximal admissible twigs of $Q_j'$ to $\ind(Q_j')$ is strictly bigger than $\frac{1}{2}$.
\item $\ind(D_i')+i\leq 5-p_2$ for every $0\leq i\leq n$.
\item $\frac{1}{2}b_0(\Delta_0')+\frac{1}{3}c_0'\leq \ind(D)\leq 5-p_2$. 
\end{enumerate}
\end{lem}

\begin{proof}
For (a), (b) see \cite[Lemma 3.5, Lemma 4.2(iv)]{Palka-Coolidge_Nagata1}. For the first inequality in (c) note that the inductance of a $(-2)$-twig is at least $\frac{1}{2}$ and that the exceptional divisor of a log resolution over an ordinary cusp is $[3,1,2]$.
\end{proof}
 
Recall that the arithmetic genus of a divisor $T$ is defined as $p_a(T):=\frac{1}{2}T\cdot (K+T)+1$.
 
\begin{prop} \label{lem:basics} Let $p_2$ and $n$ be as above and let $i\in \{0,1,\ldots,n\}$.  
The following hold: \begin{enumerate}[(a)]
    \item $K_i\cdot (K_i+D_i)=p_2-c-\tau^*-i$. 
    \item $p_a(D_i)=i+\tau^*+c$ \ and\  $\#D_i=\rho(X_i)+i$.
    \item If $c\geq 4$ then $\P^2\setminus \bar E$ does not admit a $\C^{**}$-fibration.
    \item If $\P^2\setminus \bar E$ is not $\C^{**}$-fibered then 
   \begin{equation}\label{eq:basic_bound}
   \#D_n+\tau^*+c+\delta\leq 3p_2+8+(b_0(\Delta_n)+\#\Upsilon_n^0).
   \end{equation}
    \item $b_0(\Delta_n)+\#\Upsilon_n^0\leq b_0(\Delta_0)+c_0'$.
\end{enumerate}
\end{prop}

\begin{proof}
For (a), (b) and (c) see \cite{Palka-minimal_models}: Lemma 4.3 and Lemma 5.1(iv).

(d), (e) By the proof of \cite[Lemma 4.5(6)]{Palka-minimal_models} we have $$b_0(\Delta_n')+\#\Upsilon_n^0\leq b_0(\Delta_0')+c_0'$$ and, in case $\P^2\setminus \bar E$ is not $\C^{**}$-fibered,
\begin{equation}
\rho(X_n')+n+\delta\leq 3p_2+8+(b_0(\Delta_n')+\#\Upsilon_n^0).
\end{equation}
Since $b_0(\Delta_i')=b_0(\Delta_i)+s$ (see the proof of Lemma 4.4 loc.cit.), the first inequality gives (e). Since $\rho(X_n')+n=\#D_n'=\#D_n+\tau=\#D_n+\tau^*+s+c$, the second inequality gives (d).
 \end{proof}

Note that the inequality \eqref{eq:basic_bound} is a consequence of the fact that if $\P^2\setminus \bar E$ is not $\C^{**}$-fibered (and is of log general type) then $2K_n+D_n^\flat$ is either nef or anti-ample, so $(2K_n+D_n^\flat)^2\geq 0$.

 \medskip
 \section{Geometric restrictions on exceptional divisors.} 

Let the notation be as before. In particular, $\bar E\subseteq \P^2$ is a rational cuspidal curve with $c$ cusps and $\pi_0\:(X_0,D_0)\to (\P^2,\bar E)$ is the minimal weak resolution of singularities.  In \cite[Theorem 1.4]{Palka-minimal_models} we proved that $c\leq 6$. We now show that $c=6$ is impossible. 

\begin{prop}\label{prop:c<=5} A rational cuspidal planar curve has at most $5$ cusps.
\end{prop}

\begin{proof}
Suppose that $c\geq 6$. By Proposition \ref{lem:basics}(c),(d),(e)
\begin{equation}
\#D_n+\tau-s+\delta\leq 3p_2+8+b_0(\Delta_0)+c_0'.
\end{equation}
By Lemma \ref{lem:log BMY}(c), $s+b_0(\Delta_0)=b_0(\Delta_0')\leq 10-2p_2-\frac{2}{3}c_0'$, hence 
\begin{equation}\label{eq:|Dn|}
\#D_n+\tau+\delta\leq 18+p_2+\frac{1}{3}c_0'.
\end{equation}
We have $\#\psi_*Q_j+\tau_j\geq 3$ and the inequality is strict for non-ordinary cusps $q_j\in \bar E$, so $\#D_n+\tau\geq 1+3c_0'+4(c-c_0')$. Then \eqref{eq:|Dn|} gives $$1-c_0'+4c\leq \#D_n+\tau\leq 18+p_2+\frac{1}{3}c_0',$$ hence $7\leq 4c-17\leq p_2+\frac{4}{3}c_0'$. By Lemma \ref{lem:log BMY}(a),(c) we have $2p_2\leq 9-c\leq 3$, so $p_2\leq 1$ and hence $c_0'\geq 5$. Then $4<5\cdot \frac{5}{6}\leq \ind(D)\leq 5-p_2$, so $p_2=0$. Now the above inequality gives $2c+(c-c_0')\leq 12$, hence $c_0'=c=6$. By the genus-degree formula $\deg \bar E=5$. The restriction to $\bar E$ of the projection from any of the cusps has degree $3$. By Lemma \ref{lem:Hurwitz}, $\#\Sing E<2F\cdot E=6$; a contradiction.
\end{proof}

To prove Theorem \ref{thm:at_most_4_cusps} we may, and will, assume from now on that $c=5$ or $c=4$. We will proceed in a similar way as in the proof of Proposition \ref{prop:c<=5}, but now we need to couple previous arguments with a more detailed analysis of the geometry of the divisors $Q_j$ and of the minimalization process $$\psi=\psi_n\circ\dots\circ\psi_1\: (X_0,D_0)\rightarrow (X_n,D_n).$$ The morphism $\psi_{i+1}$ lifts to the level of minimal log resolutions, see \eqref{eq:log_res}, $\varphi_i\:(X_i',D_i')\to (X_i,D_i),$ where $(X_0',D_0')=(X,D)$. We denote the lift by $\psi_{i+1}'\:(X_i',D_i')\to (X_{i+1}',D_{i+1}')$ and we put $\psi'=\psi'_n\circ\dots\circ\psi'_1$. We get a commutative diagram:
$$ \xymatrix{{} & {(X,D)}\ar@{>}[r]^-{\psi'}\ar@{>}[d]^{\varphi_0}\ar@{>}[ld]_-{\ \pi} &{(X_n',D_n')}\ar@{>}[d]^{\varphi_n} \\
 {(\P^2,\bar E)} & {(X_0,D_0)}\ar@{>}[r]_-{\psi}\ar@{>}[l]^-{\ \pi_0} &{(X_n,D_n)}}$$
We use log resolutions mostly to compute the inductance and to effectively use the logarithmic Bogomolov-Miyaoka-Yau inequality. We denote the proper transform of $E_0$ on $X_i'$ by $E_i'$. By $Q_j'\subseteq D$ we denote the exceptional divisor of $\pi_0\circ \varphi_0$ over $q_j$. Clearly, $\varphi_0(Q_j')=Q_j$. We define $\ind_j^{(i)}$ as the sum of inductances for maximal (admissible) twigs of $D_i'$ contained in the image of $Q_j'$ under $\psi_i'\circ\dots\circ\psi_1'$. Since $c\neq 1$, $E_i'$ is not a twig of $D_i'$, hence $\ind(D_i')=\ind_1^{(i)}+\cdots+\ind_c^{(i)}$. The log BMY inequality, Lemma \ref{lem:log BMY}(b), reads for $(X_i',D_i')$ as 
\begin{equation}\label{eq:ind} 
\ind_1^{(i)}+\cdots+\ind_c^{(i)}+p_2\leq 5-i, \quad  i=0,\ldots,n.
\end{equation}
We write $\ind_j:=\ind_j^{(0)}$ for simplicity. Note that $\ind_j=\ind(Q_j')$, but things are more complicated for $\ind_j^{(i)}$ with $i>0$, as for instance twigs of images of distinct $Q_j'$ may meet, and hence will no longer be twigs of $D_i'$.

Recall that given two effective $\Q$-divisors $T_1$, $T_2$ we denote by $T_1\wedge T_2$ the unique effective $\Q$-divisor such that for $j=1,2$ the divisors $T_j-T_1\wedge T_2$ are effective and have no common component. Denote by $\Delta_i\tip$ the sum of $(-2)$-tips of $D_i$. Each connected component of $\Delta_i$ contains a unique component of $\Delta_i\tip$. Put
\begin{equation}
\bar Q_j=\psi_*(Q_j), \quad \Omega_j=\bar Q_j-(\Upsilon_n^0+\Delta_n\tip)\wedge \bar Q_j 
\end{equation}  
and define the integers:
\begin{equation}\label{eq:lambda_def}
\lambda_{j}:=\#\Omega_j+\tau_j-s_j\geq \#\Omega_j+1.
\end{equation}
An important ingredient in our analysis is combining the inequality \eqref{eq:ind} with the inequality \eqref{eq:basic_bound}. Since $\#\Omega_j=\#(\bar Q_{j}- \Upsilon_{n}^{0}\wedge \bar Q_{j})-b_{0}(\Delta_{n}\wedge \bar Q_{j})$, we can write the latter as 
\begin{equation}\label{eq:lambda} 
c+\sum_{j=1}^c(\#\Omega_j+\tau_j^*)= \sum_{j=1}^c \lambda_j\leq 3p_2+7-\delta.
\end{equation}
We will see below that $p_2\leq 2$, so the right hand side is bounded from above by $13$, hence $\lambda_1+\ldots+\lambda_c$, which roughly measures the number of components of $\psi(D_0-E_0)$, cannot be too big. We now proceed to describe cusps $q_j\in \bar E$ with small $\lambda_j$. We first collect some simple observations.

\begin{lem}\label{lem:basic_bounds}\
\begin{enumerate}[(a)]  
\item We have $\#\Omega_j=0$ if and only if $q_{j}\in \bar{E}$ is an ordinary cusp.
\item If $q_{j}\in \bar{E}$ is a semi-ordinary cusp then $\lambda_{j}=\#\Omega_j+1=\#Q_{j}$.
\item If $s_j=0$ then $\#\Omega_j\geq 2$ and $\lambda_j\geq 4$.
\item If $\tau_j=2$ then $\psi$ does not touch $C_j$.
\item \label{le:typesII} We have $\lambda_j=2$ if and only if $q_j$ is a semi-ordinary cusp with $Q_j=[1,2]$. In particular, $\ind^{(n)}_j=\frac{11}{10}$ in case $\lambda_j=2$.
\item $p_2\in \{0,1,2\}$.
\item If $s_j=1$ then $\ind_j^{(n)}\geq 1-\frac{1}{\tau_j}\geq \frac{1}{2}$. 
\item If $\lambda_j=4$ then $\ind_j^{(n)}\geq \frac{1}{2}$ or $\psi(C_j)$ is a $(-1)$-tip of $\bar Q_j$, $\tau_j=2$ and $\ind_j^{(n)}\geq \frac{1}{3}$.
\end{enumerate}
\end{lem}

\begin{proof}
(a) We have $\#\Omega_j=0$ if and only if $\psi_{*}C_{j}\subseteq \Upsilon_{n}^{0}$, which is equivalent to $C_{j}\subseteq \Upsilon_{0}^{0}$, hence to $q_{j}\in \bar{E}$ being an ordinary cusp. 

(b) We have $Q_j\subseteq \Upsilon_0+\Delta_0^+$, so $Q_j$ is not touched by $\psi$. Also, $\lambda_j-1=\#\Omega_j$ and there is a unique component of $\bar Q_j$ in $\Upsilon_n^0+\Delta_n\tip$, so $\#\Omega_j=\#\bar Q_j-1$.

(c) By Lemma \ref{lem:psi-properties}(a), $C_j$ and $\tilde C_j$ are not contracted by $\psi$, so $\psi(C_j+\tilde C_j)$ is a part of $\Omega_j$, hence $\lambda_j+s_j\geq \tau_j+2\geq 4$.

(d) We may assume that $j=1$. Suppose that $\tau_1=2$ and $\psi$ touches $C_1$. By Lemma \ref{lem:psi-properties} any connected component of $\Exc \psi$ meeting $C_1$ is a chain $T+A+W$, where $A$ is the proper transform of one of the curves $A_{i+1}\subseteq X_i$ and $T$, $W$ are twigs of $D_0$, say $\#W\leq \#T$. Contract $(-1)$-curves in $T+A+W$ until $C_1$ is touched once and denote the resulting morphism by $\eta\:X_0\to \eta(X_0)$. The linear system of $F=\eta(C_1)$ defines a $\P^1$-fibration of $\eta(X_0)$. 

Suppose that $(W+A+T)\cdot E_0=0$. Then $F\cdot \eta(E_0)=C_1\cdot E_0=2$. We may assume that $W+T\subseteq Q_1+Q_2$. For $j>2$ the divisor $\eta(Q_j)$  is contained in some fiber. The restriction of the fibration to $\eta(E_0)$ has degree $2$ and has at least three ramification points, namely the points $\eta(E_0)\cap \eta(C_j)$ for $j=1,3,4$. This is a contradiction with the Hurwitz formula.

Thus $W=0$ and $A\cdot E_0=1$. Then $T=[(2)_k]$ for some $k\geq 1$ and $T+A$ is a connected component of $\Exc \psi$. We may assume that $T+A=\Exc \psi_1$, hence $\eta=\psi_1$. The restriction of the fibration to $E_1$ has degree $F\cdot E_1=3$ and has ramification points at $E_1\cap \psi_1(C_j)$ for $j=1,\ldots,c$. By the Hurwitz formula $c=4$ and all ramification indices are equal to $2$. It follows that the cusps $q_2, q_3, q_4$ are semi-ordinary, so $\psi$ does not touch $Q_2+Q_3+Q_4$. For $j=2,3,4$ we write $Q_j=[1,(2)_{t_j}]$, where $t_j\geq 0$ and we compute $\ind_j=\frac{5}{6}+\frac{4t_j}{3(2t_j+3)}\geq \frac{5}{6}$. The inequalities \eqref{eq:ind} and \eqref{eq:lambda} read as 
\begin{equation}\label{eq:sum_ti}
\ind_1^{(n)}+\frac{4}{3}\sum_{j=2}^4 \frac{t_j}{2t_j+3}\leq 2\tfrac{1}{2}-n-p_2,
\end{equation}
\begin{equation}\label{eq:lambda_1}
\lambda_1+t_2+t_3+t_4\leq 3p_2+4-\delta.
\end{equation}

We have $n>0$, so $p_2\leq 1$ and $n\leq 2$. Since $q_1$ is not semi-ordinary, there is a component $U\subseteq Q_1-T-C_1$ meeting $C_1$. By the contractibility of $Q_1$ to a smooth point $U$ is unique and $\psi_1(U)=[k+2]$. We denote by $\theta\:X_n\to \theta(X_n)$ the contraction of $\psi(Q_2+Q_3+Q_4)$. We need to determine the geometry of $V=Q_1-T-C_1-U$. 

First consider the case when $\psi_*(V)=0$. We have $\rho(\theta(X_n))=\#\theta(D_n)-n\leq 3-n$. Suppose that $n=2$. Then $\psi$ does not contract $U$ and $\theta(X_2)=\P^2$. In particular, $\theta(\psi_2(F))\subseteq \P^2$ is not a $0$-curve, so $\psi_2$ touches $F$. It follows that $A_2$ meets $F$ and hence $\psi_2$ touches $\psi_1(U)$ at most once. But then $\theta(\psi(U))^2<0$; a contradiction. Thus $n=1$. Now $\theta(X_1)$ is a Hirzebruch surface, namely $\F_{k+2}$, and $\bar U:=\theta\circ\psi(U)$ is the negative section. Put $\bar F=\theta\circ\psi(C_1)$. Using the multiplicity sequences of singularities of $\theta(E_1)$ we compute that $p_a(\theta(E_1))=3+t_2+t_3+t_4$. On the other hand, on $\F_{k+2}$ we have $K_{\F_{k+2}}\sim-2\bar U-(k+4)\bar F$ and $\theta(E_1)\sim 3\bar U+(3k+6+\alpha)\bar F$, where $\alpha=E_1\cdot \psi(U)$, so $p_a(\theta(E_1))=3k+4+2\alpha$, hence $t_2+t_3+t_4=3k+2\alpha+1\geq 4$. Since $\lambda_1\geq 2$, \eqref{eq:lambda_1} gives $k=p_2=1$ and $\alpha=0$. But the latter gives $s_1=1$, so $\ind_1^{(1)}=\frac{1}{2}+\frac{1}{3}$, which contradicts \eqref{eq:sum_ti}.

We obtain $\psi_*(V)\neq 0$. Consider the case $p_2=1$. Then $n=1$ and $\ind_1^{(1)}\leq \frac{1}{2}$, which by Lemma \ref{lem:log BMY}(a) implies that $Q_1$ is a chain and $U=\tilde C_1$. By the contractibility of $Q_1$ to a smooth point $Q_1=[(2)_k,1,k+2,(2)_{k'}]$ for some $k'\geq 0$. Since $\psi_*(V)\neq 0$, we have $k'\geq 1$ and then the inequality $\ind_1^{(1)}\leq \frac{1}{2}$ fails (see Lemma \ref{lem:ind}); a contradiction.

We are left with the case $p_2=0$. Suppose that $n=2$. By \eqref{eq:sum_ti}, $\ind_1^{(2)} \leq \frac{1}{2}$. If $\psi_2$ touches $\psi_1(C_1)$ then by the argument above (with the proper transform of $A_2$ taken for $A$) we get $\Exc \psi_2=A_2+T'$, where $T'\neq 0$ is a $(-2)$-twig of $D_1$ meeting $\psi_1(C_1)$. Since there is no such, $\psi_2$ does not touch $\psi_1(C_1)$. If $A_2\cdot E_1=0$ then $A_2\cdot \psi_1(Q_1-C_1)=2$, so since $\psi_1(Q_1-C_1)$ is connected, there is a rational circular divisor (loop) in $\psi_1(Q_1-C_1)+A_2$ and hence the intersection of this divisor with a general fiber is at least $2$. The latter is impossible, because $\psi_1(Q_1-C_1-U)+A_2$ is vertical and $\psi_1(U)$ is a $1$-section. So we infer that $A_2\cdot E_1=1$. If $s_1=1$ then, since $\ind_1^{(2)}\leq \frac{1}{2}$, we see that $\psi_1(Q_1)+A_2$ is a chain as above, in which case $\psi_*(V)=0$, contrary to the claim above. Therefore, $s_1=0$ and hence $U=\tilde C_1$. It follows that $\#\Omega_1\geq 2$, so $\lambda_1\geq \#\Omega_1+2\geq 4$. By \eqref{eq:lambda_1}, $q_2, q_3, q_4$ are ordinary cusps, $\Delta_2^-=0$ and $\lambda_1=4$, so $\psi(Q_1-C_1-\tilde C_1)\subseteq \Upsilon_2^0+\Delta_2\tip\wedge \Delta_2^+$. But since $A_2$ meets $E_1$, $\psi_2$ does not create a component of $\Upsilon_2^0$, which gives $\psi(Q_1-C_1-\tilde C_1)\subseteq \Delta_2^+$. We have $\psi_1(\tilde C_1)=[k+2]$ and $k\geq 1$, so since $\psi_2$ touches $\psi_1(\tilde C_1)$ at most once, $\psi(\tilde C_1)\nsubseteq \Upsilon_2$, hence $\Upsilon_2\wedge \psi(Q_1)=0$ and so $\Delta_2^+\wedge \psi(Q_1)=0$. Then $\psi_*(V)=0$; a contradiction.

We have now $p_2=0$ and $n=1$. Since $\Exc \psi=T+A_1$, we have $\Upsilon_1\wedge \psi(Q_1)=0$, so $\Delta_1^+\wedge \psi(Q_1)=0$. Suppose that $\lambda_1>3$. Then \eqref{eq:lambda_1} implies that $q_2, q_3, q_4$ are ordinary, $\lambda_1=4$ and $\Delta_1^-=0$. We obtain $\psi(Q_1)\wedge(\Delta_1+\Upsilon_1)=0$. Then $4=\lambda_1=\#\psi(Q_1-C_1)+3-s_1$, so $\psi(Q_1-C_1)$ has $s_1+1\leq 2$ components, hence $\psi(Q_1)$ is a chain. Since $Q_1$ contracts to a smooth point, we get $\psi(Q_1)=[0,k+2,(2)_{s_1}]$. But $\psi(Q_1)\wedge\Delta_1=0$, so $s_1=0$. Then $\psi_*(V)=0$; a contradiction. We obtain $\lambda_1\leq 3$, hence $s_1=1$ by Lemma \ref{lem:basic_bounds}(c). Since $\psi(U)=[k+2]\nsubseteq \Delta_1$, we infer that $\psi(Q_1-C_1-U)$ is contained in $\Delta_1\tip\wedge\Delta_1^-$. Since $Q_1$ contracts to a point, $\psi(Q_1)=[0,k+2,(2)_m]$ with $m\leq 1$ and $\lambda_1=3$. Since $\psi_*(V)\neq 0$, we have $m=1$. By Lemma \ref{lem:Sigma} every fiber of the fibration induced by $|\theta\circ\psi(C_1)|$ has at most one component not contained in $\theta(\psi(Q_1))$. Then the fiber containing $\Delta_1^-=[2]$ has two components, which is impossible; a contradiction.

(e) Assume that $\lambda_1=2$. We have $\#\Omega_1\geq 1$, because otherwise $q_1$ is an ordinary cusp by (a), in which case $\lambda_1=1$. Since $\lambda_1=\tau_1-s_1+\#\Omega_1\geq 1+\#\Omega_1$, we get $\tau_1-s_1=1$ and $\#\Omega_1=1$, hence $\tau_1=2$, $s_1=1$ and $\bar Q_1-\psi(C_1)\subseteq \Upsilon_n^0+\Delta_n\tip$. The divisor $Q_1$ is a chain, because otherwise by Lemma \ref{lem:psi-properties}(c) the image of the branching component would contribute to $\Omega_1-\psi(C_1)$. Suppose $C_1$ meets two components $T_1$ and $T_2$ of $Q_1$. By (d) those components are not contracted by $\psi$ and their images on $X_n$ are not in $\Upsilon_n^0$, hence are contained in $\Delta_n\tip$. But then $T_1$ and $T_2$ are contained in $\Delta_0$, so $Q_1$ contains $[2,1,2]$ and hence is not negative definite; a contradiction. It follows that $C_1$ is a tip of $Q_1$. By the contractibility of $Q_1$ to a smooth point $Q_1=[1,2,\ldots,2]$, so $q_1$ is semi-ordinary. By (b), $Q_1=[1,2]$, so the exceptional divisor of the minimal log resolution of $q_1$ is $[2,3,1,2]$. We get $\ind^{(n)}_1=\ind_1=\frac{3}{5}+\frac{1}{2}=\frac{11}{10}$.

(f) By \eqref{eq:ind} and  by Lemma \ref{lem:log BMY}(a), $2\leq \frac{1}{2}c<\ind(D)\leq 5-p_2$, so $p_2\leq 2$.

(g) The $(-2)$-chain created by resolving the tangency of $\psi(C_j)$ and $E_n$ creates a maximal $(-2)$-twig $[(2)_{\tau_j-1}]$, so $\ind_j^{(n)}\geq 1-\frac{1}{\tau_j}\geq\frac{1}{2}$.

(h) We may assume that $j=1$ and $\ind_1^{(n)}<\frac{1}{2}$. Then $s_1=0$ by (g) and $\Delta_n\wedge \bar Q_1=0$. We have $4=\lambda_1\geq \#(\psi(C_1)+\psi(\tilde C_1))+\tau_1$, hence $\Omega_1=\psi(C_1)+\psi(\tilde C_1)$ and $\tau_1=2$. It follows that $\bar Q_1-\psi(C_1)-\psi(\tilde C_1)\subseteq \Upsilon_n^0$. By (d), $\psi$ does not touch $C_1$, hence $\psi(C_1)$ is a $(-1)$-tip of $\bar Q_1$. Its proper transform is a $(-3)$-tip of $D_n'$, so $\ind_j^{(n)}\geq\frac{1}{3}$.
\end{proof}

\begin{lem}\label{lem:n<=1} The process of almost minimalization of $(X_0,\frac{1}{2}D_0)$ contracts at most one curve not contained in $D_0$, that is, $n\leq 1$.
\end{lem}

\begin{proof}
Recall that $c_0$ is the number of semi-ordinary cusps of $\bar E$. Denote by $\w_3$ be the number of non-semi-ordinary cusps of $\bar E$ for which $\lambda_j=3$. By Lemma \ref{lem:basic_bounds}(c),(g) cusps with $\lambda_j=3$ have $s_j=1$ and $\ind_j^{(n)}\geq\frac{1}{2}$. The inequality \eqref{eq:ind} gives
\begin{equation}\label{eq:1}
\frac{5}{6}c_0+\frac{1}{2}\w_3\leq \sum_{j=1}^{c}\ind_j^{(n)}\leq 5-n-p_2.
\end{equation}
On the other hand, Lemma \ref{lem:basic_bounds}(e) and \eqref{eq:lambda} give 
\begin{equation}\label{eq:2}
4c-\w_3-3c_0=4(c-\w_3-c_0)+3\w_3+c_0\leq \sum_{i=1}^c\lambda_j\leq 3p_2+7-\delta.
\end{equation}
Multiplying the first inequality by $3$ and adding the second one we obtain $\frac{7}{2}c+\frac{1}{2}\w_3+\frac{1}{2}(c-c_0)+3n\leq 22-\delta$.
In particular, $3n\leq 22-\frac{7}{2}\cdot 4=8$, so $n\leq 2$. Suppose that $n=2$. Then the previous inequality gives $\frac{7}{2}c\leq 16$, so $c=4$. We get
\begin{equation*}
9-\w_3-3c_0+\delta\leq 3p_2\leq 9-\frac{3}{2}\w_3-\frac{5}{2}c_0.
\end{equation*}
Since $n\neq 0$, we have $c_0\leq c-1=3$ and we may assume that $\psi$ touches $Q_1$. By Lemma \ref{lem:basic_bounds}(e), (f), $\lambda_1\geq 3$ and $p_2\leq 2$.

Assume $c_0=0$. By Lemma \ref{lem:basic_bounds}(a),(e), $\lambda_1+\cdots+\lambda_4\geq 4\cdot 3=12$, so $p_2=2$. Then $\frac{1}{2}\w_3\leq \sum_{j=1}^{4}\ind_j^{(n)}\leq 1$ by \eqref{eq:1}, so $\w_3\leq 2$ and hence $\lambda_1+\cdots+\lambda_4\geq 2\cdot 4+2\cdot 3=14$; a contradiction with \eqref{eq:2}.

Assume $c_0=1$. Then $6-\w_3+\delta\leq 3p_2\leq 6\frac{1}{2}-\frac{3}{2}\w_3$, so $\w_3\leq 1$ and $p_2=2$. Then $\w_3=0$, so $\lambda_1+\cdots+\lambda_4\geq 3\cdot 4+1\geq 13$. By \eqref{eq:2} in fact an equality holds, so $(\lambda_1,\lambda_2,\lambda_3,\lambda_4)=(4,4,4,1)$. By \eqref{eq:1}, $\sum_{j=1}^3\ind_j^{(2)}\leq 1-\frac{5}{6}=\frac{1}{6}$. This is a contradiction with Lemma \ref{lem:basic_bounds}(h).

Assume $c_0=2$, say $q_3$ and $q_4$ are semi-ordinary. We have $3-\w_3+\delta\leq 3p_2\leq 4-\frac{3}{2}\w_3\leq 4$, which gives $p_2=1$ and $\w_3=0$. By \eqref{eq:2},  $(\lambda_1,\lambda_2,\lambda_3,\lambda_4)=(4,4,1,1)$. By \eqref{eq:1}, $\ind_1^{(2)}+\ind_2^{(2)}\leq 2-2\cdot \frac{5}{6}=\frac{1}{3}$.  Again, this is a contradiction with Lemma \ref{lem:basic_bounds}(h).

We are left with the case $c_0=3$, say $q_2$, $q_3$ and $q_4$ are semi-ordinary. Then $\delta-\w_3\leq 3p_2\leq\frac{3}{2}(1-\w_3)\leq \frac{3}{2}$, hence $p_2=0$ and $\w_3\geq \delta$. By \eqref{eq:2}, $\lambda_3=\lambda_4=1$, that is, $q_3$ and $q_4$ are ordinary, and $\lambda_1+\lambda_2\leq 5-\delta$. By \eqref{eq:1}, $\ind_1^{(2)}+\ind_2^{(2)}\leq 3-2\cdot \frac{5}{6}=\frac{4}{3}$. If $\lambda_2=2$ then $\lambda_1=3$, so by Lemma \ref{lem:basic_bounds}(e),(g), $\ind_1^{(2)}+\ind_2^{(2)}\geq \frac{1}{2}+\frac{11}{10}>\frac{4}{3}$, which is false. Thus $\lambda_2=1$ and $\lambda_1\in \{3,4\}$. We obtain $\ind_1^{(2)}\leq \frac{4}{3}-\frac{5}{6}=\frac{1}{2}$.

Suppose that $s_1=1$. Since $\ind_1^{(2)}\leq \frac{1}{2}$, it follows that $\tau_1=2$ and $\bar Q_1+E_1$ has no tips. In particular, $\Delta_2\wedge \bar Q_1=0$. By Lemma \ref{lem:basic_bounds}(d), $\psi(C_1)$ is a $(-1)$-curve meeting any other component of $\bar Q_1-\psi(C_1)$ at most once, hence disjoint from $\Upsilon_2^0$. Let $\theta$ be the contraction of $\Upsilon_2^0+\psi(C_1)$. We have $\rho(\theta(X_2))=\#\theta(\bar Q_1+E_2)-2=\#(\bar Q_1-\psi(C_1)-\Upsilon_2^0\wedge\bar Q_1)-1=\lambda_1-3$, hence $\lambda_1=4$ and $\theta(E_2)$ is a rational cuspidal curve on $\P^2$ with exactly four cusps, all ordinary. The genus-degree formula implies that such a curve does not exist; a contradiction.

Thus $s_1=0$ and, by Lemma \ref{lem:basic_bounds}(c), $\lambda_1=4$. Since $\delta\leq \omega_3$, we have $\Delta_2^-=0$. The curve $\psi(\tilde C_1)$ is a component of $\Omega_1$, hence $\Omega_1=\psi(C_1)+\psi(\tilde C_1)$ and $\tau_1=2$. Thus $\bar Q_1-\psi(C_1+\tilde C_1)\subseteq \Delta_2\tip+\Upsilon_2^0$. But since $\bar Q_1$ contains no component of $\Upsilon_2-\Upsilon_2^0$, it contains no component of $\Delta_2^+$, hence $\bar Q_1-\psi(C_1+\tilde C_1)\subseteq \Upsilon_2^0$. Consequently, no component of $Q_1$ other than (possibly) $\tilde C_1$ is branching in $Q_1$. Since $\tau_1=2$, by Lemma \ref{lem:basic_bounds}(d), $\psi$ does not touch $C_1$, so $\psi(C_1)\cdot \Upsilon_2^0=0$, hence $\psi(C_1)$ is a $(-1)$-tip of $\bar Q_1$. Since $q_2$, $q_3$ and $q_4$ are semi-ordinary, $\psi$ does not touch $Q_2+Q_3+Q_4$, so the proper transform of $A_i$, $i\in \{1,2\}$ meets exactly one component of $T_i\wedge \Delta_0^-$, where $T_i$ is a twig of $Q_1$ meeting $\tilde C_1$, other than $C_1$. We have $\psi(T_i)\subseteq \Upsilon_2^0$, so the proper transform of $A_i$ meets $\tilde C_1$ too, and $\psi$ contracts all but one component of $T_i$. It follows that $T_1+T_2\subseteq \Delta_0^-$, hence $Q_1$ is not negative definite; a contradiction.
\end{proof}

By Lemma \ref{lem:n<=1} we have $n\leq 1$, so $\psi$ touches $Q_j$ if and only if $A_1\cdot Q_j>0$. For convenience we put  $A_1=0$ if $n=0$.

\begin{prop}\label{prop:lambda=3}Assume $\lambda_j=3$. Then $s_j=1$ and $\psi$ does not touch $Q_j$. Moreover, one of the following holds:
\begin{enumerate}[(1)]
\item $\tau_j=2$. Then $Q_j=[1,2,2]$ or $Q_j=[2,1,3]$ or $Q_j=[2,1,3,2]$. We have $\ind_j=\frac{17}{14}$ or $\ind_j=\frac{4}{3}$ or $\ind_j=\frac{8}{5}$, respectively. 

\item $\tau_j=3$. Then $Q_j=[1]$ or $Q_j=[1,2]$. We have $\ind_j=\frac{11}{12}$ and $\ind_j=\frac{26}{21}$, respectively. 

\end{enumerate}
\end{prop}

\begin{proof}
Renumbering the cusps if necessary, we may assume that $j=1$. By Lemma \ref{lem:basic_bounds}(c), $s_1=1$, so $\tau_1=4-\#\Omega_1\leq 3$. 
 
Consider the case $\tau_1=2$ and $A_1\cdot Q_1=0$. Then $\#\Omega_1=2$, so for some component $U\subseteq Q_1$ meeting $C_1$ we have $Q_1-U-C_1\subseteq\Upsilon_0^0+\Delta_0\tip$. But $q_1$ is not ordinary, so $Q_1\wedge \Upsilon_0^0=0$. Since $Q_1$ contracts to a smooth point and $U\nsubseteq \Delta_0\tip$, we infer that $Q_1$ has three or four components, hence $Q_1$ is $[1,2,2]$, $[2,1,3]$ or $[2,1,3,2]$. In the first case the exceptional divisor of the minimal log resolution over $q_1$ is $[2,1,3,2,2]$, hence $\ind_1=\frac{5}{7}+\frac{1}{2}=\frac{17}{14}$. In the second and third case the exceptional divisor is a fork with maximal twigs $[2]$, $[3]$, $[2,1]$ and $[2]$, $[2,3]$, $[2,1]$ respectively, which gives $\ind_1=\frac{1}{2}+\frac{1}{3}+\frac{1}{2}=\frac{4}{3}$ and $\ind_1=\frac{1}{2}+\frac{3}{5}+\frac{1}{2}=\frac{8}{5}$ respectively.
  
Consider the case $\tau_1=3$ and $A_1\cdot Q_1=0$.  Then $\Omega_1=C_1$, so $Q_1-C_1\subseteq \Upsilon_0^0+\Delta_0\tip$. Since $\tau_1\neq 2$, we have in fact $Q_1-C_1\subseteq \Delta_0\tip$, so because $Q_1$ contracts to a smooth point, we get $Q_1=[1]$ or $Q_1=[1,2]$. The exceptional divisors of the minimal log resolution over $q_1$ are $[4,1,2,2]$ and $[2,4,1,2,2]$, respectively, so $\ind_1=\frac{1}{4}+\frac{2}{3}=\frac{11}{12}$ and $\ind_1=\frac{4}{7}+\frac{2}{3}=\frac{26}{21}$, respectively. 

\smallskip 
By Lemma \ref{lem:n<=1}, $n\leq 1$, so it remains to prove that $A_1\cdot Q_1=0$. Suppose that  $A_1\cdot Q_1\neq 0$. We have $s_1=1$ and $\tau_1=4-\#\Omega_1\leq 3$.

(1$^\circ$) Assume $\tau_1=3$. Then $\Omega_1=\psi(C_1)$, so $$\bar Q_1-\psi(C_1)\subseteq \Upsilon_1^0+\Delta_1\tip.$$ Note that if $\psi(U)\subseteq \Upsilon_1^0+\Delta_1\tip$ for some component $U$ of $Q_1$ not contained in the twigs of $D_0$ then in fact $\psi(U)\subseteq \Upsilon_1^0$, so by Lemma \ref{lem:psi-properties}(a), $U\subseteq \Upsilon_0^0$. But $\Upsilon_0^0\wedge Q_1=0$, so we infer that images of all components of $Q_1$ not contained in twigs of $D_0$ are in $\Omega_1$. Since $\Omega_1=\psi(C_1)$ and $C_1$ is non-branching in $Q_1$, we infer that $Q_1$ is a chain. We can write it as $Q_1=T_1+C_1+T_2$, where $T_1$ and $T_2$ are zero or twigs of $D_0$ contained in $Q_1$ such that $A_1\cdot T_1\geq A_1\cdot T_2$. The curve $A_1$ can meet $T_1+T_2$ only in tips of $Q_1$, because otherwise by Lemma \ref{lem:psi-properties}(c) the image of the non-tip it meets would become a component of $\Omega_1+\bar Q_1\wedge (\Upsilon_1-\Upsilon_1^0)$, which consists of $\psi(C_1)$ only. 

Suppose that $A_1\cdot T_2=1$. Then $A_1$ meets the tips of $Q_1$ contained in $T_1$ and $T_2$. The inclusion $\bar Q_1-\psi(C_1)\subseteq \Upsilon_1^0+\Delta_1\tip$ implies now that $\psi(T_1+T_2)$ is a component of $\Upsilon_1^0$. The contraction of $\bar Q_2+\cdots+\bar Q_c+\psi(T_1+T_2)$ maps $X_1$ onto $\P^2$ so that the images of $C_1$ and $E_0$ are two singular curves with intersection number equal to $\tau_1=3$; a contradiction. 

Thus $A_1\cdot T_2=0$. Suppose that $\psi$ does not touch $C_1$. Then $C_1$ meets some component $U$ of $Q_1-C_1$ and for every such component $\psi_*(U)\neq 0$, so $\psi(U)\subseteq \Delta_1\tip$ and hence $U\subseteq \Delta_0\tip$. It follows that $Q_1=[1,2]$. But then $A_1$ meets $C_1$ or the $(-2)$-tip, hence in any case $\psi$ touches $C_1$; a contradiction.

Thus $\psi$ touches $C_1$. Suppose that $A_1\cdot E_0=1$. Since $A_1$ meets $\Delta_0$ and $\psi$ touches $C_1$, we infer that $T_1$ is a non-zero $(-2)$-chain. Since $A_1\cdot T_2=0$, we see that $\psi(T_2)\subseteq \Delta_1\tip$, so $T_2\subseteq \Delta_0\tip$, hence by the negative definiteness of $Q_1$ we get $T_2=0$, that is, $A_1+Q_1=[1,(2)_k,1]$ for some $k\geq 1$. The contraction of $D_0-E_0+A_1-C_1$ maps $X_0$ onto $\P^2$ and $C_1$ onto a $0$-curve; a contradiction.

Thus $A_1\cdot E_0=0$. Since $A_1\cdot D_0=2$, we may assume that $A_1\cdot Q_2=1$ and $A_1\cdot (Q_3+\cdots+Q_c)=0$. 
Contract $Q_3+\cdots+Q_c$ and then perform the blow-downs constituting $\psi$ until $C_1$ is touched once. If the contractions did not touch $C_2$ then we contract $C_2$, too. Denote the image of $C_1$ by $F$ and the unique intersection point of the images of $Q_2$ and $E_1$ by $q_2'$. By Lemma \ref{lem:Hurwitz}, $\displaystyle (r_F(q_2')-1)+\sum_{j=3}^c(\mu(q_j)-1)\leq 2\cdot 3-1-3=2$, where $\mu(q_j)$ is the multiplicity of the cusp $q_j$. We infer that $c=4$, $r_F(q_2')=1$ and that $q_3$, $q_4$ are semi-ordinary. The equality $r_F(q'_2)=1$ implies that $q_2'$ is a smooth point of the image of $E_1$, hence $C_2$ has not been contracted. We infer that $\psi$ touches $C_2$ before it touches $C_1$. By \eqref{eq:ind}, $\ind_1^{(1)}+p_2\leq 4-2\cdot \frac{5}{6}=2\frac{1}{3}$. Since $s_1=1$ and $\tau_1=3$, we have $\ind_1^{(1)}\geq \frac{2}{3}$, hence $p_2\leq 1$. On the other hand, since $A_1$ meets $Q_2$, we have $\lambda_2\geq 3$, so \eqref{eq:lambda} gives $3p_2-\delta\geq 3+3+1+1-7=1$, hence $p_2=1$. We get $3\leq \lambda_2\leq 7-\delta-\lambda_3-\lambda_4\leq 5-\delta$. Since $\psi$ touches $C_2$, we have $\tau_2\geq 3$ by Lemma \ref{lem:basic_bounds}(d). It follows that $\lambda_2\geq 4$, for otherwise by symmetry $\psi$ touches $C_1$ before it touches $C_2$, which is impossible. Thus $\lambda_2$ equals to $4$ or $5$. 

Suppose that $s_2=1$. Then, since $q_3$ and $q_4$ are semi-ordinary, we get $2\cdot \frac{2}{3}\leq \ind_1^{(1)}+\ind_2^{(1)}\leq 3-\ind_3^{(1)}-\ind_4^{(1)}\leq3-2\cdot \frac{5}{6}=\frac{4}{3}$. It follows that $\ind_j^{(1)}=\frac{2}{3}$ for $j=1,2$ and that $\bar Q_1+\bar Q_2+E_1$ has no twigs. Then $T_2=0$, $\tau_2=3$ and $\bar Q_2=\psi(C_2)$, in which case $\lambda_2=3$; a contradiction.

Thus $s_2=0$. We get $5-\delta \geq \lambda_2\geq 2+\tau_2\geq 5$, so $(\lambda_1,\lambda_2,\lambda_3,\lambda_4)=(3,5,1,1)$, $\tau_2=3$ and $\delta=0$. We have $\Omega_1=\psi(C_1)$, hence $T_2\subseteq \Delta_0^-$ and so $T_2=0$, because $\delta=0$. Then $T_1=[(2)_k]$ for some $k\geq 1$ and hence $A_1$ meets $C_2$ and the tip of $D_0$ contained in $T_1$. Since $\delta=0$ and $\#\Omega_2=\lambda_2-3=2$, we have $\psi(Q_2)=\psi(C_2+\tilde C_2)$, so $Q_2=[1,2]$. Then $\pi_0(A_1)^2=2$, which is not a square; a contradiction.

\smallskip
(2$^\circ$) Now assume $\tau_1=2$.  Since $\lambda_1=3$ and $s_1=1$, we have $\Omega_1=\psi(C_1)+\psi(U)$ for some component $U$ of $Q_1$ and hence  
\begin{equation}\label{eq:lambda=3_1}
\bar Q_1-\psi(C_1)-\psi(U)\subseteq \Upsilon_1^0+\Delta_1\tip.
\end{equation} By Lemma \ref{lem:basic_bounds}(d), $\psi$ does not touch $C_1$, so $U$ meets $C_1$.  By Lemma \ref{lem:psi-properties}(c) all components of $Q_1-U$ are non-branching in $Q_1$.

Suppose that $U$ is a branching component of $Q_1$. Denote by $T_1$ and $T_3$ the twigs of $D_0$ meeting $U$. By the contractibility of $Q_1$ to a smooth point, $T_2:=Q_1-U-T_1-T_3-C_1$ is a $(-2)$-chain and $C_1+T_2$ is a twig of $Q_1$. Note that $A_1$ can meet $T_1+T_2+T_3$ only in tips of $D_0$. Indeed, otherwise the component of $Q_1$ it meets is branching in $D_0+A_1$ and hence is neither contracted by $\psi$ nor contained in $\Upsilon_1^0+\Delta_1\tip$, which contradicts \eqref{eq:lambda=3_1}. We have $A_1\cdot (T_1+T_3)>0$, because otherwise by \eqref{eq:lambda=3_1}, $T_1+T_3\subseteq \Delta_1\tip$, in which case $Q_1$ would not contract to a smooth point. We may therefore assume that $A_1$ meets the tip of $D_0$ contained in $T_1$. Also, $A_1$ does not meet the tip of $D_0$ contained in $T_2$, because otherwise $\psi$ contracts $T_2$ and hence touches $C_1$, contrary to Lemma \ref{lem:basic_bounds}(d). Now \eqref{eq:lambda=3_1} gives $T_2=0$ or $T_2=[2]$. 

Consider the case $A_1\cdot (Q_1+E_0)>1$. If $A_1$ meets $U$, $C_1$ or $E_0$ then by \eqref{eq:lambda=3_1}, $T_3=[2]$ and $T_1=[(2)_k]$ for some $k\geq 1$, so $Q_1$ does not contract to a smooth point, which is false. So the only possibility is that $A_1$ meets the tip of $D_0$ contained in $T_3$. Then by \eqref{eq:lambda=3_1}, $\psi(T_1+T_3)$ is a component of $\Upsilon_1^0$. Let $\theta$ be the composition of $\psi$ with the contraction of $D_1-E_1-\psi(U)$. Then $\theta(U)$ is a rational curve on $\P^2$ with one singular point, a normal crossing node, and it meets $\theta(E_0)$ only at the point $\theta(C_1)$. But the degree of $\theta(U)$ equals $3$ and $\theta(E_0)\cdot \theta(U)$ equals $\tau_1=2$ if $T_2=0$ and equals $2\tau_1=4$ if $T_2=[2]$; a contradiction.

We obtain $A_1\cdot Q_1=1$ and $A_1\cdot E_0=0$. By \eqref{eq:lambda=3_1}, $T_3=[2]$ and $T_2$ is either $0$ or $[2]$. By the contractibility of $Q_1$ to a smooth point we get $T_1=[(2)_k,3]$ for some $k\geq 0$ and $U^2=-2-\#T_2\in \{-2,-3\}$. Since $A_1\cdot Q_1=1$, we have $\psi_*T_1\nsubseteq \Upsilon_1^0$, hence by \eqref{eq:lambda=3_1}, $\psi$ contracts $T_1$. We may assume that $A_1\cdot Q_2>0$. Then $A_1\cdot Q_2=1$ and $A_1\cdot Q_j=0$ for $j\geq 3$. Perform the contractions within $\psi$ until $U$ is touched once, then contract $Q_3+\ldots+Q_c$ and denote the resulting morphism by $\theta$. If $T_2=0$ then $F=\theta(U+C_1)=[1,1]$ and by Lemma \ref{lem:Hurwitz}, $c<4$, which is false. Thus $T_2=[2]$, hence for $F=\theta(U+2C_1+T_2)$ we have $F\cdot E_0=2\tau_1=4$, in which case Lemma \ref{lem:Hurwitz} gives $\sum_{j=2}^c(r_F(q_j)-1)=3$. Then $r_F(q_j)\leq 2$ for some $j\geq 3$, say for $j=3$. We have $\mu(q_3)\leq r_F(q_3)$, so $q_3\in\bar E$ is semi-ordinary. We get $\ind_1^{(1)}=\frac{2}{3}+\frac{1}{2}+\frac{1}{2}=\frac{5}{3}$, $\ind_3^{(1)}\geq \frac{5}{6}$ and $\ind_4^{(1)}>\frac{1}{2}$ (cf.\ Lemma \ref{lem:log BMY}(a)), so by \eqref{eq:ind}, $p_2<1-\ind_2^{(1)}\leq 1$. Thus $p_2=0$, so by \eqref{eq:lambda}, $\lambda_1+\cdots+\lambda_c\leq 7-\delta\leq 6$. But since $A_1$ meets $Q_1$ and $Q_2$, the cusps $q_1$ and $q_2$ are not semi-ordinary, hence by Lemma \ref{lem:basic_bounds}(a),(e), $\lambda_1+\lambda_2\geq 6$; a contradiction.

\smallskip
Thus we proved that $Q_1$ is a chain. We write it as $Q_1=T_1+U+C_1+T_2$, where $T_1$ and $T_2$ are zero or twigs of $Q_1$ meeting $U$ and $C_1$, respectively. As before, we note that $A_1$ can meet $T_1$ and $T_2$ only in tips of $D_0$. If $A_1\cdot T_2=1$ then by \eqref{eq:lambda=3_1}, $\psi$ contracts $T_2$, hence touches $C_1$, which is false by Lemma \ref{lem:basic_bounds}(d). Thus $A_1\cdot T_2=0$. Then $T_2=[2]$ or $T_2=0$. But the second case is impossible, because $q_1$ is not semi-ordinary. By the contractibility of $Q_1$ to a smooth point it follows that $Q_1=[(2)_k,3,1,2]$ for some $k\geq 0$. 

Suppose that $A_1\cdot T_1=0$. Then $A_1$ meets $U$ and some component $T_4\subseteq Q_2\wedge \Delta_0$. Let $\theta$ be the contraction of $Q_3+\ldots+Q_c$ and let $F=\theta(C_1+U+2A_1+T_4)$. Then $F\cdot \theta(E_0)=\tau_1=2$, so by Lemma \ref{lem:Hurwitz}, $1+(\mu(q_3)-1)+(\mu(q_4)-1)\leq 2$; a contradiction. Thus $A_1$ meets the tip of $D_0$ contained in $T_1$. Suppose that it meets $U$ too. Let now $\theta$ be the contraction of $A_1+D_0-E_0-U$. Then $\theta(U)$ is a uninodal curve on $\P^2$ (hence of degree $3$) whose intersection with $\theta(E_0)$ is $4$; a contradiction. 

It follows that $A_1+Q_1=[1,(2)_k,3,1,2]$. Since $A_1$ does not meet $Q_j$ for $j\geq 3$, we have $\ind_j^{(1)}=\ind_j^{(0)}>\frac{1}{2}$ for $j\geq 3$ by Lemma \ref{lem:log BMY}(a), hence by \eqref{eq:ind}, $1+\ind_2^{(1)}+(c-2)\frac{1}{2}<4-p_2$, so $p_2\leq 1$. Suppose that $p_2=0$. By \eqref{eq:lambda}, $(\lambda_1,\lambda_2,\lambda_3,\lambda_4)=(3,1,1,1)$. Since $q_2$ is ordinary, $A_1$ does not meet $Q_2$, so it meets $E_0$. But since $U=[3]$, the contraction of $D_0-E_0+A_1-U$ maps $U$ onto a $0$-curve on $\P^2$; a contradiction.

Thus we are left with the case $p_2=1$. We have $\ind_1^{(1)}=1$, so by \eqref{eq:ind}, $\frac{5}{6}c_0\leq \ind_2^{(1)}+\cdots+\ind_c^{(1)}\leq 2$. It follows that $c_0\leq 2$. By \eqref{eq:lambda}, $\lambda_1+\cdots+\lambda_c\leq 10-\delta<10$, so $c_0=2$, $c=4$ and $\lambda_2\in\{3,4\}$. We get $\ind_2^{(1)}\leq \frac{1}{3}$, so $s_2=0$. By  Lemma \ref{lem:basic_bounds}(c), $\lambda_2=4$. By \eqref{eq:lambda} we have $(\lambda_1,\lambda_2,\lambda_3,\lambda_4)=(3,4,1,1)$. By Lemma \ref{lem:basic_bounds}(h), $\tau_2=2$, $\psi(C_2)$ is a $(-1)$-tip of $\bar Q_2$ and the only contribution to $\ind_2^{(1)}$ comes from the $(-3)$-tip in $D_1'$ which is the proper transform of $\psi(C_2)$. Since $A_1\cdot Q_2\leq 1$, it follows that $Q_2=[1,(2)_{m}]$ for some $m\geq 1$ and $Q_1+A_1+Q_2$ is a chain with $T_2$ and $C_2$ as tips. We have $\Exc \psi\wedge Q_1=T_1$, which is possible only if $m=1$. Let $\theta$ be the contraction of $Q_3+Q_4+A_1+T_1^0$, where $T_1^0$ is the $(-2)$-tip of $T_1$ met by $A_1$. Then $F=\theta(\tilde C_2)$ is a $0$-curve and $F\cdot \theta(E_0)=1$. This is impossible, because $\theta(E_0)$ is singular; a contradiction.
\end{proof}

\begin{lem}\label{lem:n=1_cases} Assume $n=1$.  Put $\Lambda=(\lambda_1,\ldots,\lambda_c)$. 
Then $c=4$ and one may order the cusps so that one of the following holds:
\begin{enumerate}
\item[(0.a)] $p_2=0$, $\delta=0$, $\Lambda=(4,1,1,1)$,  $\ind_1^{(1)}\leq \frac{3}{2}$,  
\item[(1.a)] $p_2=1$, $\delta\in \{0,\frac{1}{2}\}$, $\Lambda=(\lambda_1,1,1,1)$ with $\lambda_1\in\{4,5,6,7\}$,  $\ind_1^{(1)}\leq \frac{1}{2}$,  
\item[(1.b)] $p_2=1$, $\delta=0$, $\Lambda=(\lambda_1,2,1,1)$ with $\lambda_1\in\{4,5,6\}$,  $\ind_1^{(1)}\leq \frac{7}{30}$,  
\item[(1.c)] $p_2=1$, $\delta=0$, $\Lambda=(\lambda_1,3,1,1)$ with $\lambda_1\in\{4,5\}$,   $\ind_1^{(1)}\leq \frac{5}{12}$, $\ind_1^{(1)}+\ind_2^{(1)}\leq \frac{4}{3}$,  
\item[(1.d)] $p_2=1$, $\delta=0$, $\Lambda=(4,4,1,1)$,   $\ind_1^{(1)}+\ind_2^{(1)}\leq \frac{4}{3}$,
\item[(2.a)] $p_2=2$, $\delta=0$, $\Lambda=(\lambda_1,5,1,1)$ with $\lambda_1\in\{5,6\}$,   $\ind_1^{(1)}+\ind_2^{(1)}\leq \frac{1}{3}$,  
\item[(2.b)] $p_2=2$, $\delta=0$, $\Lambda=(5,5,2,1)$, $\ind_1^{(1)}+\ind_2^{(1)}\leq \frac{1}{15}$.   
\end{enumerate}
\end{lem}

\begin{proof} For a positive integer $k$ we denote by $\w_k$ the number of cusps of $\bar E$ for which $\lambda_j=k$. We put $\wg=\displaystyle \sum_{k\geq 5}\w_k$. We may assume that $\lambda_1\geq \lambda_2\geq\ldots\geq \lambda_c.$ Since $\wg+\w_4+\w_3+\w_2+\w_1=c$, \eqref{eq:lambda} gives
\begin{equation}\label{eq:main}
5\wg+4\w_4+3\w_3+2\w_2+\w_1\leq \sum_{k\geq 1} k\w_{k}\leq 3p_2+7-\delta.
\end{equation}
For a cusp with $\lambda_j=1, 2,3$ and $4$ we know (see Lemma \ref{lem:basic_bounds}(e),(h) and Proposition \ref{prop:lambda=3}) that $\ind^{(1)}_j$ is bounded from below by $\frac{5}{6}$, $\frac{11}{10}$, $\frac{11}{12}$ and $\frac{1}{3}$, respectively, so \eqref{eq:ind} gives
\begin{equation}\label{eq:main_ind}
\frac{5}{6}\w_1+\frac{11}{10}\w_2+\frac{11}{12}\w_3+\frac{1}{3}\w_4\leq \sum_{j=1}^c \ind_j^{(1)}\leq 4-p_2.
\end{equation}
Put $\omega_{1,2}=\omega_1+\omega_2$. Since $n=1$, by Proposition \ref{prop:lambda=3}, $\w_4+\wg\geq 1$. We have $p_2\leq 2$ by Lemma \ref{lem:basic_bounds}(f).

\noindent Consider the case $p_2=2$. By \eqref{eq:main_ind}, $\frac{5}{6}\w_{1,2}\leq \sum_{j=1}^c\ind_j^{(1)}\leq 2$, so $\w_{1,2}\leq 2$. If $\omega_{1,2}=0$ then by \eqref{eq:main_ind}, $\w_3\leq 2$, so $\w_4+\wg\geq c-\w_3\geq 2$, in which case \eqref{eq:main} fails. Thus $\omega_{1,2}\in \{1,2\}$. Assume $\omega_{1,2}=2$. Then $\frac{11}{12}\w_3+\frac{1}{3}\w_4\leq \frac{1}{3}$, so $\w_3=0$ and $\w_4\leq 1$. Suppose that $\w_4=1$. By \eqref{eq:main_ind}, $\w_2=0$, so $\w_1=2$. Then by \eqref{eq:main}, $c=4$ and $\Lambda=(\lambda_1,4,1,1)$ for some $\lambda_1\in\{5,6,7\}$. By \eqref{eq:main_ind}, $\ind_1^{(1)}+\ind_2^{(1)}\leq \frac{1}{3}$ and by Lemma \ref{lem:basic_bounds}(h) the latter number is a contribution from a single twig. This is impossible, as $D$ has at least four twigs in total over $q_1$ and $q_2$ and the proper transform of $A_1$ meets at most two of them; a contradiction. Thus $\w_4=0$. By \eqref{eq:main}, $c=4$ and $\wg=\w_5+\w_6\leq 2$, which leads to cases (2.a) and (2.b). Note that the equality $\delta=0$ is a consequence of the inequality $\ind_1^{(1)}+\ind_2^{(1)}\leq \frac{1}{3}$. Assume $\omega_{1,2}=1$. By \eqref{eq:main_ind}, $\frac{11}{12}\w_3+\frac{1}{3}\w_4\leq \frac{7}{6}$, so $\w_3\leq 1$. If $\w_3=1$ we get $\w_4=0$, hence $\sum_{k\geq 1} k\w_{k}\geq 14$, in contradiction to \eqref{eq:main}. Hence $\w_3=0$ and $\wg+\w_4\geq c-1\geq 3$. Then $5\wg+4\w_4\leq 12-\delta$, so $\Lambda=(4,4,4,1)$. Since $A_1\cdot D_0=2$, we may assume that $A_1\cdot Q_3=0$, so by Lemma \ref{lem:log BMY}(a), $\sum_{j=1}^c\ind_j^{(1)}> \frac{5}{6}+2\cdot \frac{1}{3}+\frac{1}{2}=2$; a contradiction.

\smallskip Consider the case $p_2=1$. If $\omega_{1,2}\leq 1$ then by \eqref{eq:main}, $\w_4+\wg+10\leq 1+3(c-1)+\w_4+2\wg\leq 10-\delta$, which is impossible, as $\w_4+\wg\geq 1$ by Proposition \ref{prop:lambda=3}. By \eqref{eq:main_ind}, $\w_{1,2}\in \{2,3\}$. Assume $\omega_{1,2}=2$. By \eqref{eq:main}, $3(c-2)+2\wg+\w_4=5\wg+4\w_4+3\w_3\leq 8-\delta-\w_2$. For $\w_2\neq 0$ we get $c=4$ and $\Lambda=(4,3,2,1)$, in which case \eqref{eq:main_ind} fails. Thus $\w_2=0$ and $\w_1=2$, which gives cases (1.c) and (1.d). Assume $\omega_{1,2}=3$. By \eqref{eq:main_ind}, $\w_3=0$. By \eqref{eq:main}, $4(c-3)+3\leq 5\wg+4\w_4+\w_2+3\leq 10-\delta$, so $c=4$. Now \eqref{eq:main_ind} gives $\w_1\in \{2,3\}$. This gives cases (1.a) and (1.b).

\smallskip
Finally, consider the case $p_2=0$. We have now $\sum_{j=1}^c\ind_j^{(1)}\leq 4$ and $\sum_{j=1}^c \lambda_j\leq 7-\delta$. Since $\lambda_1\geq 4$, we get $\Lambda=(4,1,1,1)$ and $\delta=0$, which is case (0.a).

\smallskip In all cases other than case (1.a) and case (1.c) for $\lambda_1=4$ we have either $3p_2+7-\lambda_1-\ldots-\lambda_c=0$, hence  $3p_2+7-\lambda_1-\ldots-\lambda_c-\delta=\delta=0$ or the sum of $\ind_j^{(1)}$ for non-semi-ordinary cusps is less than $\frac{1}{2}$, which gives $\Delta_1^-=0$ and hence $\delta=0$. In case (1.a) we have $\ind_1^{(1)}\leq \frac{1}{2}$, hence $\Delta_1^-$ is $0$ or $[2]$, which gives $\delta\in\{0,\frac{1}{2}\}$. Finally, in case (1.c) for $\lambda_1=4$ we have $\ind_2^{(1)}\leq \frac{4}{3}-\frac{1}{3}=1$ by Lemma \ref{lem:basic_bounds}(h), so $\delta=0$ by Proposition \ref{prop:lambda=3}.
\end{proof}

Our goal is to eliminate all cases listed in Lemma \ref{lem:n=1_cases}. We collect some observations. Recall that $K_1=K_{X_1}$ and $\bar Q_j=\psi(Q_j)$.

\begin{lem}\label{lem:elements}
Assume that $n=1$ and $j\in\{1,\ldots,4\}$. Put $|\Lambda|=\lambda_1+\ldots+\lambda_4$ and $\gamma_i=-E_i^2$. We have: 
\begin{enumerate}[(a)]
\item $K_1\cdot D_1=p_2+\#D_1-\tau^*-16$,
\item $\#D_1=|\Lambda|+b_0(\Delta_1^-)+\#\Upsilon_1-\tau^*-3$,
\item  $|\Lambda|\leq 3p_2+7-\delta$, 
\item If $U$ is a component of $\bar Q_j-\psi(C_j)$ then $U^2<0$.
\item $K_1\cdot \psi(C_j)+\tau_j^*\geq -1$ and the equality holds if and only if $\tau_j^*=0$ and $\psi(C_j)^2=-1$.
\item Let $r_j$ be the number of outer blow-ups (see Section \ref{sec :5A}) over $q_j\in \bar E$ in the minimal log resolution of $(\P^2,\bar E)$. Then $$\gamma_0+\tau^*=p_2+1-s+\sum_{j=1}^4 (r_j-1).$$
\end{enumerate}
\end{lem}

\begin{proof} (a) By Noether's formula $K_n^2=10-\rho(X_n)=10+n-\#D_n$, so \cite[Lemma 4.3]{Palka-minimal_models} gives $K_n\cdot D_n=p_2-c-\tau^*-n-K_n^2=p_2-c-\tau^*-2n-10+\#D_n$. Since $c=4$ by Lemma \ref{lem:n=1_cases}, we get $K_1\cdot D_1=p_2+\#D_1-\tau^*-16$.

(b) The definition of $\lambda_j$ gives $|\Lambda|=\#(D_n-E_n)-b_0(\Delta_n)-\#\Upsilon_n^0+\tau^*+c$, hence $|\Lambda|+b_0(\Delta_1^-)+\#\Upsilon_1=|\Lambda|+b_0(\Delta_1)+\#\Upsilon_1^0=\#D_1+\tau^*+3$.

(c) This is \eqref{eq:lambda}.  

(d) Suppose that $U^2\geq 0$. By blowing up over $U\setminus E_1$ we may assume that $U^2=0$. Put $d:=U\cdot E_1$. We may assume that $U':=\psi_*^{-1}(U)\subseteq Q_1$ and $A_1\cdot (Q_3+Q_4)=0$. Contract $Q_3+Q_4$ and for $j=1,2$ denote by $m_j$ be the image of $C_j\cap E_0$. By Lemma \ref{lem:Hurwitz} applied to the image of $E_1$, $0\leq r_U(m_1)+r_U(m_2)-2\leq 2d-4$. In particular, $d\geq 2$. The curve $U'$ is a component of $Q_1-C_1$, so $U'\cdot E_0\leq 1$. Since $d>1$, Lemma \ref{lem:psi-properties}(a) gives $U'\cdot E_0=1$ and $U\cdot E_1=2$. It follows that $U'=\tilde C_1$ and $\Exc(\psi)\cdot Q_2=0$. Then $U\cdot \bar Q_2=0$ and so $r_U(m_2)\geq 2$; a contradiction.

(e) We may assume that $\psi(C_j)^2\geq 0$, otherwise the claim is clear. We may also assume that $j=1$ and, since $n=1$, that $A_1\cdot (Q_3+Q_4)=0$. It follows that $\psi$ touches $C_1$ and, by Lemma \ref{lem:basic_bounds}(d), that $\tau_1\geq 3$. Suppose that $K_1\cdot \psi(C_1)+\tau_1^*\leq -1$. We have $K_1\cdot \psi(C_1)+\tau_1^*=\tau_1-\psi(C_1)^2-s_1-3$, so $\tau_1-\psi(C_1)^2\leq s_1+2$. Blow up over $\psi(C_1\cap E_0)$, each time on the intersection of the proper transforms of $\psi(C_1)$ and $\psi(E_0)$, until the proper transform $C'$ of $\psi(C_1)$ is a $0$-curve. Denote the proper transform of $E_1$ by $E'$. Over $\psi(C_1\cap E_0)$ the intersection of $C'$ and $E'$ is $\tau_1'=\tau_1-\psi(C_1)^2$, hence $\tau_1'\leq s_1+2$. 

Assume first that $C'\cdot E'=\tau_1'$. Lemma \ref{lem:Hurwitz} gives $\sum_{j=2}^4(r_{C'}(q_j)-1)\leq \tau_1'-1\leq s_1+1$. Since $\psi$ does not touch $Q_3+Q_4$, the left hand side is at least $2$, so we obtain $s_1=1$, $\tau_1'=3$, $r_{C'}(q_j)=2$ for $j=3,4$ and $r_{C'}(q_2)=1$. It follows that $q_3$ and $q_4$ are semi-ordinary and $\psi$ touches $C_2$. In particular, $\Upsilon_1^0=\psi_*\Upsilon_0^0$ and, by Lemma \ref{lem:basic_bounds}(d), $\tau_2\geq 3$. By Lemma \ref{lem:basic_bounds}(e) and Proposition \ref{prop:lambda=3}, $\lambda_1, \lambda_2\geq 4$. Since $s_1=1$ and $\tau_1'=3$, we have $\ind_1^{(1)}\geq \frac{2}{3}$. Then we are in case (1.d) of Lemma \ref{lem:n=1_cases}, hence $\lambda_1=\lambda_2=4$ and $\delta=0$, so $\Delta_1^-=0$. We have $s_2=1$, because otherwise $\lambda_2\geq 2+\tau_2\geq 5$, which is false. The inequality $\ind_1^{(1)}+\ind_2^{(1)}\leq \frac{4}{3}$ implies that $Q_1+Q_2+E_1+A_1$ has no tips and $\tau_1=\tau_2=3$, which implies that for $j=1,2$, $\#\Omega_j=\lambda_j-2=2$  and $Q_j- C_j$ is a $(-2)$-chain met by $A_1$, one of them necessarily empty. But then one of the $Q_j$ for $j=1,2$, is irreducible, which is impossible.

We are left with the case $C'\cdot E'>\tau_1'$. This is possible only when $A_1\cdot E_0=1$ and $\psi$ contracts $A_1$ and a $(-2)$-twig of $D_0$ meeting $C_1$. Then $\psi(C_1)^2=0$, $\tau_1=\tau_1'$, $\psi$ does not touch $Q_2$ and we have $C'\cdot E'=\tau_1'+1$. Moreover, $(\Delta_1^++\Upsilon_1)\wedge \bar Q_j=0$, unless $q_j\in \bar E$ is semi-ordinary. Since $s_1+2\geq \tau_1'=\tau_1\geq 3$ by Lemma \ref{lem:basic_bounds}(d), we have $s_1=1$ and $\tau_1=3$. Then $\ind_1^{(1)}\geq \frac{2}{3}$, so we are in cases (0.a) or (1.d) of Lemma \ref{lem:n=1_cases}, hence $\lambda_1=4$ and $\Delta_1^-=0$. We obtain $\#\bar Q_1=\#\Omega_1=2$. Then $Q_1=[(2)_k,1,k+2]$ for some $k\geq 1$, so $\ind_1^{(1)}>\frac{2}{3}$. In case (1.d) we get $s_2=0$, so since $\lambda_2=4$, we have $Q_2=C_2+\tilde C_2$ and then the inequality $\ind_2^{(1)}<\frac{2}{3}$ fails. Thus we are in case (0.a). We denote by $\theta\:X_1\to \theta(X_1)$ the contraction of $\psi(Q_2+Q_3+Q_4)$. We have $\rho(\theta(X_1))=2$, so $\theta(X_1)=\F_{k+2}$, $F=\theta(\psi(C_1))$ is a fiber and $U=\theta_*(\bar Q_1)-F$ is the negative section. Since $\theta(E_1)$ has three cusps, all ordinary, $p_a(\theta(E_1))=3$. On the other hand, on $\F_{k+2}$ we have $K_{\F_{k+2}}\sim-2 U-(k+4)F$ and $\theta(E_1)\sim 4 U+4(k+2)F$, so $p_a(\theta(E_1))=6k+9$; a contradiction.

(f) Recall that $Q_j'$ is the reduced exceptional divisor of the minimal log resolution of $(\P^2,\bar E)$ over $q_j$. By \cite[Lemma 2.4(i)]{KoPa-CooligeNagata2}, $K\cdot Q_j'+1$ equals the number of inner blow-ups over $q_j$, so
\begin{equation}\label{eq:KQ_inner_bl}
K\cdot Q_j'=\#Q_j'-r_j-2,
\end{equation}
which gives $K\cdot D=-E^2-3+\#D-\sum_{j=1}^cr_j-2c$. Since $\#D=\rho(X)=10-K^2$, we obtain 
\begin{equation}\label{eq:p2_formula}
p_2=K\cdot (K+D)=-E^2+7-2c-\sum_{j=1}^cr_j.
\end{equation}
Since $E^2=E_0^2-\tau$, we have $-E^2=-E_0^2+\tau^*+c+s$, hence $\gamma_0+\tau^*=p_2+2c-7-s+\sum_{j=1}^c(r_j-1)$.
\end{proof}

We are now ready to make the key step in the proof of Theorem \ref{thm:at_most_4_cusps}.
 
\begin{tw}
\label{thm:almost_minimal}
The surface $(X_0,\frac{1}{2}D_0)$ is almost minimal.
\end{tw}

\begin{proof} By Lemma \ref{lem:n<=1}, $n\leq 1$. Suppose that $n=1$. We assume that $\lambda_1\geq \lambda_2\geq \ldots \geq \lambda_c$. We need to rule out the cases listed in Lemma \ref{lem:n=1_cases}. For all of them $c=4$, $q_4$ is ordinary and $q_3$ is either ordinary or semi-ordinary with $\lambda_3=2$. Let $\cL\subseteq D_1-E_1$ be the sum of $(-1)$-curves created by $\psi$, i.e.\ of $(-1)$-curves in $D_1-E_1$ whose proper transforms on $X_0$ are not $(-1)$-curves. Put $\eta=\#(\cL\wedge \Upsilon_1)$.  By the definition of $\psi$ we have $\eta\leq \#\cL\leq 2$.  Our goal is to analyze the divisor $$R=\bar Q_1+\bar Q_2-\psi(C_1)-\psi(C_2)-\cL.$$ Recall that $c_0$ denotes the number of semi-ordinary cusps, which are cusps of multiplicity $2$, and that $\gamma_i=-E_i^2$. By Lemma \ref{lem:elements}(a),(b), $K_1\cdot D_1=p_2+\#D_1-\tau^*-16$ and $\#D_1=|\Lambda|+b_0(\Delta_1^-)+\#\Upsilon_1-\tau^*-3$, hence $$K_1\cdot E_1+\sum_{j=1}^4K_1\cdot \psi(C_j)+K_1\cdot \cL+K_1\cdot R=
K_1\cdot D_1=p_2+|\Lambda|+b_0(\Delta_1^-)+\#\Upsilon_1-2\tau^*-19.$$ We have $K_1\cdot E_1=\gamma_1-2$ and $\#\Upsilon_1=\eta+c_0$, so putting $\alpha_j=K_1\cdot \psi(C_j)+\tau_j^*$ we obtain
\begin{equation}\label{eq:KR}
(\gamma_1+\tau^*+4-2c_0)+\sum_{j:\mu(q_j)\geq 3}\alpha_j+K_1\cdot R=p_2+|\Lambda|+b_0(\Delta_1^-)+\eta+\#\cL-13.
\end{equation}
The proof below is based on analyzing summands of both sides of this equality. Note that by Lemma \ref{lem:elements}(d),(e) we have $\alpha_j\geq -1$ and that $K_1\cdot U\geq 0$ for all components $U\subseteq R$. 

\smallskip

\begin{claim}$\kappa(K_0+\frac{1}{2}D_0)\geq 0$. \end{claim}

\begin{proof} Suppose that $\kappa(K_0+\frac{1}{2}D_0)=-\8$. In particular, $p_2=0$, hence we are in case (0.a) of Lemma \ref{lem:n=1_cases}. Then in the inequality \eqref{eq:lambda} we have in fact equality, which is equivalent to the equality $(2K_1+D_1^\flat)^2=0$. By Proposition \ref{lem:basics}(c), $\P^2\setminus \bar E$ has no $\C^{**}$-fibration, hence by Corollary \ref{cor:almost_min_models}(b), $2K_1+D_1^\flat$ is a pull-back of an anti-ample divisor. But then $(2K_1+D_1^\flat)^2>0$; a contradiction.
\end{proof}

Put $\epsilon=A_1\cdot E_0-E_1\cdot (\cL\wedge \Upsilon_1)$. We have $\epsilon\in\{0,1\}$. Moreover, $\epsilon=1$ if and only if $A_1\cdot E_0=1$ and $A_1\cdot \Delta_0\tip=1$. In the latter case we have in fact $\cL\wedge \Upsilon_1=0$ (otherwise $\psi_*^{-1}\cL$ would meet two $(-2)$-twigs, which is impossible) and $\lambda_j=\#(\bar Q_j-\bar Q_j\wedge\Delta_1\tip)+\tau_j^*+1$ for $j\in \{1,2\}$ such that $A_1\cdot Q_j=1$.

\begin{claim} $\gamma_1+\tau^*+4-2c_0\geq -\epsilon$. \end{claim}

\begin{proof} By Claim 1, $\kappa(K_0+\frac{1}{2}D_0)\geq 0$, hence by Lemma \ref{cor:almost_min_models}(b) the divisor $$2K_1+D_1^\flat=\displaystyle 2K_1+E_1+\sum_{j:\mu(q_j)\geq 3}(\bar Q_j-\bar Q_j\wedge (\Upsilon_1+\Delta_1^+))-\Bk'\Delta_1^-$$ is nef. Since $E_1\cdot \Delta_1=0$, intersecting the above divisor with $E_1$ we obtain $$\displaystyle 4-\gamma_1\leq \sum_{j:\mu(q_j)\geq 3}E_1\cdot\bar Q_j-E_1\cdot (\cL\wedge \Upsilon_1)=A_1\cdot E_0+\sum_{j:\mu(q_j)\geq 3}E_0\cdot Q_j-E_1\cdot (\cL\wedge \Upsilon_1)=\epsilon+\sum_{j:\mu(q_j)\geq 3}(\tau_j^*+2),$$ which proves the claim.
\end{proof}

\begin{claim} $p_2\neq 0$. \end{claim}

\begin{proof} Suppose that $p_2=0$. By Lemma \ref{lem:n=1_cases} we are in case (0.a), hence $\Delta_1^-=0$. Claim 2 and the equality \eqref{eq:KR} give $-1-\epsilon\leq \alpha_1-\epsilon\leq \eta+\#\cL-6$, hence $2\#\cL\geq \eta+\#\cL\geq 5-\epsilon\geq 4$. From the definition of $\psi$ it follows that $\cL$ has at most $2$ components, hence it has exactly two and $\epsilon=1$. But if $\epsilon=1$ then $A_1\cdot E_0=1$, so $\#\cL\leq 1$, because by definition $\cL\wedge E_1=0$; a contradiction.
\end{proof}

\begin{claim} $p_2=1$. \end{claim}

\begin{proof} Suppose that $p_2\neq 1$. By Claim 3 and by Lemma \ref{lem:basic_bounds}(f), $p_2=2$, so we are in case (2.a) or (2.b) of Lemma \ref{lem:n=1_cases}. In particular, $\ind_1^{(1)}+\ind_2^{(1)}\leq \frac{1}{3}$, hence $\Delta_1\wedge (\bar Q_1+\bar Q_2)=0$, $s_1=s_2=0$ and, by Lemma \ref{lem:log BMY}(a), $A_1$ meets both $Q_1$ and $Q_2$. The latter gives $c_0=2$ and implies that $\Upsilon_1\wedge (\bar Q_1+\bar Q_2)=0$ and $A_1\cdot E_0=0$. It follows that $\Upsilon_1=\psi(C_3)+\psi(C_4)$, $\eta=0$ and $\epsilon=0$. We obtain $\#\bar Q_2+\tau_2=\lambda_2=5$ and $\#\bar Q_1+\tau_1=\lambda_1\in \{5,6\}$. The equality \eqref{eq:KR} reads as
\begin{equation}\label{eq:4.1}
(\gamma_1+\tau^*-\lambda_3)+\alpha_1+\alpha_2+K_1\cdot R=\lambda_1-5+\#\cL\leq 1+\#\cL.
\end{equation}
For $j=1,2$ we have $s_j=0$, hence $\tau_j^*\geq 1$, which gives $\alpha_j\geq 0$ by Lemma \ref{lem:elements}(e). By Lemma \ref{lem:elements}(f) we have $\gamma_0+\tau^*=r_1+r_2+r_3-2$. Since $A_1\cdot E_0=0$, we have $\gamma_0=\gamma_1$. For a semi-ordinary cusp $r_j=\#Q_j=\lambda_j$, hence  $\gamma_1+\tau^*-\lambda_3= r_1+r_2-2\geq 0$. It follows that $K_1\cdot R\leq 1+\#\cL\leq 3$, hence every component $U$ of $\bar Q_1+\bar Q_2$ satisfies $U^2\geq -5$ and the equality may hold for at most one such component. For $j=1,2$ we have $\tau_j=\lambda_j-\#\bar Q_j\leq \lambda_j-2=4$, hence $((\varphi_1^{-1})_*\psi_*C_j)^2\geq -1-\tau_j\geq -5$. It follows that every component $U'$ of $Q_1'+Q_2'$ satisfies $(U')^2\geq -6$ and the equality may hold for at most one such component. For $j=1,2$ let now $U_j$ be a tip of $D_1'$ lying in $\varphi_1^{-1}(\bar Q_j)$. Such a tip exists, because there are at least two tips of $D$ over $q_j$ and, since $A_1\cdot Q_j=1$, (the proper transform of) $\Exc \psi$ meets at most one of them. From Lemma \ref{lem:elements}(d) we infer that $U_j^2\leq -2$. We obtain $\ind_1^{(1)}+\ind_2^{(1)}\geq \frac{1}{-U_1^2}+\frac{1}{-U_2^2}\geq \frac{1}{5}+\frac{1}{6}>\frac{1}{3}$; a contradiction with Lemma \ref{lem:n=1_cases}.
\end{proof}

By Claim 4 and Lemma \ref{lem:n=1_cases} the cusps $q_3$ and $q_4$ are ordinary and either $\lambda_2\leq 3$ or $\lambda_1=\lambda_2=4$. By Proposition \ref{prop:lambda=3}, $\psi$ does not touch $Q_j$ with $\lambda_j\leq 3$, hence we may, and will, assume that $A_1\cdot Q_1\neq 0$. Put $\displaystyle \alpha=\sum_{j:\mu(q_j)\geq 3}\alpha_j$. The equality \eqref{eq:KR} reads as
\begin{equation}\label{eq:KR2}
(\gamma_1+\tau^*+4-2c_0)+\alpha+K_1\cdot R=\lambda_1+\lambda_2-10+b_0(\Delta_1^-)+\eta+\#\cL.
\end{equation}

\begin{claim} $\#\cL\leq 1$. \end{claim}

\begin{proof} Suppose that $\#\cL>1$. By Lemma \ref{lem:psi-properties}(a), $\#\cL$ is connected and has exactly two components, they meet normally and $A_1\cdot E_0=0$. We have $\cL\cdot E_1=\psi_*^{-1}\cL\cdot E_0\leq 2$. Consider the case when the components of $\cL$ have intersection number bigger than $1$. By Lemma  \ref{lem:psi-properties}(a) one of them is the image of a component of some twig of $D_0$, so it does not meet $E_0$, hence $\cL\cdot E_1\leq 1$. By Lemma \ref{lem:ell_fib} after contracting  $Q_3+Q_4$ and blowing up three times over $\cL$ we obtain an elliptic fibration for which the proper transform of $E_1$ has at least two cusps and meets a general fiber at most once. But this is impossible, hence $\cL=[1,1]$. After the contraction of $Q_3+Q_4$ the linear system of $\cL$ induces a $\P^1$-fibration of a smooth surface on which the image of $E_1$ is singular (hence horizontal) and meets a general fiber at most twice. In fact, being singular, it meets a general fiber exactly twice, hence $\psi_*^{-1}\cL\cdot E_0=2$, which gives $\psi_*^{-1}\cL=\tilde C_1+\tilde C_2$. Thus $\eta=0$, $s_1=s_2=0$ and $\Exc \psi-A_1$ consists of twigs meeting $\tilde C_1$ and $\tilde C_2$. It follows from Lemma \ref{lem:elements}(e) that $\alpha_j\geq 0$ for $j=1,2$. Moreover, $\psi$ touches $Q_2$ and does not touch $E_0$, hence $\epsilon=0$ and, by Lemma \ref{prop:lambda=3}, $\lambda_2\geq 4$. Then we are in case (1.d) of Lemma \ref{lem:n=1_cases}, hence \eqref{eq:KR2} and Claim 2 give $K_1\cdot R\leq b_0(\Delta_1^-)$. Since $\Delta_1^-=0$ in case (1.d), we infer that $R=\bar Q_1+\bar Q_2-\psi(C_1+C_2+\tilde C_1+\tilde C_2)$ consists of $(-2)$-curves and has no tips. It follows that $R=0$, hence $\tau_j=\lambda_j-2=2$ for $j=1,2$. Let $\theta\:X_1\to \theta(X_1)$ be the contraction of $\psi(C_1+C_2+C_3+C_4)$. Then $\theta(X_1)\cong \P^1\times \P^1$ and $\theta(E_1)\sim 2f_1+2f_2$, where $f_1, f_2$ are fibers of two projections of $\theta(X_1)$ onto  $\P^1$. We compute $\theta(E_1)^2=8$ and hence $\gamma_1=-E_1^2= -\theta(E_1)^2+\sum_{j=1}^4\tau_j^2= -8+16=8$. Then \eqref{eq:KR2} fails; a contradiction.
\end{proof}

\begin{claim} $\Delta_1^-=0$. \end{claim}

\begin{proof} Suppose that $\Delta_1^-\neq 0$. Then we are in case (1.a) of Lemma \ref{lem:n=1_cases} with $\lambda_1\leq 6$, $\Delta_1^-=[2]$ and $\Delta_1^+=0$. Let $E_1'\subseteq D_1'$, $U_1\subseteq D_1'$ and $\bar Q_1'\subseteq D_1'$ be, respectively, the proper transforms of $E_1$, $\Delta_1$ and the reduced total transform of $\bar Q_1$ under the minimal log resolution $\varphi_1\:(X_1',D_1')\to (X_1,D_1)$. Put $A_1':=(\varphi_0)_*^{-1}A_1$. Clearly, $U_1=[2]$  and $U_1$ is a twig of $D_1'$ contained in $\bar Q_1'$. Since $\ind_1^{(1)}\leq \frac{1}{2}$, we infer that $s_1=0$ and that $U_1$ is in fact a unique maximal twig of $D_1'$ contained in $\bar Q_1'$. It follows that the proper transform of $U_1$ on $X$, call it $U_0$, is the unique maximal twig of $D+A_1'$ contained in $Q_1'$ and hence $Q_1'$ has at most one branching  component. Moreover, if $\epsilon=1$ then $Q_1'$ is a chain and $U_0$ is a maximal twig of $D$, so $s_1=1$, which is false. Thus $\epsilon=0$ and so \eqref{eq:KR2} gives $\alpha_1+K_1\cdot R\leq \lambda_1-8+\eta+\#\cL\leq \eta+\#\cL-2$. Since $\tau_1^*\geq 1$, we have $\alpha_1\geq 0$, hence $2\leq \eta+\#\cL\leq 2\#\cL$. By Claim 5 we get $\eta=\#\cL=1$. Since $\Delta_1^+=0$, we have $\cL\subseteq \Upsilon_1^0$. Denote by $C_1'$ the unique $(-1)$-curve in $Q_1'$. The fact that $\psi$ creates a new component of $\Upsilon_1^0$ implies that both components of $Q_1'$ met by $A_1'$ are contained in the same connected component of $Q_1'-C_1'$, hence $U_0$ meets $C_1'$. But then again $s_1=1$; a contradiction.
\end{proof}

\begin{claim} $\alpha\geq -1$. \end{claim}

\begin{proof} Suppose otherwise. Let $j\in \{1,2\}$. By assumption the cusp $q_j\in \bar E$ is not semi-ordinary and  $\alpha_j=-1$. By Lemma \ref{lem:elements}(e), $\tau_j^*=0$ and $\psi$ does not touch $C_j$. We have $s_j=1$, so $\ind_j^{(1)}\geq \frac{1}{2}$, hence by Claim 4 we are in case (1.d) of Lemma \ref{lem:n=1_cases}. In particular, the contribution to $\ind(D_1')$ coming from twigs of $D_1'$ contained in $\bar Q_j'$ and not contracted by $\varphi_1$ is $\ind':=\ind_1^{(1)}+\ind_2^{(1)}-2\cdot \frac{1}{2}\leq \frac{1}{3}$. We obtain $\Delta_1^+\wedge \bar Q_j=0$, hence $(\Upsilon_1-\Upsilon_1^0)\wedge \bar Q_j$. Since $q_j\in \bar E$, $j\in \{1,2\}$ is not semi-ordinary and $\tau_j^*=0$, the divisor $Q_j'$ is not a chain. From Lemma \ref{lem:log BMY}(a) we infer that $A_1$ meets $Q_j$. It follows that $\epsilon=0$ and $\eta=0$, hence \eqref{eq:KR2} gives $\gamma_1+K_1\cdot R=\#\cL\leq 1$ by Claim 5. Let $T_j=\Exc \psi\wedge Q_j$. We have $\#\bar Q_j=\#\Omega_j=\lambda_j+1-\tau_j=3$, so $\bar Q_j$ is a chain. We observe that $\bar Q_j$ contains some tip $U_j$ of $D_1$. Indeed, otherwise $\psi(C_j)$ is a tip of $\bar Q_j$ and $\psi(T_j)$ belongs to the second tip of $\bar Q_j$. But then $Q_j$ is a chain with $C_j$ as a tip, hence it is a chain $[1,(2)_{k_j}]$ for some $k_j\geq 0$, that is, $q_j\in \bar E$ is semi-ordinary, which is false. Since $\epsilon=0$, by Claim 2, $\gamma_1\geq 0$, so we obtain $K_1\cdot (U_1+U_2)\leq K_1\cdot R\leq \#\cL\leq 1$. It follows that $U_1$ or $U_2$ is a $(-2)$-tip, so $\ind'\geq \frac{1}{2}$; a contradiction.
\end{proof}

\begin{claim} $\eta=1$. \end{claim}

\begin{proof} Suppose first that $\#\cL=0$. By Lemma \ref{lem:n=1_cases}, $\lambda_1+\lambda_2\leq 8$. By \eqref{eq:KR2} and by previous claims we have $-1\leq \alpha+K_1\cdot R\leq \epsilon+\lambda_1+\lambda_2-10\leq \epsilon-2$. It follows that $\epsilon=1$, $R$ consists of $(-2)$-curves and there is some $j\in \{1,2\}$ for which $q_j$ is not semi-ordinary and $\alpha_j=-1$, hence $s_j=1$, $\tau_j=2$ and $\psi$ does not touch $C_j$. It follows that $A_1$ meets $Q_j$. Since $\epsilon=1$, we have $A_1\cdot E_0=1$, so $A_1\cdot Q_1=1$. Then the equality $\epsilon=1$ implies that $T:=\Exc \psi-A_1$ is a $(-2)$-twig of $Q_1$, so $Q_1-C_1$ consists of $(-2)$-curves and a unique $(-3)$-curve meeting $T$. Since $\Delta_1^-=0$, the divisor $Q_1+A_1+E_0$ has no tips, hence $Q_1$ is a chain and $C_1$ is its tip. Then $Q_1$ does not contract to a smooth point; a contradiction. Thus $\#\cL=1$. The equality \eqref{eq:KR2} gives
\begin{equation}\label{eq:KR3}
-1+K_1\cdot R\leq \alpha+K_1\cdot R\leq \epsilon+\lambda_1+\lambda_2-9+\eta\leq \epsilon+\eta-1.
\end{equation}

Suppose that $\eta=0$ and $\epsilon=1$. By Claim 6, $\Delta_1\wedge \bar Q_1=0$. The divisor $T=\Exc \psi-A_1$ is a $(-2)$-twig of $Q_1$ and $\psi$ does not touch $Q_2$. Since $\#\cL=1$, the component $B\subseteq Q_1-T$ meeting $T$ is a $(-2)$-curve. We obtain $$K_0\cdot (Q_1-C_1)=K_0\cdot (Q_1-C_1-T-B)=K_1\cdot (\bar Q_1\wedge R)\leq K_1\cdot R\leq 1,$$ so $Q_1-C_1$ consists of $(-2)$-curves and at most one $(-3)$-curve. If $B=\tilde C_1$ then, $C+B+T=[1,(2)_k]$ for some $k\geq 2$, so since $Q_1$ contains no curves with self-intersection number smaller than $-3$, $Q_1=C+B+T$, hence $\bar Q_1=[1,1]$. But the latter is impossible, as the contraction of $\psi(C_1)+\bar Q_2+\bar Q_3+\bar Q_4$ would map $\psi(\tilde C_1)$ onto a $0$-curve on $\P^2$. Thus $B\neq \tilde C_1$. Since $B$ is a $(-2)$-curve not contracted by $\psi$, it is necessarily a branching component of $Q_1$. We have $\Delta_1\wedge \bar Q_1=0$, so the maximal twigs of $Q_1$ are $T=[2]$, $[3]$ and $[1,(2)_k]$ for some $k\geq 0$. In particular, $K_1\cdot R=1$, which gives $\alpha=-1$. We have $\alpha_2\geq 0$, because otherwise $\mu(q_2)>2$ and $\tau_2^*=0$, which is impossible, as $Q_2-C_2$ consists of $(-2)$-curves. It follows that $\alpha_1=-1$, so $\tau_1^*=0$. Let $U$ denote the $(-3)$-tip of $\bar Q_1$. After the contraction of $\bar Q_3+\bar Q_4$ the linear system of $\bar Q_1-U=[1,2,\ldots,2,1]$ induces a $\P^1$-fibration for which Lemma \ref{lem:Hurwitz} gives $\tau_1-1+\mu(q_2)-1+2\leq 2\cdot (1+\tau_1)-2$, which implies that $q_2$ is semi-ordinary. But then $\ind_1^{(1)}+\ind_2^{(1)}\geq \frac{1}{3}+\frac{1}{2}+\frac{5}{6}=\frac{5}{3}$, which contradicts Lemma \ref{lem:n=1_cases}. 

Suppose that $\eta=0$ and $\epsilon=0$. Now \eqref{eq:KR3} reads as
$\alpha+K_1\cdot R\leq \lambda_1+\lambda_2-9\leq -1$, hence $R$ consists of $(-2)$-curves, $\lambda_1+\lambda_2=8$, and $\alpha=-1$. Since there is $j\in \{1,2\}$ with $\alpha_j=-1$ which is not semi-ordinary, hence with $s_j=1$, we have $\ind_j^{(1)}\geq \frac{1}{2}$, so by Lemma \ref{lem:n=1_cases}, $(\lambda_1,\lambda_2)=(7,1)$ or $\lambda_1=\lambda_2=4$.  We have $A_1\cdot Q_1\geq 1$. Suppose that $A_1$ meets $Q_2$. Then $\lambda_1=\lambda_2=4$ and, say, $j=1$, so $\tau_1^*=0$ by Lemma \ref{lem:elements}(e). But since $\Delta_1^-=0$ by Claim 6, there is no tip of $D_1$ contained in $\bar Q_1$, so $Q_1$ is necessarily a chain and $C_1$ is its tip. This means that $q_1$ is semi-ordinary; a contradiction. Thus $A_1\cdot Q_2=0$. Again, if $j=2$ then $q_2$ is not semi-ordinary and $\alpha_2=-1$, which is impossible, as then $\tau_2^*=0$ and $Q_2-C_2$ consists of $(-2)$-curves. Thus $\alpha_1=-1$. By Lemma \ref{lem:elements}(e) it follows that $\tau_1^*=0$ and $\psi$ does not touch $C_1$. The divisor $\bar Q_1-\psi(C_1)$ consists of one $(-1)$-curve $\cL$ and some number of $(-2)$-curves. Because $\Delta_1^-=0$, no tip of $D_1$ is contained in $\bar Q_1$. It follows that $\psi(C_1)$ is a tip of $\bar Q_1$ and $\bar Q_1-\psi(C_1)$ is a circular divisor with $\cL$ as its unique component meeting $\psi(C_1)$. Contract $\bar Q_2+\bar Q_3+\bar Q_4+\cL$. Then the image of $E_1$ has three cusps and is a $2$-section of the $\P^1$-fibration induced by the linear system of the image of $\psi(C_1)$. This is a contradiction with the Hurwitz formula.
\end{proof}

The inequality \eqref{eq:KR2} and Claims 6 and 8 give
\begin{equation}\label{eq:KR4}
\alpha+K_1\cdot R\leq \epsilon+\lambda_1+\lambda_2-8\leq \epsilon.
\end{equation}

\begin{claim} $A_1\cdot E_0=0$. \end{claim}

\begin{proof}Suppose that $A_1\cdot E_0=1$. Let $U=\psi_*^{-1}\cL$. Then $\cL\cdot E_1=1$ and, by Claim 8, $\cL\subseteq \Upsilon_1$. For $\epsilon=1$ this is impossible, hence $\epsilon=0$. Then $\Exc \psi=A_1$, $U=[2]$ and $\Delta_1^+\wedge \bar Q_1\neq 0$. By Claim 7 and \eqref{eq:KR4} $$\displaystyle K_0\cdot (Q_1-C_1+Q_2-C_2)=K_1\cdot R\leq -\alpha\leq 1.$$ Let $\ind_U$ be the contribution to $\ind_1^{(1)}$ from the $(-2)$-twig in $\Delta_1^+$ meeting $\cL=\psi(U)$. Since $Q_1'$ contains at least two tips of $D$ and the proper transform of $A_1$ meets none of them, we get $\ind_1^{(1)}>\ind_U\geq \frac{1}{2}$, hence by Lemma \ref{lem:n=1_cases} we are in case (1.d). Put $\ind_1'=\ind_1^{(1)}-\ind_U>0$. We have $\ind_1'+\ind_2^{(1)}\leq \frac{4}{3}-\frac{1}{2}=\frac{5}{6}$. By Lemma \ref{lem:log BMY}(a), $\ind_1'<\frac{5}{6}-\frac{1}{2}=\frac{1}{3}$, hence $s_1=0$ and if $\bar Q_1-\Delta_{1}^+\wedge\bar Q_1$ contains some tip $V$ of $D_1$ then $V^2\leq -4$. Since $K_0\cdot (Q_1-C_1)\leq 1$, there is no such tip. In particular, $\bar Q_1$ is a chain with $\psi(C_1)$ or $\psi(\tilde C_1)$ as a tip. Then $Q_1$ is a chain, either $Q_1=[(2)_k,3,1,2]$ or $Q_1=[1,(2)_k]$ for some $k\geq 0$. By the inequality above $\tau_1=\alpha_1+2\leq 2$, hence $\tau_1=2$ and $\alpha_1=0$. It follows that $Q_1=[(2)_k,3,1,2]$, because otherwise we get $\ind_1'\geq \frac{1}{3}$. By \eqref{eq:KR4}, $\alpha=-1$ and $Q_2-C_2$ consists of $(-2)$-curves. The former gives $\alpha_2=-1$, which implies that $q_2$ is semi-ordinary. Then $\alpha=\alpha_1=0$; a contradiction.
\end{proof}

By Claim 9, $A_1\cdot E_0=0$, so $\epsilon=0$ and for every component $V\subseteq D_1-E_1$ we have $V\cdot E_1=\psi_*^{-1}V\cdot E_0$. By \eqref{eq:KR4} and by Claim 7, 
\begin{equation}\label{eq:KR4_2}
\displaystyle K_1\cdot R\leq -\alpha\leq 1.
\end{equation}
If $A_1$ meets both $Q_1$ and $Q_2$ then by Lemma \ref{lem:n=1_cases}, $\lambda_1=\lambda_2=4$, so in any case we may, and will, assume that $\cL\subseteq \bar Q_1$.

\begin{claim} $\mathcal{L}$ meets three components of $D_1-\mathcal{L}$ and does not meet $E_1$. \end{claim}

\begin{proof} First suppose that $\cL\cdot E_1\neq 0$. We have $\cL\cdot E_1=\psi_*^{-1}\cL\cdot E_0$, so $\psi_*^{-1}\cL=\tilde C_1$. Since $\psi$ touches $\tilde C_1$, we have $\beta_{D_1}(\cL)\geq 3$, hence $\cL$, as a component of $\Upsilon_1$, meets $\Delta_1^+$. Then $C_1$ meets a connected component of $\Delta_0$ not contracted by $\psi$, so $\beta_{Q_1+A_1}(\tilde C_1)\geq 3$, hence $\beta_{D_1}(\cL)>3$. Then $\cL\nsubseteq \Upsilon_1$; a contradiction.

Thus $\cL\cdot E_1=0$. Suppose that $\cL $ does not meet three components of $D_1-\cL$. We have $\cL\subseteq \Upsilon_1$, so then $\cL$ meets at most two components of $D_1-\cL$, namely some $B\subseteq R+\psi(C_1)$ with $B\cdot \cL=2$ and possibly some component of $\Delta_1^+$. It follows that $\psi$ does not touch $Q_2$. Let $\theta\:X_1\to \theta(X_1)$ be the contraction of $\bar Q_2+\bar Q_3+\bar Q_4$. Suppose that $B\neq \psi(C_1)$. Then $K_1\cdot B\leq K_1\cdot R\leq 1$, so $B^2\in \{-2,-3\}$. By Lemma \ref{lem:ell_fib}(b) we may blow up on $\theta(B)$ so that the linear system of the proper transform of $F=B+2\cL$ induces an elliptic fibration. But $F\cdot \theta(E_1)\leq \psi_*^{-1}B\cdot E_0\leq 1$, so $\theta(E_1)$ is a $1$-section or is vertical. The former is impossible, because $\theta(E_1)$ is not smooth and the latter is impossible, because degenerate fibers of elliptic fibrations of smooth surfaces cannot contain components with more than $1$ cusp; a contradiction. Thus $B=\psi(C_1)$. In particular, $\psi$ touches $C_1$, hence $\tau_1\geq 3$ by Lemma \ref{lem:basic_bounds}(c). By \eqref{eq:KR4_2}, $\alpha\in \{-1,0\}$, hence $\alpha_1\leq 1$. From the definition of $\alpha_j$ we get $\psi(C_1)^2+3=\tau_1-\alpha_1-s_1\geq \tau_1-2$. Again contract $\bar Q_2+\bar Q_3+\bar Q_4$ and blow up $\tau_1-\alpha_1-s_1+1$ times over $\psi(C_1)$, of which the first $\tau_1-\alpha_1-s_1$ times on the intersection of the proper transforms of $\psi(C_1)$ and $E_1$, and then once on the proper transform of $\psi(C_1)$. Denote the final proper transforms of $\psi(C_1)$, $E_1$ and $\cL$ by $C'$, $E'$ and $\cL'$, respectively. We have $(C')^2=-4$ and $C'\cdot E'\leq \alpha_1+s_1\leq 2$. By Lemma \ref{lem:ell_fib}(b) the last center may be chosen so that $|C'+2\cL'|$ induces an elliptic fibration such that $E'$ meets a general fiber at most twice. Since $E'$ has at least two cusps, it cannot be vertical, hence it meets the fiber once or twice. This is impossible by the Hurwitz formula.
\end{proof}

\begin{claim}$\cL$ meets one $(-1)$-curve and one $(-3)$-curve. \end{claim}

\begin{proof}
We have $\eta=1$, so $\cL\subseteq \Upsilon_1$. By Claim 10 there are two distinct components $B_1, B_2\subseteq \bar Q_1+\bar Q_2-\Delta_1$ and a component $T\subseteq \Delta_1^+$ such that $\cL\cdot B_1=\cL\cdot B_2=\cL\cdot T=1$ and $\cL$ does not meet $D_1-\cL-B_1-B_2-T$. It follows that $B_1$ and $B_2$ meet $\cL$ in different points. 

Suppose that $B_j^2\geq 0$ for some $j\in \{1,2\}$. By Lemma \ref{lem:elements}(d), $B_j=\psi(C_1)$ and from the definition of $\alpha_j$ we have $\psi(C_1)^2=\tau_1-\alpha_1-s_1-3$. Contract $\bar Q_3+\bar Q_4$ and blow up $\psi(C_1)^2+2=\tau_1-\alpha_1-s_1-1\geq 2$ times over $\psi(C_1)$, each time on the intersection of the proper transforms of $\psi(C_1)$ and $E_1$. Denote the final proper transforms of $\psi(C_1)$, $E_1$, $T$ and $\cL$ by $C'$, $E'$, $T'$ and $\cL'$, respectively.  We have $(C')^2=-2$ and $C'\cdot E'=\alpha_1+s_1+1\leq \alpha_1+2$. Since $\alpha\in \{-1,0\}$, we have $\alpha_1\leq 1$. The linear system of $F=C'+2\cL+T'$ gives a $\P^1$-fibration such that $F\cdot E'=C'\cdot E'$, so by the Hurwitz formula $C'\cdot E'-1+\mu(q_3)-1+\mu(q_4)-1\leq 2 C'\cdot E'-2$, hence $C'\cdot E'\geq 3$. We obtain $\alpha_1=s_1=1$, hence $R$ consists of $(-2)$-curves, $\alpha_2=-1$ and $q_2$ is not semi-ordinary. Then $s_2=1$ and, by Lemma \ref{lem:basic_bounds}(d), $\tau_1\geq 3$. The contribution of $\Delta_1^+$ to $\ind_1^{(1)}+\ind_2^{(1)}$ is at least $\frac{1}{2}$, so we obtain $\ind_1^{(1)}+\ind_2^{(1)}\geq \frac{2}{3}+\frac{1}{2}+\frac{1}{2}=\frac{5}{3}$. This is a contradiction with Lemma \ref{lem:n=1_cases}. 

By \eqref{eq:KR4_2}, $B_j^2\in \{-1,-2,-3\}$ for $j\in \{1,2\}$. If $B_j^2=-2$ for some $j\in\{1,2\}$ then $B_j\cdot E_1=\psi_*^{-1}B_j\cdot E_0\leq 1$, so after the contraction of $\bar Q_3+\bar Q_4$ the linear system of the image of $F=B_j+2\cL+T$ gives a $\P^1$-fibration for which the image of $E_1$ is singular and meets a general fiber at most once, which is impossible. Thus $B_j^2\in \{-1,-3\}$ for $j\in \{1,2\}$. 

If $B_1^2=B_2^2=-1$ then $B_j=\psi(C_j)$ for $j=1,2$, hence $\psi$ does not touch $C_1+C_2$ and then $\psi_*^{-1}\cL$ is a component of $D_0-E_0$ meeting both $C_1$ and $C_2$, which is impossible. Since $K_1\cdot R\leq 1$, we obtain $\{B_1^2,B_2^2\}=\{-1,-3\}$
\end{proof}

Let $B_1$, $B_2$ and $T$ be the three components of $D_1-E_1-\cL$ meeting $\cL$. By Claim 11 we may assume that $B_1=\psi(C_1)$ is a $(-1)$-curve and $B_2$ is a $(-3)$-curve. We have $T\subseteq \Delta_1^+$. From the definition of $\alpha_j$ we have $\tau_1=2+s_1+\alpha_1\leq 4$, as $\alpha_1\leq 1$. Contract $\bar Q_3+\bar Q_4+\cL+T$ and blow up four times over $\psi(C_1)$, the first $\tau_1$ times on the intersections of the proper transforms of $\psi(C_1)$ and $E_1$. Blow up once more at some point of the proper transform of $\psi(C_1)$, off the proper transform of $B_2$. Denote the final proper transforms of $B_1, B_2$ and $E_1$ by $B_1', B_2'$ and $E_1'$, respectively. We have $(B_1')^2=-4$, $(B_2')^2=-1$ and $B_1'\cdot E'=0$. Since $\psi(C_1)$ is a $(-1)$-curve, $\psi$ does not touch $C_1$, so $C_1$ meets $\psi_*^{-1}\cL$. If $B_1'\cdot B_2'\neq 2$ then, by construction, $C_1$ meets $\psi_*^{-1}(B_2)$ too, $\psi_*^{-1}B_2\neq \tilde C_1$ and $A_1\cdot Q_1=2$. But in the latter case $\Delta_1^+\wedge \bar Q_1\neq 0$ and $s_1=1$ (as $C_1$ is non-branching in $Q_1$), so $\ind_1^{(1)}+\ind_2^{(1)}>1+\frac{1}{2}$, in contradiction to Lemma \ref{lem:n=1_cases}. Thus $B_1'\cdot B_2'=2$. By Lemma \ref{lem:ell_fib}(b) the last center may be chosen so that $|B_1'+2B_2'|$ induces an elliptic fibration. A general fiber meets $E'$ at most $2B_2'\cdot E'$ times. We have $B_2'\cdot E'=\psi_*^{-1}B_2\cdot E_0$ and the latter number is at most $1$, because $B_2\neq \psi(C_2)$. Since $E'$ has two cusps, it is neither a $1$-section nor it is contained in a fiber, hence it meets a general fiber twice. Because $E'$ meets $B_1'+B_2'$ in a unique point, Hurwitz formula gives $1+(\mu(q_3)-1)+(\mu(q_4)-1)\leq 2$; a contradiction.
\end{proof}

\clearpage

\section{Consequences of the almost minimality} \label{sec:n=0}

We now analyze various consequences of Theorem \ref{thm:almost_minimal}. The essential one is that the inequality \eqref{eq:lambda} results with an upper bound on the number of components of $D$ and hence on the degree of $\bar E$, see Remark \ref{rem:bounds}. We first describe the geometry of exceptional divisors of resolutions of cusps in terms of Hamburger-Noether pairs, which are a compact and geometrically meaningful way to keep track of multiplicity sequences. For a detailed treatment see \cite{Russell_HN_pairs} and \cite[Appendix]{KR-C*_actions_on_C3}. For a discussion of relations between standard Hamburger-Noether pairs, Puiseux pairs and other numerical characteristics of cusps see \cite[Appendix]{PaPe_cuspidal-Cstst-fibrations}. 

\subsection{Description in terms of Hamburger-Noether pairs}\label{sec :5A}

Recall that we say that a divisor $Q$ on a projective surface \emph{contracts to a smooth point} if it is equal to the reduced exceptional divisor of some birational morphism mapping $Q$ onto a smooth point of some surface. Equivalently, $Q$ has a negative definite intersection matrix and $d(Q)=1$, cf. \eqref{eq:d(T)}, see \cite[1.18.1]{Russell-formal_aspects_M-R}. 

\begin{dfn}[Hamburger-Noether pairs]\label{def:HN} Let $Q$ be a divisor which contracts to a smooth point and contains a unique $(-1)$-curve $L$. By induction with respect to the number of branching components of $Q$ we define a sequence of positive integers, called \emph{standard Hamburger-Noether (HN-) pairs} of $Q$
\begin{equation}\label{eq:HN_sequence}
\binom{c^{(1)}}{p^{(1)}}, \binom{c^{(2)}}{p^{(2)}}, \ldots,\binom{c^{(h)}}{p^{(h)}}
\end{equation}
as follows. Let $\sigma$ denote the contraction of the maximal twig of $Q$ containing $L$. Write $\Exc \sigma-L=C+P$, where $C, P$ are reduced connected and disjoint, and such that $d(C)\geq d(P)$ if $Q$ is a chain and that $P$ meets the proper transform of $\sigma(Q)$ otherwise. In the first case put $h=1$ and $\binom{c^{(1)}}{p^{(1)}}=\binom{d(C)}{d(P)}$. Otherwise let $\left(\binom{\bar c^{(k)}}{\bar p^{(k)}}\right)_{k=1}^{\tilde h}$ denote the sequence of standard HN-pairs of $\sigma(Q)$. Put $h=\tilde h+1$, $c=d(C)$ and define the sequence of standard HN-pairs of $Q$ as 
\begin{equation}\label{eq:composition}
\binom{c^{(h)}}{p^{(h)}}=\binom{c}{d(P)} \text{\ \ and\ \ } \binom{c^{(k)}}{p^{(k)}}=\binom{c\cdot \bar c^{(k)}}{c\cdot \bar p^{(k)}} \text{\ \ for } k\in \{1,\ldots,h-1\}.
\end{equation} 
\end{dfn}

\noindent Note that $Q$ has exactly $h-1$ branching components. 

\smallskip 
By $\lfloor x\rfloor$ and $\lceil x \rceil$ we denote respectively the biggest integer not greater than $x$ and the smallest integer not smaller than $x$. We have $c-\lfloor c/p\rfloor \cdot p=(c \bmod p)$. 

\begin{lem}[Properties of contractible chains]\label{lem:c,p}
Let $Q^+$ be a reduced chain with tips $Z_P\neq Z_C$ and such that $Q=Q^+-Z_P-Z_C$ contains a unique $(-1)$-curve $L$ and contracts to a smooth point. Write $Q-L=P+C$, where $P$, $C$ are reduced connected and disjoint and $C\cdot Z_P=P\cdot Z_C=0$; see Figure \ref{Fig:c,p}. The following hold:
\begin{enumerate}[(a)]
\item $c:=d(C)$ and $p:=d(P)$ are relatively prime, cf.\ \eqref{eq:d(T)},
\item Denote the tips of $P$ and $C$ meeting $Z_P$ and $Z_C$ by $L_P$ and $L_C$, respectively; put $L_P=0$ if $P=0$ and $L_C=0$ if $C=0$. Then
\begin{equation}\label{eq:HN_prop1}
d(P-L_P)=p-(c \bmod p) \text{\ \ \ and\ \ \ } d(C-L_C)=c-(p \bmod c).
\end{equation}
If $L_P\neq 0$ and, respectively, $L_C\neq 0$ then
\begin{equation}\label{eq:HN-tips}
-L_P^2=\lceil \frac{p}{p- (c \bmod p)}\rceil \text{\ \ \ and\ \ \ }-L_C^2=\lceil \frac{c}{c- (p \bmod c)}\rceil.
\end{equation}
\item Let $\chi$ be a smooth germ meeting $L$ normally, not in a node of $Q^+$, and let $\sigma$ be the contraction of $Q$. Then:
\begin{equation}\label{eq:HN_prop2}
\sigma(Z_P)\cdot \sigma(\chi)=c \text{\ \ \  and\ \ \ }\sigma(Z_C)\cdot \sigma(\chi)=p.
\end{equation}
\end{enumerate}
  \begin{figure}[h]\centering
  \includegraphics[scale=0.65]{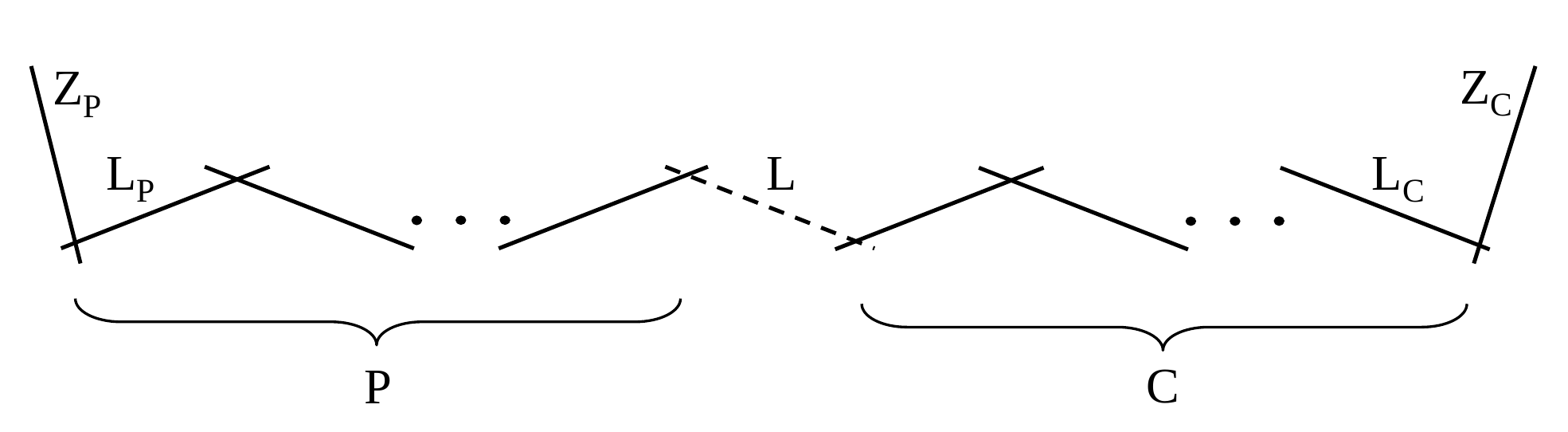}\caption{The chain $Z_P+Q+Z_C$ from Lemma \ref{lem:c,p}.}\label{Fig:c,p}
  \end{figure}
\end{lem}

\begin{proof}
First consider a rational chain $Z$ with a unique $(-1)$-curve $L$ and such that $Z-L$ consists of two admissible chains $A$ and $B$. We claim that if $Z$ contracts to a $0$-curve then $d(A)=d(B)$. This was shown in \cite[3.7, 4.7]{Fujita-noncomplete_surfaces}, but it is worth recalling a simple proof. We order the components of $A$ and $B$ so that the first ones meet $L$. Using \eqref{eq:d_recurrence} we show by induction on $n$ that $\gcd(d(A),d'(A))=1$; similarly $\gcd(d(B),d'(B))=1$. Using \eqref{eq:d(T1+T2)} we note that if we replace a tree by its reduced total transform under a blowup then its discriminant does not change. Finally, using \eqref{eq:d(T1+T2)} twice gives $$0=d([0])=d(Z)=d(A)d(B)-d'(A)d(B)-d(A)d'(B),$$ hence $d(A)|d(B)$ and $d(B)|d(A)$, so indeed $d(A)=d(B)$.

(a) We choose an order on the set of irreducible components of $P$ in which $L_P$ is the last one; we do the same for $C$. As in the argument above, we have $$1=d([1])=d(Q)=pc-pd'(C)-cd'(P),$$ hence $\gcd(c,p)=1$.
 
(b),(c) We proceed by induction with respect to $\#Q$. If $Q=L$ then the statements hold, so assume $\#Q\geq 2$. Write $\sigma=\sigma_0\circ \sigma'$, where $\sigma_0$ is a single blowup. The chain $Q'=\Exc \sigma'\subseteq Q$ contracts to a smooth point, $L$ is its unique $(-1)$-curve and $Q-Q'$ is a tip of $Q$. Since the formulas to prove are symmetric, we may assume that $Q-Q'=L_P$. Let $L_P'$ be the tip of $Q'-L$ meeting $L_P$; put $L_P'=0$ if $P=L_P$. Put $p'=d(P-L_P)$. By induction 
\begin{equation*}
d(P-L_P-L_P')=p'-(c\bmod p'), \ \  d(C-L_C)=c-(p' \bmod c)
\end{equation*}
and
\begin{equation*}
\sigma'(L_P)\cdot \sigma(\chi)=c, \ \  \sigma'(Z_C)\cdot \sigma'(\chi)=p'.
\end{equation*}
Since the statements to prove do not depend on properties of $Z_P$ other than normality of the intersection with $L_P$, we may assume that $Z_P$ is a $(-1)$-curve. Then $\sigma(Z_P)$ is a $0$-curve, so $d(Z_P+P)=d(C)$ by the observation above. By \eqref{eq:d_recurrence} we obtain $1\cdot p-p'=c$, that is, $0<p'=p-c=p-(c\bmod p).$ Hence $d(C-L_C)=c-(p'\bmod c)=c-(p\bmod c)$, which proves \eqref{eq:HN_prop1}. Using the projection formula and the inductive assumption we obtain $$\sigma(Z_P)\cdot \sigma(\chi)=(\sigma'(Z_P)+\sigma'(L_P))\cdot \sigma'(\chi)=c$$ and $$\sigma(Z_C)\cdot \sigma(\chi)=(\sigma'(Z_C)+\sigma'(L_P))\cdot \sigma'(\chi)=c+p'=p,$$ which proves \eqref{eq:HN_prop2}.

From \eqref{eq:HN_prop2} we have $\sigma'(L_P)\cdot \sigma'(\chi)=c$ and $\sigma'(Z_C)\cdot \sigma'(\chi)=p-c$, so when blowing up over $\sigma'(L)$ the center stays on the proper transform of $L_P$ exactly $\lceil \frac{c}{p-c}\rceil$ times, which gives $$-L_P^2=1+\lceil \frac{c}{p-c}\rceil=\lceil \frac{p}{p-c}\rceil=\lceil \frac{p}{p-(c \bmod p)}\rceil.$$ If $L_C\neq 0$ then, by induction, $$-L_C^2=\lceil \frac{c}{c-((p-c)\bmod c)}\rceil=\lceil \frac{c}{c-(p\bmod c)}\rceil,$$ which proves \eqref{eq:HN-tips}.
\end{proof}

Assume that $Q$ is as in Definition \ref{def:HN}. Putting formally $c^{(h+1)}=1$ and using Lemma \ref{lem:c,p}(a) we note that:
\begin{align}
\begin{split}
\gcd(c^{(k)},p^{(k)})&= c^{(k+1)} \text{\ \ for } k\in\{1,\ldots,h\}, \\ 
c^{(k)}& \neq p^{(k)} \text{\ \ \ \ for } k\in\{2,\ldots,h\}, \\
c^{(1)} & > p^{(1)} \text{\ \ \ \ if\ } Q\neq L.
\end{split}
\end{align} 

\smallskip
\begin{rem}
[cf.\ Remark 2.10 \cite{PaPe_cuspidal-Cstst-fibrations}] 
Lemma \ref{lem:c,p}(c) shows that the above definition of HN-pairs for $Q$ is close to the one used in \cite[\S 3]{Palka-Coolidge_Nagata1}. The difference is that we no longer insist that $c^{(k)}\geq p^{(k)}$, so while we lose some flexibility, our sequences are usually shorter than the ones in loc.\ cit. One may easily pass from one sequence to another by replacing our pairs $\binom{c^{(k)}}{p^{(k)}}$ for which $c^{(k)}<p^{(k)}$ with $(\binom{c^{(k)}}{c^{(k)}})_q, \binom{c^{(k)}}{p^{(k)} -q c^{(k)}}$, where $q$ is the integer part of $p^{(k)}/c^{(k)}$. 
\end{rem}

For a given logical formula $\varphi$ without free parameters we denote by $[\varphi]\in \{0,1\}$ its logical value, that is, we put
\begin{equation}\label{eq:characteristic_function}
[\varphi]=1 \text{\ if\ } \varphi \text{\ is true and\ } [\varphi]=0 \text{\ otherwise}.
\end{equation}
We denote by $\Delta(Q)$ the sum of maximal $(-2)$-twigs of $Q$. We put 
\begin{align}\label{eq:def_s,lambda,M,I}
\begin{split}
s(Q)&=\Delta(Q)\cdot L\in \{0,1\},\\
\lambda(Q)&=\#Q-b_0(\Delta(Q))-[\#Q=3],\\
M(Q)& =c^{(1)}+\sum_{k=1}^{h}p^{(k)}-1,\\
I(Q)&=\sum_{k=1}^{h}c^{(k)}p^{(k)}.
\end{split}
\end{align}
Note that if $L$, the unique $(-1)$-curve of $Q$, is not a tip of $Q$ then $Q$ has $h+1$ maximal admissible twigs. We define $\ind(Q)$ as in Section \ref{sec:prelim}. 
Recall that given a nonzero reduced snc divisor on a smooth surface we say that a blow-up at a point $p$ of the divisor is \emph{inner} if $p$ is a node, otherwise it is \emph{outer}. By $r(Q)$ we denote the number of outer blow-ups used to create $Q$. We note that the first blow-up in the process of reconstruction of $Q$ is neither outer nor inner. 

Given a pair of positive integers $c$, $p$ we put 
\begin{equation}\label{eq:|(c,p)|}
|\binom{c}{p}|=[a_1,\ldots,a_m,1,b_1,\ldots,b_n],
\end{equation}
where $A=[a_1,\ldots,a_m]$ and $B=[b_1,\ldots,b_n]$ are unique admissible chains with $d(A)=p/\gcd(c,p)$ and $d(B)=c/\gcd(c,p)$ such that $[a_1,\ldots,a_m,1,b_1,\ldots,b_n]$ contracts to a smooth point. 
Note that in the Euclidean algorithm for $(c,p)$ we reach $0$ after exactly $\#|\binom{c}{p}|=m+n+1$ steps.

We put (cf.\ \eqref{eq:characteristic_function}) 
\begin{equation}
b^-(\binom{c}{p})=[p\geq 2 (c \bmod p)>0] \text{\ \ and\ \ } b^+(\binom{c}{p})=[c\geq 2(p \bmod c)>0],
\end{equation}
and we define 
\begin{equation}
\lambda^+(\binom{c}{p})=\#|\binom{c}{p}|-b^+(\binom{c}{p}).
\end{equation}

\begin{ex}
The sequences of standard HN-pairs $\binom{15}{6}\binom{3}{2}$ and $\binom{10}{4}\binom{2}{3}$ describe two forks, say, $Q_1'$ and $Q_2'$. Both divisors have a unique branching component $B$ and three maximal twigs $T_1$, $T_2$, $T_3$. The first pairs create a chain $[2,3,1,2]$ and the second pairs change this chain into a fork with a branching curve $B$ and three maximal twigs $T_1=[2,3]$, $T_2=[2]$ and $T_3$. In the first case we have $B^2=-3$ and $T_3=[3,1,2]$ and in the second $B^2=-2$ and $T_3=[2,1,3]$. We compute $M(Q_1')=22$, $I(Q_1')=96$ and $M(Q_2')=16$, $I(Q_2')=46$. The first fork requires $r(Q_1')=3$ outer blowups and the second requires $r(Q_2')=4$ outer blow-ups. In both cases for the first pair $b^+=1$ and for the second pair $b^+=0$ and $b^+=1$, respectively. In the first case the sum of $\lambda^+$ equals $6$ and $\lambda=5$. In the second case the sum of $\lambda^+$ equals $5$ and $\lambda=4$.
\end{ex}

To effectively make computations we prove formulas for various quantities in terms of HN-pairs. 

\begin{lem}\label{lem:ind_lambda_using_HN} Let $Q$ be a divisor which contracts to a smooth point and contains a unique $(-1)$-curve, which is not a tip of $Q$. 
Let $\left(\binom{c^{(k)}}{p^{(k)}}\right)_{k=1}^{h}$ be the standard HN-sequence of $Q$.  Then:
\begin{enumerate}[(a)]
\item $\displaystyle  \ind(Q)=1-\frac{c^{(1)} \bmod p^{(1)}}{p^{(1)}}+\sum_{k=1}^{h}(1-\frac{p^{(k)} \bmod c^{(k)}}{c^{(k)}})$,  
 
\item  $\displaystyle r(Q)=\lceil c^{(1)}/p^{(1)} \rceil-1+\sum_{k=2}^{h} \lceil p^{(k)}/c^{(k)}\rceil\geq h$,  

\item $\displaystyle b_0(\Delta(Q))=b^-(\binom{c^{(1)}}{p^{(1)}})+\sum_{k=1}^{h}b^+(\binom{c^{(k)}}{p^{(k)}})$,

\item  $\displaystyle  \lambda(Q)=-b^-(\binom{c^{(1)}}{p^{(1)}})+\sum_{k=1}^{h}\lambda^+(\binom{c^{(k)}}{p^{(k)}})-[\binom{c^{(1)}}{p^{(1)}}=\binom{3}{2}]$,

\item $s(Q)=[h>1 \text{\ and  } (p^{(h)} \bmod c^{(h)})=1]+[h=1 \text{\ and  } (c^{(1)} \bmod p^{(1)})=1]$.  
\end{enumerate}
\end{lem}

\begin{proof} Let $L$ denote the unique $(-1)$-curve of $Q$. Let $Q^{(h)}$ be the maximal twig of $Q$ containing $L$ and let $\sigma$ denote its contraction. Put $\binom{c}{p}=\binom{c^{(h)}}{p^{(h)}}$.  We have $Q^{(h)}-L=P+C$, where the divisors $P,C$ are reduced connected and disjoint, $d(P)=p$ and $d(C)=c$. Since $L$ is not a tip of $Q$, we have $C\neq 0$ (hence $c>1$) and $p^{(1)}>1$. By Lemma \ref{lem:c,p}(b)
\begin{equation*}
d(P-L_P)=p-(c \bmod p) \text{\ \ \ and\ \ \ } d(C-L_C)=c-(p \bmod c),
\end{equation*}
so (a) follows from the definition of the inductance. For (b) note that by definition the first blow-up within the HN-pair $\binom{c^{(k)}}{p^{(k)}}$ is outer if and only if $k\geq 2$. Then the formula follows from Lemma \ref{lem:c,p}(c). By Lemma \ref{lem:c,p}(b), $L_P=[2]$ if and only if $p\geq 2(c \bmod p)>0$, and $L_C=[2]$ if and only if $c\geq 2(p \bmod c)>0$, which gives (c). 

Part (d) follows from (c) and from the definition of $\lambda(Q)$. For (e) we observe that if $h>1$ then by \eqref{eq:HN_prop1}, $C$ is a $(-2)$-chain if and only if $p^{(h)} \bmod c^{(h)}=1$. Similarly, if $h=1$ then $L$ meets a $(-2)$-twig if and only if $c^{(1)} \bmod p^{(1)}=1$ (the twig is necessarily $P$).  % or $p^{(1)}|c^{(1)}$ (then the twig is $C$).
\end{proof}
\smallskip

We define the following function on the set of positive integers:
\begin{equation}
    \nu(\lambda)= \begin{cases}
               5/6                   & \lambda=1,\\
              11/10                & \lambda=2,\\
              11/12                & \lambda=3,\\
              \frac{1}{\lambda-1}+\frac{\lambda-2}{2\lambda-3}		     & \lambda\geq 4.\\
           \end{cases}
\end{equation}
\noindent We note the following values: $\nu(4)=\frac{11}{15}$, $\nu(5)=\frac{19}{28}$, $\nu(6)=\frac{29}{45}$ and $\nu(7)=\frac{41}{66}$.

\begin{lem}[Lower bound on inductance] \label{lem:lambda_bounds} Let $Q$ be a divisor which contracts to a smooth point and contains a unique $(-1)$-curve, which is not a tip of $Q$. Then:
\begin{enumerate}[(a)]
\item $\min\{\ind(Q): Q \text{\ such that\ }\lambda(Q)=\lambda\}=\nu(\lambda)$ and the minimum for every $\lambda\geq 4$ (respectively for $\lambda=3$) is attained only at the chain $[\lambda-1,1,(2)_{\lambda-3},3]$ (respectively $[4,1,2,2]$). 
\item If $\ind(Q)\leq \nu(\lambda(Q)-1)$ then $Q$ is a chain.
\item If $\lambda(Q)\leq 7$ and $\ind(Q)\leq \frac{2}{3}$ then:
\begin{enumerate}[(i)]
\item $Q=[5,1,2,2,2,3]$, $(\lambda(Q),\ind(Q))=(6,\frac{2}{3}-\frac{1}{45})$ or 
\item $Q=[6,1,2,2,2,2,3]$, $(\lambda(Q),\ind(Q))=(7,\frac{2}{3}-\frac{1}{22})$ or 
\item $Q=[5,2,1,3,2,2,3]$, $(\lambda(Q),\ind(Q))=(7,\frac{2}{3}-\frac{1}{144})$.
\end{enumerate} 
\end{enumerate}
\end{lem}

\begin{proof} By Proposition \ref{prop:lambda=3} and Lemma \ref{lem:basic_bounds}(e) the above statements hold if $\lambda<4$. For $Q=[\lambda-1,1,(2)_{\lambda-3},3]$ we have $\lambda(Q)=\lambda$ and $\ind(Q)=\nu(\lambda)$, so $$\min\{\ind(Q): \lambda(Q)=\lambda\}\leq \nu(\lambda).$$ To prove the opposite inequality w argue by induction with respect to $\lambda$. Let $\lambda_0\geq 4$ and assume (a) and (b) hold for every $Q$ with $\lambda(Q)\leq \lambda_0-1$. By induction and by the definition of $\nu(\lambda)$, $\min\{\ind(Q): \lambda(Q)=\lambda\}$ is a decreasing function of $\lambda$ for $2\leq \lambda\leq \lambda_0$ and its value for $\lambda=4$ is smaller than for $\lambda=1$.

To prove (b) suppose that there exists a branched divisor $Q$ with $\lambda(Q)=\lambda_0$ and $\ind(Q)\leq \nu(\lambda_0-1)$. Contract $(-1)$-curves in $Q$ and its images until there is no branching component and denote the image of $Q$ by $Q'$. We have $\lambda(Q')\leq \lambda_0-1$, so by the inductive assumption $\nu(\lambda(Q'))\leq\ind(Q')<\nu(\lambda_0-1)$, hence $\lambda(Q')=1$, which gives $Q'=[3,1,2]$. Now $\frac{5}{6}<\nu(\lambda_0-1)$, so since $\lambda_0\geq 4$, we get $\lambda_0=4$ and $\nu(\lambda_0-1)=\frac{11}{12}$. Since $\lambda(Q)=4$, the contracted divisor is a twig, either $[2,2,1]$ or $[2,1,3]$. In both cases the inequality $\ind(Q)\leq \frac{11}{12}$ fails; a contradiction. 

To prove (a) suppose that $Q$ is a divisor with $\lambda(Q)=\lambda_0$ and $\ind(Q)\leq \nu(\lambda_0)$. Then $\ind(Q)\leq \nu(\lambda_0-1)$, so part (b) implies that $Q$ is a chain. We will repeatedly use Lemma \ref{lem:ind} to bound $\ind(Q)$ from below. For instance, since $\nu(\lambda_0)<1$, we have $b_0(\Delta(Q))\leq 1$. Let $U_i$ for $i\in\{1,\ldots,\#Q\}$ denote the proper transform of the exceptional curve of the $i$'th blow-up in the sequence of blow-ups reversing the contraction of $Q$ to a smooth point. We treat $Q$ as the exceptional divisor of a resolution of a cusp. After the first blow-up the proper transform of the germ meets $U_1$. It separates from $U_1$ after $x_1$ further blow-ups, $x_1\geq 1$. 

Suppose that $x_1=1$. Then $b_0(\Delta(Q))=1$ and $\#Q=\lambda_0+1$. After $x_2\geq 1$ more blowups the germ separates from $U_2$ and the created chain is $[2,x_2+1,1,(2)_{x_2-1}]$. Since $b_0(\Delta(Q))=1$, we have $x_2\leq 2$. If $x_2=1$ then $Q$ contains the twig $[2,2]$, so $\nu(5)<\frac{2}{3}<\ind(Q)\leq \nu(\lambda_0)$ and hence $\lambda_0=4$ and $\#Q=\lambda_0+b_0(\Delta(Q))=5$, in which case $Q=[2,2,3,1,2]$, in contradiction to the equality $b_0(\Delta(Q))=1$. Thus $x_2=2$ and hence $Q$ contains a twig $[2,3]$. There are $\#Q-4=\lambda_0-3$ blow-ups remaining to be done, so the self-intersection of the tip of $Q$ other than $U_1$ is not smaller than $-2-(\lambda_0-3)=1-\lambda_0$. Then $\frac{3}{5}+\frac{1}{\lambda_0-1}\leq \ind(Q)\leq \nu(\lambda_0)$, which is impossible; a contradiction.

Thus $x_1\geq 2$. After the first $x_1+1$ blow-ups the created chain is $[(2)_{x_1-1},1,1+x_1]$. Suppose that $x_1\geq 3$. Then $x_1\leq \#Q-1= \lambda_0$ and $Q$ contains the twigs $[(2)_{x_1-2}]$ and $[1+x_1]$, hence $\frac{x_1-2}{x_1-1}+\frac{1}{x_1+1}\leq \ind(Q)\leq \nu(\lambda_0)\leq\nu(x_1)$, which gives $x_1=3$ and $\nu(\lambda_0)\geq \frac{3}{4}$; a contradiction. 

Thus $x_1=2$. After the first three blow-ups the created chain is $[2,1,3]=U_2+U_3+U_1$. Since $\#Q>3$, now the germ meets $U_2$ and $U_3$, so $\Delta(Q)=0$ and hence $\#Q=\lambda_0$. After $x_2\geq 1$ additional blow-ups the germ separates from $U_2$ and the chain becomes $[2+x_2,1,(2)_{x_2},3]$. We see that $Q$ contains the twigs $[2+x_2]$ and $[3,(2)_{x_2-1}]$. Clearly, $x_2+3\leq \#Q=\lambda_0$. Consider the case $x_2+4\leq \lambda_0$. Then 
$\frac{1}{2+x_2}+\frac{x_2}{2x_2+1}\leq \ind(Q)\leq \nu(\lambda_0)\leq \nu(x_2+4)$, which gives $x_2=1$ and $\lambda_0=5$. Then $Q=[3,2,1,3,3]$, so $\ind(Q)=\frac{2}{5}+\frac{3}{8}\leq \nu(5)$; a contradiction. We obtain $x_2+3=\lambda_0=\#Q$. Then $Q=[\lambda_0-1,1,(2)_{\lambda_0-3},3]$, which gives (a).
 
(c) Since $\ind(Q)\leq \frac{2}{3}$, we have $\lambda_0:=\lambda(Q)\in \{6,7\}$ by (a). Suppose that $Q$ is branched. We have $\ind(Q)\leq \frac{2}{3}\leq \nu(6-1)$, so by part (b), $\lambda_0=7$. Let $Q'$ be as in the proof of (b). Then $\ind(Q')\geq \nu(6)=\frac{2}{3}-\frac{1}{45}$, so if $T$ is a tip of $Q$ not contained in the proper transform of $Q'$ then $\frac{1}{-T^2}\leq \ind(Q)-\ind(Q')\leq \frac{1}{45}$, i.e.\ $T^2\leq -45$. But creating such a component requires at least $45$ blow-ups and, since we easily prove by induction that $b_0(\Delta(Q))\leq \lambda(Q)+1$, we have $\#Q\leq 2\lambda_0+1\leq 15$; a contradiction. It follows that $Q$ is a chain. Using the assumption $\ind(Q)\leq \frac{2}{3}$, which for $\lambda_0=6,7$ is only slightly weaker than $\ind(Q)\leq \nu(\lambda_0)$, we argue as in the proof of (a) that $x_1=2$ and that either $x_2=\lambda_0-3$ and $\lambda_0\in \{6,7\}$ or $x_2=\lambda_0-4$ and $\lambda_0=7$. In the first case we get  (i) and (ii) above and in the latter case we get (iii).
\end{proof}

For $Q$ as in Lemma \ref{lem:lambda_bounds} we denote by $\Delta^-(Q)$ the sum of those maximal $(-2)$-twigs of $Q$ which are disjoint from the unique $(-1)$-curve of $Q$.

\begin{lem}\label{lem:divisor_list_lambda} Let $Q$ be as in Lemma \ref{lem:lambda_bounds} and let $\lambda:=\lambda(Q)\leq 7$. In case $\lambda=6$ assume that $\ind(Q)\leq \frac{3}{2}$ and in case $\lambda=7$ that $\Delta^-(Q)=0$. Then the sequence of (standard) HN-pairs of $Q$ belongs to the following list (where $h$ denotes the length of the sequence):
\begin{enumerate}[(a)]
\item $\lambda=1$ : $\binom{3}{2}$,
\item $\lambda=2$ : $\binom{5}{2}$,
\item $\lambda=3$ : 
\begin{enumerate}[(i)]
\item[$h=1$]: $\binom{4}{3}$,  $\binom{7}{2}$, $\binom{7}{3}$,
\item[$h=2$]: $\binom{6}{4}\binom{2}{1}$, $\binom{10}{4}\binom{2}{1}$,
\end{enumerate} 
\item $\lambda=4$ :
\begin{enumerate}[(i)]
\item[$h=1$]: $\binom{5}{3}$, $\binom{5}{4}$, $\binom{7}{5}$, $\binom{8}{3}$, $\binom{9}{2}$, $\binom{9}{4}$, $\binom{10}{3}$, $\binom{12}{5}$,
\item[$h=2$]: $\binom{6}{4}\binom{2}{3}$, $\binom{10}{4}\binom{2}{3}$, $\binom{9}{6}\binom{3}{1}$, $\binom{15}{6}\binom{3}{1}$, $\binom{8}{6}\binom{2}{1}$,  $\binom{14}{4}\binom{2}{1}$, $\binom{14}{6}\binom{2}{1}$
\item[$h=3$]: $\binom{12}{8}\binom{4}{2}\binom{2}{1}$, $\binom{20}{8}\binom{4}{2}\binom{2}{1}$.
\end{enumerate} 
\item $\lambda=5$ :
\begin{enumerate}[(i)]
\item[$h=1$]: $\binom{7}{4}$, $\binom{8}{5}$, $ \binom{6}{5}$, $ \binom{9}{7}$, $ \binom{10}{7}$, $ \binom{11}{3}$, $ \binom{11}{4}$, $ \binom{11}{8}$, $ \binom{13}{5}$, $ \binom{11}{2}$, $ \binom{11}{5}$, $ \binom{13}{3}$, $ \binom{13}{4}$, $ \binom{16}{7}$, $ \binom{17}{5}$,~$ \binom{17}{7}$,~$\binom{19}{8}$,

\item[$h=2$]: $\binom{6}{4}\binom{2}{5},$ $ \binom{10}{4} \binom{2}{5},$ $ \binom{9}{6} \binom{3}{2},$ $ \binom{15}{6} \binom{3}{2},$ $ \binom{9}{6} \binom{3}{4},$ $ \binom{15}{6} \binom{3}{4},$ $ \binom{12}{8}  \binom{4}{1},$ $ \binom{20}{8} \binom{4}{1},$ $ \binom{15}{10}  \binom{5}{2},$ $ \binom{25}{10} \binom{5}{2},$ $ \binom{8}{6}  \binom{2}{3},$ 
$\binom{14}{4} \binom{2}{3},$ $ \binom{14}{6}  \binom{2}{3},$ $ \binom{12}{9} \binom{3}{1},$ $ \binom{21}{6}  \binom{3}{1},$ $ \binom{21}{9} \binom{3}{1},$ $\binom{10}{6}  \binom{2}{1},$ $ \binom{10}{8} \binom{2}{1},$ $ \binom{14}{10}  \binom{2}{1},$ $ \binom{16}{6} \binom{2}{1},$ $ \binom{18}{4}  \binom{2}{1},$ $ \binom{18}{8} \binom{2}{1},$ $\binom{20}{6}  \binom{2}{1},$ $ \binom{24}{10} \binom{2}{1}$,

\item[$h=3$]: $\binom{12}{8} \binom{4}{2} \binom{2}{3},$ $ \binom{20}{8} \binom{4}{2} \binom{2}{3},$ $ \binom{18}{12} \binom{6}{3} \binom{3}{1},$ $ \binom{30}{12} \binom{6}{3} \binom{3}{1},$ $\binom{12}{8} \binom{4}{6}  \binom{2}{1},$ $\binom{20}{8} \binom{4}{6} \binom{2}{1},$ $\binom{18}{12} \binom{6}{2}  \binom{2}{1},$ $ \binom{30}{12} \binom{6}{2}  \binom{2}{1},$ $ \binom{16}{12} \binom{4}{2}  \binom{2}{1},$ $ \binom{28}{8} \binom{4}{2}  \binom{2}{1},$ $ \binom{28}{12} \binom{4}{2} \binom{2}{1}$,

\item[$h=4$]: $\binom{24}{16} \binom{8}{4} \binom{4}{2}  \binom{2}{1},$ $ \binom{40}{16}\binom{8}{4} \binom{4}{2} \binom{2}{1}$,
 \end{enumerate}
\item $\lambda=6$, $\ind(Q)\leq \frac{3}{2}$:
\begin{enumerate}[(i)]
\item[$h=1$]: $\binom{9}{5},$ $\binom{11}{7},$ $\binom{12}{7},$ $\binom{13}{8},$ $\binom{7}{6},$ $\binom{11}{9},$ $\binom{13}{9},$ $\binom{13}{10},$ $\binom{14}{3},$ $\binom{14}{5},$ $\binom{14}{11},$ $\binom{15}{4},$ $\binom{15}{11},$ $\binom{17}{12},$ $\binom{18}{5},$ $\binom{18}{7},$ $\binom{18}{13},$ $\binom{19}{7},$ $\binom{21}{8},$ $\binom{13}{2},$ $\binom{13}{6},$ $\binom{16}{3},$ $\binom{16}{5},$ $\binom{20}{9},$ $\binom{22}{5},$ $\binom{22}{9},$ $\binom{23}{7},$ $\binom{23}{10},$ $\binom{24}{7},$ $\binom{25}{11},$ $\binom{26}{11},$ $\binom{27}{8},$ $\binom{29}{12},$ $\binom{31}{13}, $

\item[$h=2$]: $\binom{6}{4}\binom{2}{7},$ $\binom{9}{6} \binom{3}{5},$ $\binom{15}{6} \binom{3}{5},$ $\binom{9}{6} \binom{3}{7},$ $\binom{12}{8} \binom{4}{3},$ $\binom{20}{8} \binom{4}{3},$ $\binom{15}{10} \binom{5}{3},$ $\binom{25}{10} \binom{5}{3},$ $\binom{15}{10} \binom{5}{7},$ $\binom{21}{14} \binom{7}{3},$ $\binom{24}{16} \binom{8}{3},$ $\binom{8}{6} \binom{2}{5},$ $\binom{12}{9} \binom{3}{2},$ $\binom{10}{6} \binom{2}{3},$ $\binom{10}{8} \binom{2}{3},$ $\binom{14}{10} \binom{2}{3},$ $\binom{16}{6} \binom{2}{3},$ $\binom{15}{9} \binom{3}{1},$ $\binom{14}{8} \binom{2}{1},$ $\binom{16}{10} \binom{2}{1},$ $\binom{12}{10} \binom{2}{1},$ $\binom{18}{14} \binom{2}{1},$ $\binom{20}{14} \binom{2}{1},$ $\binom{22}{8} \binom{2}{1},$ $\binom{22}{16} \binom{2}{1}$,
\end{enumerate}

\item $\lambda=7$, $\Delta^-(Q)=0$: 
\begin{enumerate}[(i)]
\item[$h=1$]: $\binom{11}{6},$ $\binom{14}{9},$ $\binom{16}{9},$ $\binom{17}{10},$ $\binom{18}{11},$ $\binom{19}{11},$ $\binom{19}{12},$ $\binom{21}{13},$ $\binom{8}{7},$ $\binom{15}{2},$ 

\item[$h=2$]: $\binom{10}{6} \binom{2}{5},$ $\binom{15}{9} \binom{3}{2},$ $\binom{15}{9}\binom{3}{4},$ $\binom{20}{12}\binom{4}{1},$ $\binom{14}{8}\binom{2}{3},$ $\binom{16}{10} \binom{2}{3},$ $\binom{21}{12}\binom{3}{1},$ $\binom{24}{15} \binom{3}{1},$ $\binom{18}{10}\binom{2}{1},$ $\binom{22}{14} \binom{2}{1},$ $\binom{24}{14}\binom{2}{1},$ $\binom{26}{16} \binom{2}{1}$.
\end{enumerate}
\end{enumerate}
\end{lem}

\begin{proof} We explain how to obtain the list. First of all, it is easy to create inductively a list of HN-pairs $\binom{c}{p}$ with a given number of components of the chain $|\binom{c}{p}|$, cf.\ \eqref{eq:|(c,p)|}. Indeed, $\#|\binom{c}{p}|=k$ if and only if the Euclidean algorithm for $(c,p)$ reaches $0$ after exactly $k$ steps. It follows that having a complete list of pairs with $\#|\binom{c}{p}|=k$ and replacing each $\binom{c}{p}$ with $\{\binom{c+p}{p}, \binom{c+p}{c}\}$ we obtain a complete list of pairs with $\#|\binom{c}{p}|=k+1$. By Definition \ref{def:HN} the HN-sequence for a divisor $Q$ with $h\geq 2$ is obtained from the HN-sequence of the divisor $\sigma(Q)$ and the HN-pair of $\Exc \sigma$, which allows to list sequences of standard HN-pairs for all divisors with a given number of components. Finally, for $Q\neq [3,1,2]$ we have $h+1\leq \lambda(Q)$, so $\#Q-\lambda(Q)=b_0(\Delta(Q))\leq h+1\leq \lambda(Q)$, hence $Q$ can be found among divisors with at most $2\lambda(Q)$ components. Conversely, given a sequence of HN-pairs of some $Q$ we may compute $\lambda(Q)$, $\ind(Q)$, $b_0(\Delta(Q))$ and $s(Q)$ using formulas from Lemma \ref{lem:ind_lambda_using_HN}(d).
\end{proof}

\begin{rem} If a chain $Q$ with a standard HN-pair $\binom{c}{p}$ is blown up in a node belonging to the unique $(-1)$-curve then the HN-pair of the resulting chain is not necessarily one of $\binom{c+p}{c}$ or $\binom{c+p}{p}$. For instance, for $\binom{5}{3}$ the new HN-pairs are $\binom{8}{5}$ and $\binom{7}{4}$. The geometry behind  the change $\binom{c}{p}\mapsto \binom{c+p}{p}$ is that to $Q$ we add a $(-2)$-tip at the side with discriminant $c$. Similarly, the change $\binom{c}{p}\mapsto \binom{c+p}{c}$ is related to adding a $(-q)$-tip at the side with discriminant $p$, where $q=1+\lceil c/p\rceil $.
\end{rem}

\smallskip 
\subsection{Further bounds on the geometry of the boundary}

In this section we work more with the minimal log resolution $\pi\:(X,D)\to (\P^2,\bar E)$ instead of the minimal weak resolution $\pi_0\:(X_0,D_0)\to (\P^2,\bar E)$. We denote by $Q_j'$ the exceptional divisor of $\pi$ over the cusp $q_j\in \bar E$, $j=1,\ldots,c$. We define the sequence of standard HN-pairs of $q_j\in \bar E$ as the sequence of standard HN-pairs of $Q_j'$ and we decorate the respective quantities $c^{(k)}$, $p^{(k)}$ and $h$ with lower index $j\in \{1,\ldots,c \}$.  Clearly, $\Delta(Q_j')=Q_j'\wedge \Delta_0'$, $s_j=s(Q_j')$ and $\ind_j=\ind(Q_j')$. For cusps of multiplicity bigger than $2$ we have $(\Upsilon_0+\Delta_0^+)\wedge Q_j=0$, so by \eqref{eq:lambda_def}
\begin{equation}\label{eq:lambda_rel_to_Q'}
\lambda_j=\#Q_j-b_0(Q_j\wedge \Delta_0^-)+\tau_j-s_j=\#Q_j'-b_0(Q_j'\wedge \Delta_0'),
\end{equation} hence by \eqref{eq:def_s,lambda,M,I} for all cusps  $\lambda(Q_j')=\lambda_j$. We put (cf.\ \eqref{eq:def_s,lambda,M,I})
\begin{equation*}
r_j=r(Q_j'), \ \ M_j=M(Q_j') \text{\ \ and\ \ } I_j=I(Q_j').
\end{equation*}
The numbers $M_j$ and $I_j$ compute the sums of multiplicities and, respectively, squares of multiplicities, of successive proper transforms of $q_j\in\bar E$ under blow-ups constituting $\pi^{-1}$ (see \cite[Corollary 3.2]{Palka-Coolidge_Nagata1}), so the genus-degree formula takes the form
\begin{equation}\label{eq:genus}
(\deg \bar E-1)(\deg \bar E-2)=\sum_{j=1}^c (I_j-M_j).
\end{equation}

Proposition \ref{prop:c<=5} gives $c\leq 5$, hence we may and will assume that $c=5$ or $c=4$.

\begin{cor}\label{cor:omega_ind} For a positive integer $k$ let $\w_k$ denote the number of cusps of $\bar E$ for which $\lambda_j=k$. Then $\sum_{k=1}^{\8} k\w_{k}=c$ and the following inequality holds:
\begin{equation}\label{eq:main_ind_2}
\frac{5}{6}\w_1+\frac{11}{10}\w_2+\frac{11}{12}\w_3+\frac{11}{15}\w_4+\frac{19}{28}\w_5+\frac{29}{45}\w_6+\frac{1}{2}\sum_{k=7}^\8 \w_k\leq \sum_{j=1}^c \ind_j.
\end{equation}
\end{cor}

\begin{proof}We have $\ind_j>\frac{1}{2}$ for any cusp by Lemma \ref{lem:log BMY}(a). Using the bounds from Lemmas \ref{lem:lambda_bounds}, \ref{prop:lambda=3} and \ref{lem:basic_bounds}(a)(e) we get the required inequality.
\end{proof}

\begin{lem}\label{lem:n=0_new} Let the notation be as above and let $\delta=\delta(\Delta_0^-)$.  We have:
\begin{enumerate}[(a)]
\item $\deg \bar E=\frac{3}{2}+\frac{1}{2}\sqrt{1+4\sum_{j=1}^c(I_j-M_j)}$,
\item $p_2=9-2c+\sum_{j=1}^c(M_j-r_j)-3\deg \bar E$,
\item $\sum_{j=1}^c\ind_j\leq 5-p_2$,
\item $\delta\leq 7+3p_2-\sum_{j=1}^c\lambda_j$.
\end{enumerate}
\end{lem}

\begin{proof}
Part (a) follows from \eqref{eq:genus}. By \cite[Corollary 3.3(i)]{Palka-Coolidge_Nagata1}, $-E^2=-3\deg \bar E+2+\sum_{j=1}^cM_j$. On the other hand, by \eqref{eq:p2_formula}, $p_2=K\cdot (K+D)=-E^2+7-2c-\sum_{j=1}^cr_j$, so we obtain (b). Parts (c) and (d) are \eqref{eq:ind} and \eqref{eq:lambda}, respectively.
\end{proof}

\begin{rem}\label{rem:bounds} From the above bounds we can derive further ones. We will not use them, but we discuss them to convince the reader that, given Theorem \ref{thm:almost_minimal}, the situation is numerically completely bounded. From the definition of $\lambda_j$ we have $\lambda_1+\cdots+\lambda_c=\#D-1-b_0(\Delta_0')-c_0'$, hence $\#D\leq 3p_2+8+c_0'+b_0(\Delta_0')-\delta$ by Lemma \ref{lem:n=0_new}(d), where $c_0'$ is the number of ordinary cusps ($c_0'=\w_1$). Since $\ind(D)\geq \frac{1}{2}b_0(\Delta_0')+\frac{1}{3}c_0'$ by Lemma \ref{lem:log BMY}(c), Lemma \ref{lem:n=0_new}(c) gives $$b_0(\Delta_0')\leq 10-\frac{2}{3}c_0'-2p_2$$ and hence $\#D\leq p_2+18+\lfloor \frac{1}{3}\w1-\delta\rfloor\leq 21$, as $p_2\leq 2$ by Lemma \ref{lem:basic_bounds}(f). It follows that $\#Q_j'\leq 20-3\cdot 3=11$. Since $\#Q_j'\neq 1$, we have $r_j\geq 1$, so by \eqref{eq:KQ_inner_bl} $$K\cdot Q_j'+3\leq \#Q_j'\leq 11.$$ From \cite[Lemma 2.4(ii)]{KoPa-CooligeNagata2} we obtain $\mu(q_j)\leq F_{K\cdot Q_j'+3}\leq F_{\#Q_j'}\leq F_{11}=89$, where $F_k$ denotes the $k$'th Fibonacci number ($F_1=F_2=1$). Then by the Matsuoka-Sakai inequality \cite{MaSa-degree_of_cusp_curves} $$\deg \bar E\leq 3\max\{\mu(q_1),\ldots,\mu(q_c)\}-1\leq 266.$$
\end{rem}

\begin{prop}\label{prop:not_5} Let $\bar E\subseteq \P^2$ be a rational cuspidal curve with $c\geq 4$ cusps. Put $\Lambda=(\lambda_1,\ldots,\lambda_c)$ with $\lambda_1\geq \ldots\geq \lambda_c$. Then we have $c=4$, $p_2\leq 1$ and $h_1\leq 2$. Moreover, if $\lambda_1\geq 5$ then $p_2=1$ and one of the following holds:
\begin{enumerate}[(1)]
\item $\Lambda=(7,1,1,1)$, $\delta=0$, $\ind_1\leq \frac{3}{2}$,
\item $\Lambda=(6,2,1,1)$, $\delta=0$, $\ind_1\leq \frac{37}{30}$,
\item $\Lambda=(5,3,1,1)$, $\delta=0$, $\ind_1\leq \frac{17}{12}$,
\item $\Lambda=(5,2,2,1)$, $\delta=0$, $\ind_1\leq \frac{29}{30}$,
\item $\Lambda=(6,1,1,1)$, $\delta\leq 1$, $\ind_1\leq \frac{3}{2}$,
\item $\Lambda=(5,2,1,1)$, $\delta\leq 1$, $\ind_1\leq \frac{37}{30}$,
\item $\Lambda=(5,1,1,1)$, $\delta\leq 2$, $\ind_1\leq \frac{3}{2}$,
\end{enumerate}
\end{prop}

\begin{proof}
By Proposition \ref{prop:c<=5}, $c\leq 5$. Suppose that $c=5$. Elementary calculations show that the only cases satisfying Corollary \ref{cor:omega_ind} and Lemma \ref{lem:n=0_new}(c),(d) are:
\begin{enumerate}[(a)]
\item $p_2=1$, $\Lambda=(6,1,1,1,1)$,
\item $p_2=0$, $\Lambda=(3,1,1,1,1)$,
\item $p_2=0$, $\Lambda=((2)_{\w_2},(1)_{5-\w_2})$, where $\w_2\in\{0,1,2\}$.
\end{enumerate}
Put $d=\deg \bar E$. In cases (a) and (b) we have $|\Lambda|-3p_2=7$, where $|\Lambda|=\lambda_1+\cdots+\lambda_c$, so by Lemma \ref{lem:n=0_new}(d), $\Delta_0^-=0$. In case (b) Proposition \ref{prop:lambda=3} implies that either $q_1$ is semi-ordinary or $Q_1=C_1$, hence $q_1$ is described by one of the HN-pairs $\binom{7}{2}$ or $\binom{4}{3}$. In both cases \eqref{eq:genus} gives $(d-1)(d-2)=14$, which is impossible. In case (a) Lemma \ref{lem:n=0_new}(c) gives $\ind_1\leq \frac{2}{3}$, so by Lemma \ref{lem:lambda_bounds}(c), $q_1$ is described by the HN-pair $\binom{9}{5}$, hence $(d-1)(d-2)=40$, which is impossible.

In case (c) the minimal log resolution is described by $\w_2$ pairs $\binom{5}{2}$ and $5-\w_2$ pairs $\binom{3}{2}$. We compute $\sum_{j=1}^5M_j=6\w_2+4(5-\w_2)=2\w_2+20$ and $\sum_{j=1}^5I_j=10\w_2+6(5-\w_2)=4\w_2+30$, hence by \eqref{eq:genus}, $(d-1)(d-2)=2\w_2+10$, which gives $\w_2=1$ and $d=5$. But after blowing up $q_1$ once we get a $\P^1$-fibration for which the proper transform of $E$ meets a general fiber three times. Then the Hurwitz formula, Lemma \ref{lem:Hurwitz}, gives $c\leq 2\cdot 3-2=4$; a contradiction.

Thus $c=4$. Suppose that $p_2=2$. From Corollary \ref{cor:omega_ind} and Lemma \ref{lem:n=0_new}(c),(d) we have $\Lambda=(6,5,1,1)$ or $\Lambda=(7,4,1,1)$ or $\Lambda=(\lambda_1,1,1,1)$ for some $\lambda_1\in \{7,8,9,10\}$. But in the latter case $\ind_1+3\cdot \frac{5}{6}\leq 5-p_2=3$, hence $\ind_1\leq \frac{1}{2}$, which is impossible by Lemma \ref{lem:log BMY}(a). By Lemma \ref{lem:n=0_new}(c) we have $\ind_1+\ind_2\leq 3-2\cdot \frac{5}{6}=\frac{4}{3}$. Then Lemma \ref{lem:lambda_bounds} gives a contradiction in the second case and in the first case gives $Q_1=[5,1,2,2,2,3]$ and hence $\ind_2\leq \frac{2}{3}+\frac{1}{45}$, which results with $\ind_2<\nu(5)+\frac{1}{96}$. Under the latter inequality, which is only slightly weaker than $\ind(Q_2')\leq \nu(5)$, the arguments in the proof of Lemma \ref{lem:lambda_bounds}(a) show that $Q_2=[4,1,2,2,3]$. Then \eqref{eq:genus} gives $(d-1)(d-2)=54$; a contradiction. 

Thus $p_2\leq 1$. Suppose that $h_1\geq 3$ and $\delta=0$. Then $q_1\in \bar E$ is not semi-ordinary and $b_0(Q_1'\wedge \Delta_0')=s_1\leq 1$, so $\#Q_1'=\lambda_1+s_1\leq 8$. Since $\delta=0$, the first two HN-pairs of $q_1$ cannot create a $(-2)$-tip, so they contain at least $4$ and $3$ components, respectively, hence $\#Q_1'\geq 4+3+2=9$; a contradiction.

Suppose that $h_1\geq 3$ and $\delta>0$. Since $h_1\geq 3$, we have $\lambda_1\geq 4$ by Proposition \ref{prop:lambda=3}. By Lemma \ref{lem:n=0_new}(d), $p_2=1$, so $\ind_1+\ldots+\ind_4\leq 4$ and $|\Lambda|\leq 9$, hence $\lambda_1\leq 6$ and $\lambda_2\leq 3$. Then $\ind_j\geq \frac{5}{6}$ for $j\in\{2,3,4\}$, so $\ind_1\leq \frac{3}{2}$. But using Lemma \ref{lem:divisor_list_lambda} we check that for $\lambda_1\in\{4,5,6\}$ the latter inequality fails in all cases with $h_1\geq 3$; a contradiction.

Thus $h_1\leq 2$. The cases listed in the statement of the proposition are the only ones with $\lambda_1\geq 5$ which satisfy Lemma \ref{lem:n=0_new}(d). The bound on $\ind_1$ follows from Lemma \ref{lem:n=0_new}(c).
\end{proof}

\begin{ex}\label{ex:Q4} 
Consider the quintic
\begin{equation}
\Qa=\{27 x^5+18 x^3 y z-2 x^2 y^3-x^2 z^3+2 x y^2 z^2-y^4 z=0\}\subseteq \P^2.
\end{equation} It has parameterization $\varphi\:\P^1\to \P^2$: 
\begin{equation}
\varphi([s:t])=[s^4t:s^2t^3-s^5:t^5+2s^3t^2].
\end{equation}
Its singular points are cusps $q_1=\varphi([0:1])$, $q_j=\varphi([\sqrt[3]{2}\theta^{j-2}:-1])$, $j=2,3,4$, where $\theta$ is a primitive third root of unity. The local analytic type of the cusp $q_1$ is $x^2=y^7$ and the local analytic type of the cusps $q_j$, $j=2,3,4$ is $x^2=y^3$.
\end{ex}

 %\smallskip

 \begin{proof}[Proof of Theorem \ref{thm:at_most_4_cusps}]
By Proposition \ref{prop:not_5}, $p_2\leq 1$ and either $\lambda_1\leq 4$ or $p_2=1$ and $(\lambda_1,\lambda_2,\lambda_3,\lambda_4)$ is one of the $7$ sequences listed there. In Lemma \ref{lem:divisor_list_lambda} we have a complete list of possible sequences of HN-pairs with small $\lambda$. Using \eqref{eq:def_s,lambda,M,I} and Lemma \ref{lem:ind_lambda_using_HN} in each case we may compute $M_j$, $I_j$, $\ind_j$, $\lambda_j$, $r_j$, $b_0(Q_j'\wedge \Delta_0')$, $s_j$ and $\displaystyle b_0(\Delta_0^-)=\sum_{j:\mu(q_j>2)}(b_0(Q_j'\wedge \Delta_0')-s_j)$, which then allows to compute subsequently $d:=\deg \bar E$ and $p_2$ as in Lemma \ref{lem:n=0_new}(a),(b). We then check whether the so-computed $p_2$ belongs to $\{0,1\}$ and whether the inequality in Lemma \ref{lem:n=0_new}(c) and the implication $b_0(\Delta_0^-)\neq 0 \Rightarrow 7+3p_2>\sum_{j=1}^4\lambda_j$, a consequence of the inequality  in Lemma \ref{lem:n=0_new}(d), hold. The only cases where this is true are:
\begin{enumerate}[(a)]
\item $\binom{5}{2}$, $\binom{5}{2}$, $\binom{3}{2}$, $\binom{3}{2}$, $(d,p_2)=(5,0)$, 
\item $\binom{7}{2}$, $\binom{3}{2}$, $\binom{3}{2}$, $\binom{3}{2}$, $(d,p_2)=(5,0)$,  
\item $\binom{7}{2}$, $\binom{4}{3}$, $\binom{4}{3}$, $\binom{3}{2}$, $(d,p_2)=(6,1)$, 
\item $\binom{5}{3}$, $\binom{7}{2}$, $\binom{5}{2}$, $\binom{3}{2}$, $(d,p_2)=(6,1)$,
\item $\binom{8}{3}$, $\binom{7}{3}$, $\binom{3}{2}$, $\binom{3}{2}$, $(d,p_2)=(7,1)$,  
\item $\binom{9}{2}$, $\binom{5}{3}$, $\binom{3}{2}$, $\binom{3}{2}$, $(d,p_2)=(6,1)$,  
\item $\binom{11}{2}$, $\binom{4}{3}$, $\binom{3}{2}$, $\binom{3}{2}$, $(d,p_2)=(6,1)$,  
\item $\binom{15}{2}$, $\binom{3}{2}$, $\binom{3}{2}$, $\binom{3}{2}$, $(d,p_2)=(6,1)$, 
\item $\binom{4}{3}$, $\binom{3}{2}$, $\binom{3}{2}$, $\binom{3}{2}$, $(d,p_2)=(5,0)$, 
\item $\binom{4}{3}$, $\binom{4}{3}$, $\binom{4}{3}$, $\binom{3}{2}$, $(d,p_2)=(6,1)$,  
\item $\binom{5}{3}$, $\binom{5}{2}$, $\binom{4}{3}$, $\binom{3}{2}$, $(d,p_2)=(6,1)$,  
\item $\binom{5}{3}$, $\binom{5}{3}$, $\binom{3}{2}$, $\binom{3}{2}$, $(d,p_2)=(6,1)$.  
\end{enumerate}
Let $(\mu_j,\mu_j',\ldots )$ denote the multiplicity sequence of $q_j\in \bar E$. By the Hurwitz formula applied to the $\P^1$-fibration of the blow-up of $\P^2$ at $q_1$ we have $\sum_{j=2}^4(\mu(q_j)-1)+(\mu_1'-1)\leq 2(d-\mu_1-1)$. This inequality fails in cases (i), (j), (k), (l).

In case (h) all cusps are semi-ordinary, so the Matsuoka-Sakai inequality \cite{MaSa-degree_of_cusp_curves}, $d<3\max(\mu_1,\mu_2,\mu_3,\mu_4)$, fails.

In cases (c), (d), (e) and (f) we find a $(-1)$-curve $A\nsubseteq D$ on $X$ such that $A\cdot D=2$. Assume that such a curve exists. Since $X\setminus D$ contains no affine line, $A$ meets $D$ in two different points and $X\setminus (D+A)$ contains no affine line, so possibly after some number of inner contractions in $D+A$ the surface $(X,D+A)$ becomes almost minimal (see \cite[6.20, 6.24]{Fujita-noncomplete_surfaces}, \cite[\S 2.3.11]{Miyan-OpenSurf}). Since the contractions do not affect maximal twigs of $D+A$ and the self-intersection of the log canonical divisor, the log BMY inequality gives $(K+D+A)^2+\ind(D+A)\leq 3e_{top}(X\setminus~(D+A))$. We have $e_{top}(X\setminus (D+A))=e_{top}(X\setminus D)=1$ and $(K+D+A)^2=(K+D)^2+1=p_2-1$. Since in the cases considered $p_2=1$, we get 
\begin{equation*}\label{eq:ind(D+A)}
\ind(D+A)\leq 3.
\end{equation*}
Let $L_{i,j}\subseteq X$ for $i\neq j$ denote the proper transform of the line on $\P^2$ joining the cusps $q_i$ and $q_j$, and let $L_i\subseteq X$ denote the proper transform of the line tangent to $q_i\in \bar E$. 

In case (c) we take $A=L_{2,3}$. Since $\mu_2+\mu_3=6=d$, $A$ meets the $(-4)$-tips of $Q_2'$ and $Q_3'$ and it is a $(-1)$-curve with $A\cdot D=2$. We compute $\ind(D+A)=(\frac{5}{7}+\frac{1}{2})+\frac{2}{3}+\frac{2}{3}+\frac{5}{6} =\frac{71}{21}>3$; a contradiction.

In case (d) we take $A=L_1$. We have $(\pi(A)\cdot \bar E)_{q_1}\geq \mu_1+\mu_1'=5=d-1$, so since $X\setminus D$ contains no affine line, the equality holds, so $A$ is as needed. We compute $\ind(D+A)=\frac{2}{5}+(\frac{5}{7}+\frac{1}{2})+(\frac{3}{5}+\frac{1}{2})+\frac{5}{6}=3\frac{23}{42}>3$; a contradiction.

In case (e) we take $A=L_1$. We have $(\pi(A)\cdot \bar E)_{q_1}\geq \mu_1+\mu_1'=6=d-1$, so again equality holds and $A$ is as required. We compute $\ind(D+A)=(\frac{1}{2}+\frac{1}{3})+(\frac{4}{7}+\frac{2}{3})+\frac{5}{6}+\frac{5}{6}=3\frac{31}{42}>3$; a contradiction.

In case (f) we take $A=L_2$. Similarly, $(\pi(A)\cdot \bar E)_{q_2}=d-1$, so $A$ is as required. We have 
$\ind(D+A)=(\frac{7}{9}+\frac{1}{2})+\frac{2}{5}+\frac{5}{6}+\frac{5}{6}=3\frac{31}{90}>3$; a contradiction.

Consider case (a).   %Ruled out also by \cite{Namba_geometry_of_curves} or \cite{FlZa_cusps_d-3}. Independent argument is useful.
We note as above that $(\pi(L_j)\cdot \bar E)_{q_j}=4$ and $L_{1,2}\cdot Q_j'=1$ for $j=1,2$. Blow up twice over each $q_j$ for $j=1,2$ and denote the first exceptional curve by $U_j$. Let $\ell_j$ be the proper transform of $L_j$ and let $\ell_{1,2}$ be the proper transform of $L_{1,2}$. The divisors $\ell_1+\ell_2$ and $U_1+2\ell_{1,2}+U_2$ are complete fibers of the same $\P^1$-fibration. The proper transform of $\bar E$ is a $2$-section and has a unique point of intersection with $\ell_{1,2}$. The latter is a third, besides $q_3$ and $q_4$, ramification point of the projection of the $2$-section onto the base of the fibration, which contradicts the Hurwitz formula.

Consider case (g).  %Ruled out also by \cite{FlZa_cusps_d-3}. An independent argument is useful.
We have $\rho(X)=\#D=18$, so $K^2=-8$ and then $K\cdot D=p_2-K^2=9$ and $K\cdot E=9-K\cdot Q_2'=8$. Put $B=D-C_1'-C_2'$, where $C_j'$ is the unique $(-1)$-curve in $Q_j'$, and denote by $U$ the $(-4)$-curve which is a tip of $Q_2'$. We analyze the surface $(X,B)$. First of all we note that on $X_0$ the divisor $2K_0+\varphi_0(U)+E_0$ intersects trivially all components of $D_0$, hence is numerically trivial, because the components of $D_0$ generate $\NS(X_0)\otimes \Q$. This gives a numerical equivalence
\begin{equation}\label{eq:B}
2K+U+E\equiv \sum_{j\in \{1,3,4\}}(2C_j'+G_j),
\end{equation}
where $G_j\subseteq Q_j'$ for $j\in\{1,3,4\}$ is the $(-2)$-curve meeting $C_j'$. It follows that $\kappa(K_X+B)\geq 0$. We claim that $(X,B)$ is almost minimal. Suppose otherwise. By \cite[6.20]{Fujita-noncomplete_surfaces} there exists a $(-1)$-curve $L\nsubseteq B$ meeting at most two connected components of $B$ (hence $L\nsubseteq D$), each at most once. Intersecting \eqref{eq:B} with $L$ we have $L\cdot (2K+U+E)\geq 0$, hence $L\cdot (U+E)\geq 2$. It follows that $L\cdot U=L\cdot E=1$ and $L\cdot (D-C_1'-C_2'-U-E)=0$. Let $(X',B')$ be the image of $(X,B)$ after the contraction of $L$. Repeating the argument we see that $(X',B')$ is almost minimal and $\kappa(K_{X'}+B')\geq 0$. Then $B'$ has $4$ connected components, one of which has $5$ maximal twigs: $[3]$, $[2]$, $[3]$, $[2]$, $[3]$. But $e_{top}(X'\setminus B')=e_{top}(X\setminus D)-2=-1$, so the log BMY inequality (\cite[2.6.6.2]{Miyan-OpenSurf}) gives $0\leq -1+\frac{1}{c_1^{(1)}} +\frac{1}{p_1^{(1)}}+ \frac{1}{p_2^{(1)}} =-1+\frac{1}{11}+\frac{1}{2}+\frac{1}{3}$; a contradiction. Thus $(X,B)$ is almost minimal. We have $e_{top}(X\setminus B)=-1$, $K\cdot (K+B)=K\cdot (K+D)+2=p_2+2=3$ and $B\cdot (K+B)=-2b_0(B)=-10$. For an admissible chain $T$ which is a connected component of $B$ we can consider two orders $T=T_1+\ldots+T_m$ and $T^t=T_m+\ldots+T_1$ as in Section \ref{sec:prelim} and we put $\Bk T=\Bk' T+\Bk' T^t$, cf.\ \cite[\S 2.3.3.3]{Miyan-OpenSurf}, which gives $$\ind(T):=-\Bk^2T=\frac{d(T-T_1)+d(T-T_m)+2}{d(T)}.$$ We compute $\ind(Q_1'-C_1')=\frac{9+5+2}{11}+\frac{1+1+2}{2}=3\frac{5}{11}$, $\ind(Q_2'-C_1')=\frac{1+1+2}{4}+\frac{2+2+2}{3}=3$ and $\ind(E+Q_3'+Q_4')=2\cdot \frac{5}{6}=\frac{5}{3}$, hence $\ind(B)=8\frac{4}{33}$ and $(K+B-\Bk B)^2=(K+B)^2+\ind(B)=1\frac{4}{33}$. The log BMY inequality for $(X,B)$ gives $1\frac{4}{33}\leq 3\cdot (-1+\frac{1}{11}+\frac{1}{2}+\frac{1}{4}+\frac{1}{3})=\frac{23}{44}$; a contradiction.

Thus, we have shown that the only possibility is case (b). This is indeed the type of the quintic from Theorem \ref{thm:at_most_4_cusps}. We note that $L_1$ is a $(-1)$-curve meeting $Q_1=[1,2,2]$ in the middle  component. The contraction of $L_1+(Q_1-C_1)+Q_2+Q_3+Q_4$ transforms $E_0$, and hence $\bar E$, onto a tricuspidal quartic with all cusps ordinary and with the image of $C_1$ as a bitangent line. Using this transformation one proves the projective uniqueness of $\bar E$ of type (b), see \cite[Proposition 4.9]{PaPe_cuspidal-delPezzo}. The uniqueness follows also from the classification of rational cuspidal quintics in \cite{Namba_geometry_of_curves}. 
\end{proof}

 %\medskip

\bibliographystyle{amsalpha}
\bibliography{C:/KAROL/PRACA/PUBLIKACJE/BIBL/bibl2019}

\newcommand{\etalchar}[1]{$^{#1}$}
\providecommand{\bysame}{\leavevmode\hbox to3em{\hrulefill}\thinspace}
\providecommand{\MR}{\relax\ifhmode\unskip\space\fi MR }
\providecommand{\MRhref}[2]{%
  \href{http://www.ams.org/mathscinet-getitem?mr=#1}{#2}
}
\providecommand{\href}[2]{#2}
\begin{thebibliography}{FdBLMHN07b}

\bibitem[ABLMH06]{BaLuMH-superisolated_sing}
Enrique Artal~Bartolo, Ignacio Luengo, and Alejandro Melle~Hern\'{a}ndez,
  \emph{Superisolated surface singularities}, Singularities and computer
  algebra, London Math. Soc. Lecture Note Ser., vol. 324, Cambridge Univ.
  Press, Cambridge, 2006, pp.~13--39.

\bibitem[AM73]{AbMo-approximate_roots}
Shreeram~S. Abhyankar and Tzuong~Tsieng Moh, \emph{Newton-{P}uiseux expansion
  and generalized {T}schirnhausen transformation. {I}, {II}}, J. Reine Angew.
  Math. \textbf{260} (1973), 47--83; ibid. 261 (1973), 29--54.

\bibitem[BBSdR16]{BarBoSe_HF_and_RCC}
Adam Baranowski, Maciej Borodzik, and Juan Serrano~de Rodrigo,
  \emph{Heegaard-{F}loer homologies and rational cuspidal curves. {L}ecture
  notes}, Winter Braids Lect. Notes \textbf{3} (2016), no.~Winter Braids VI
  (Lille, 2016), Exp. No. 1, 39.

\bibitem[BHS19]{BoHoSch-involutive_Fl_hom-_RCC}
Maciej Borodzik, Jennifer Hom, and Andrzej Schinzel, \emph{Involutive
  {H}eegaard-{F}loer homology and rational cuspidal curves}, Proceedings of the
  London Mathematical Society \textbf{118} (2019), no.~3, 441--472.

\bibitem[BL14]{BoroLivi-HeegaardFloer_and_cusps}
Maciej Borodzik and Charles Livingstone, \emph{Heegaard-{F}loer homology and
  rational cuspidal curves}, Forum Math. Sigma \textbf{2} (2014), e1, 28,
  \arxiv{1304.1062}.

\bibitem[BL16]{BoLi-semigroups_RCC}
Maciej Borodzik and Charles Livingston, \emph{Semigroups, {$d$}-invariants and
  deformations of cuspidal singular points of plane curves}, J. Lond. Math.
  Soc. (2) \textbf{93} (2016), no.~2, 439--463.

\bibitem[BN16]{BoNe-lattices}
J\'{o}zsef Bodn\'{a}r and Andr\'{a}s N\'{e}methi, \emph{Lattice cohomology and
  rational cuspidal curves}, Math. Res. Lett. \textbf{23} (2016), no.~2,
  339--375.

\bibitem[Bod16]{Bodnar}
J\'{o}zsef Bodn\'{a}r, \emph{Construction of bicuspidal rational complex
  projective plane curves}, \arxiv{1608.02921}, 2016.

\bibitem[BZ10]{BoZo-annuli}
Maciej Borodzik and Henryk \.{Z}ol\c{a}dek, \emph{Complex algebraic plane
  curves via {P}oincar\'e-{H}opf formula. {II}. {A}nnuli}, Israel J. Math.
  \textbf{175} (2010), 301--347,
  \href{http://arxiv.org/abs/0708.1661}{arXiv:0708.1661}.

\bibitem[CNKR09]{CKR-Cstar_good_asymptote}
Pierrette Cassou-Nogues, Mariusz Koras, and Peter Russell, \emph{Closed
  embeddings of {$\mathbb{C}^*$} in {$\mathbb{C}^2$}. {I}}, J. Algebra
  \textbf{322} (2009), no.~9, 2950--3002.

\bibitem[DS17]{DiSt-free-divisors}
Alexandru Dimca and Gabriel Sticlaru, \emph{Free divisors and rational cuspidal
  plane curves}, Math. Res. Lett. \textbf{24} (2017), no.~4, 1023--1042.

\bibitem[DS18a]{DiSt-free-divisors2_RCC}
\bysame, \emph{Free and nearly free curves vs. rational cuspidal plane curves},
  Publ. Res. Inst. Math. Sci. \textbf{54} (2018), no.~1, 163--179.

\bibitem[DS18b]{DiSt3-freenes_RCC}
\bysame, \emph{On the freeness of rational cuspidal plane curves}, Mosc. Math.
  J. \textbf{18} (2018), no.~4, 659--666.

\bibitem[FdBLMHN06]{FLMN_cuspidal_curves}
Javier Fern{\'a}ndez~de Bobadilla, Ignacio Luengo, Alejandro
  Melle-Hern{\'a}ndez, and Andras N{\'e}methi, \emph{On rational cuspidal
  projective plane curves}, Proceedings of the London Mathematical Society
  \textbf{92} (2006), no.~1, 99--138.

\bibitem[FdBLMHN07a]{FLMN_cusps_and_open_surfaces}
J.~Fern{\'a}ndez~de Bobadilla, I.~Luengo, A.~Melle-Hern{\'a}ndez, and
  A.~N{\'e}methi, \emph{On rational cuspidal plane curves, open surfaces and
  local singularities}, Singularity theory, World Sci. Publ., Hackensack, NJ,
  2007, pp.~411--442 (English).

\bibitem[FdBLMHN07b]{FLMN_one_pair}
Javier Fern{\'a}ndez~de Bobadilla, Ignacio Luengo, Alejandro
  Melle-Hern{\'a}ndez, and Andras N{\'e}methi, \emph{Classification of rational
  unicuspidal projective curves whose singularities have one {P}uiseux pair},
  Real and complex singularities, Trends Math., Birkh\"auser, Basel, 2007,
  pp.~31--45.

\bibitem[Fen99a]{Fenske_1and2-cuspidal_curves}
Torsten Fenske, \emph{Rational 1- and 2-cuspidal plane curves.}, Beitr\"{a}ge
  zur Algebra und Geometrie \textbf{40} (1999), no.~2, 309--329.

\bibitem[Fen99b]{Fenske_cusp_d-4}
\bysame, \emph{Rational cuspidal plane curves of type {$(d,d-4)$} with
  {$\chi(\Theta_V\langle D\rangle)\leq0$}}, Manuscripta Math. \textbf{98}
  (1999), no.~4, 511--527.

\bibitem[Fuj82]{Fujita-noncomplete_surfaces}
Takao Fujita, \emph{On the topology of noncomplete algebraic surfaces}, J. Fac.
  Sci. Univ. Tokyo Sect. IA Math. \textbf{29} (1982), no.~3, 503--566.

\bibitem[FZ96]{FLZa-_class_of_cusp}
Hubert Flenner and Mikhail Zaidenberg, \emph{On a class of rational cuspidal
  plane curves}, Manuscripta Math. \textbf{89} (1996), no.~4, 439--459.

\bibitem[FZ00]{FlZa_cusps_d-3}
\bysame, \emph{Rational cuspidal plane curves of type $(d,d-3)$.}, Math. Nachr.
  \textbf{210} (2000), 93--110.

\bibitem[GKK{\etalchar{+}}95]{Zaid-open_MONTREAL_problems}
Rajendra~Vasant Gurjar, Shulim Kaliman, N.~Mohan Kumar, Masayoshi Miyanishi,
  Peter Russell, Fumio Sakai, David Wright, and Mikhail Zaidenberg, \emph{Open
  problems on open algebraic varieties},
  \href{http://arxiv.org/abs/alg-geom/9506006}{http://arxiv.org/abs/alg-geom/9506006},
  1995.

\bibitem[GM96]{GurjarMiyanishi_AMS_and_LZ_thms}
Rajendra~Vasant Gurjar and Masayoshi Miyanishi, \emph{On contractible curves in
  the complex affine plane}, Tohoku Math. J. (2) \textbf{48} (1996), no.~3,
  459--469.

\bibitem[Kas87]{Kashiwara}
Hiroko Kashiwara, \emph{Fonctions rationnelles de type $(0,1)$ sur le plan
  projectif complexe}, Osaka J. Math. \textbf{24} (1987), no.~3, 521--577.

\bibitem[KK94]{KollarKovacs-2DlogMMP}
J.~Koll\'ar and S\'andor Kov\'acs, \emph{Birational geometry of log surfaces},
  \url{https://web.math.princeton.edu/~kollar/FromMyHomePage/BiratLogSurf.ps},
  1994.

\bibitem[KM98]{KollarMori-bir_geom}
J{\'a}nos Koll{\'a}r and Shigefumi Mori, \emph{Birational geometry of algebraic
  varieties}, Cambridge Tracts in Mathematics, vol. 134, Cambridge University
  Press, Cambridge, 1998, With the collaboration of C. H. Clemens and A. Corti,
  Translated from the 1998 Japanese original. \MR{1658959 (2000b:14018)}

\bibitem[KP17]{KoPa-CooligeNagata2}
Mariusz Koras and Karol Palka, \emph{The {C}oolidge--{N}agata conjecture}, Duke
  Math. J. \textbf{166} (2017), 1--61, \arxiv{1502.07149}.

\bibitem[KR99]{KR-C*_actions_on_C3}
Mariusz Koras and Peter Russell, \emph{{${\bf C}\sp *$}-actions on {${\bf C}\sp
  3$}: the smooth locus of the quotient is not of hyperbolic type}, J.
  Algebraic Geom. \textbf{8} (1999), no.~4, 603--694.

\bibitem[Liu14]{Liu-thesis}
Tiankai Liu, \emph{On planar rational cuspidal curves}, Ph.\ D.\ thesis,
  \href{http://hdl.handle.net/1721.1/90190}{http://hdl.handle.net/1721.1/90190},
  2014.

\bibitem[Lue87]{Luengo-superisolated_sing}
Ignacio Luengo, \emph{The {$\mu$}-constant stratum is not smooth}, Invent.
  Math. \textbf{90} (1987), no.~1, 139--152.

\bibitem[LVMHN05]{LMN-links_superisolated_sing}
I.~Luengo-Velasco, A.~Melle-Hern\'{a}ndez, and A.~N\'{e}methi, \emph{Links and
  analytic invariants of superisolated singularities}, J. Algebraic Geom.
  \textbf{14} (2005), no.~3, 543--565.

\bibitem[Mat02]{Matsuki_MMP_intro}
Kenji Matsuki, \emph{Introduction to the {M}ori program}, Universitext,
  Springer-Verlag, New York, 2002.

\bibitem[Miy01]{Miyan-OpenSurf}
Masayoshi Miyanishi, \emph{Open algebraic surfaces}, CRM Monograph Series,
  vol.~12, American Mathematical Society, Providence, RI, 2001.

\bibitem[MKM83]{KumarMurthy-rectifiability}
N.~Mohan~Kumar and M.~Pavaman Murthy, \emph{Curves with negative
  self-intersection on rational surfaces}, J. Math. Kyoto Univ. \textbf{22}
  (1982/83), no.~4, 767--777.

\bibitem[MS89]{MaSa-degree_of_cusp_curves}
Takashi Matsuoka and Fumio Sakai, \emph{The degree of rational cuspidal
  curves}, Math. Ann. \textbf{285} (1989), no.~2, 233--247.

\bibitem[MT92]{MiTs-lines_on_qhp}
M.~Miyanishi and S.~Tsunoda, \emph{Absence of the affine lines on the homology
  planes of general type}, J. Math. Kyoto Univ. \textbf{32} (1992), no.~3,
  443--450.

\bibitem[Nam84]{Namba_geometry_of_curves}
M.~Namba, \emph{Geometry of projective algebraic curves}, Monographs and
  textbooks in pure and applied mathematics, M. Dekker, 1984.

\bibitem[Ore02]{OrevkovCurves}
Stephan~Yu. Orevkov, \emph{On rational cuspidal curves}, Math. Ann.
  \textbf{324} (2002), no.~4, 657--673.

\bibitem[Pal11]{Palka-recent_progress_Qhp}
Karol Palka, \emph{Recent progress in the geometry of {$\mathbb{Q}$}-acyclic
  surfaces}, Affine Algebraic Geometry: The Russell Festschrift, CRM Proc.
  Lecture Notes, vol.~54, Amer. Math. Soc., Providence, RI, 2011,
  \arxiv{1003.2395}, pp.~271--287.

\bibitem[Pal14]{Palka-Coolidge_Nagata1}
\bysame, \emph{The {C}oolidge--{N}agata conjecture, part {I}}, Adv. Math.
  \textbf{267} (2014), 1--43, \arxiv{1405.5917}.

\bibitem[Pal15]{Palka-AMS_LZ}
\bysame, \emph{A new proof of the theorems of {L}in-{Z}aidenberg and
  {A}bhyankar-{M}oh-{S}uzuki}, J. Algebra Appl. \textbf{14} (2015), No. 8,
  \arxiv{1405.5391}.

\bibitem[Pal19]{Palka-minimal_models}
\bysame, \emph{Cuspidal curves, minimal models and {Z}aidenberg's finiteness
  conjecture}, J. Reine Angew. Math. \textbf{747} (2019), 147--174,
  \arxiv{1405.5346}.

\bibitem[Pio07]{Piontkowski-number_of_cusps}
Jens Piontkowski, \emph{On the number of cusps of rational cuspidal plane
  curves}, Experiment. Math. \textbf{16} (2007), no.~2, 251--256.

\bibitem[PP17]{PaPe_cuspidal-Cstst-fibrations}
Karol Palka and Tomasz Pe{\l}ka, \emph{Classification of planar rational
  cuspidal curves {I}. {$\Bbb C^{**}$}-fibrations}, Proc. Lond. Math. Soc. (3)
  \textbf{115} (2017), no.~3, 638--692.

\bibitem[PP20]{PaPe_cuspidal-delPezzo}
\bysame, \emph{Classification of planar rational cuspidal curves. {II}. {L}og
  del {P}ezzo models}, Proc. Lond. Math. Soc. (3) \textbf{120} (2020), no.~5,
  642--703, \arxiv{1810.08180}.

\bibitem[Rus80]{Russell_HN_pairs}
Peter Russell, \emph{Hamburger-{N}oether expansions and approximate roots of
  polynomials}, Manuscripta Math. \textbf{31} (1980), no.~1-3, 25--95.

\bibitem[Rus02]{Russell-formal_aspects_M-R}
\bysame, \emph{Some formal aspects of the theorems of {M}umford-{R}amanujam},
  Algebra, arithmetic and geometry, {P}art {I}, {II} ({M}umbai, 2000), Tata
  Inst. Fund. Res. Stud. Math., vol.~16, Tata Inst. Fund. Res., Bombay, 2002,
  pp.~557--584.

\bibitem[Suz74]{Suzuki_AMSthm}
Masakazu Suzuki, \emph{Propri\'et\'es topologiques des polyn\^omes de deux
  variables, complexes, et automorphismes alg\'ebriques de l'espace {${\bf
  C}^{2}$}}, J. Math. Soc. Japan \textbf{26} (1974), 241--257.

\bibitem[Ton00]{Tono-equations_cusp_curves}
Keita Tono, \emph{Defining equations of certain rational cuspidal curves. {I}},
  Manuscripta Mathematica \textbf{103} (2000), no.~1, 47--62.

\bibitem[Ton01]{Tono_1cusp_with_kod_1}
\bysame, \emph{Rational unicuspidal plane curves with $\bar{\kappa}=1$}, Newton
  polyhedrons and Singularities, vol. 1233, RIMS Kokyuroku, 2001, pp.~82--89.

\bibitem[Ton05]{Tono-number_of_cusps}
\bysame, \emph{On the number of the cusps of cuspidal plane curves}, Math.
  Nachr. \textbf{278} (2005), no.~1-2, 216--221.

\bibitem[Ton12]{Tono_nie_bicuspidal}
\bysame, \emph{On a new class of rational cuspidal plane curves with two
  cusps}, \arxiv{1205.1248}, 2012.

\bibitem[Wak78]{Wakabayashi-cusp}
Isao Wakabayashi, \emph{On the logarithmic {K}odaira dimension of the
  complement of a curve in {$P\sp{2}$}}, Proc. Japan Acad. Ser. A Math. Sci.
  \textbf{54} (1978), no.~6, 157--162.

\bibitem[Yos88]{Yos-cusp}
Hisao Yoshihara, \emph{Plane curves whose singular points are cusps}, Proc.
  Amer. Math. Soc. \textbf{103} (1988), no.~3, 737--740.

\bibitem[ZL83]{LinZaid-LZ_theorem}
Mikhail. Zaidenberg and Vladimir~Ya. Lin, \emph{An irreducible, simply
  connected algebraic curve in {$\mathbb{C}^{2}$} is equivalent to a
  quasihomogeneous curve}, Dokl. Akad. Nauk SSSR \textbf{271} (1983), no.~5,
  1048--1052.

\bibitem[ZO96]{OrZa-rigid_cusp}
M.~G. Zaidenberg and S.~Yu. Orevkov, \emph{On rigid rational cuspidal plane
  curves}, Uspekhi Mat. Nauk \textbf{51} (1996), no.~1(307), 149--150.

\end{thebibliography}
\end{document}